\documentclass[12pt]{amsart}

\usepackage{lineno,hyperref}
\usepackage{amsmath,amsfonts,amsthm}
\usepackage{amssymb,latexsym}
\usepackage{cases}
\usepackage[numbers]{natbib}
\usepackage{booktabs}
\usepackage{bm} 
\usepackage{cleveref}
\usepackage{paralist}
\usepackage{multirow}
\usepackage{enumerate}
\usepackage{adjustbox}
 \usepackage[shortlabels]{enumitem}
 \setlist[enumerate]{nosep}
 
\usepackage[section]{placeins}

\renewcommand\appendix{\setcounter{secnumdepth}{-2}}

\usepackage{url,cancel}
\usepackage{float}

\usepackage{caption,enumitem}

\newtheorem{thm}{Theorem}[section]

\newtheorem{lem}[thm]{Lemma}
\newtheorem{prop}[thm]{Proposition}
\newtheorem{cor}[thm]{Corollary}

\theoremstyle{remark}
\newtheorem{re}[thm]{Remark}

\theoremstyle{definition}

\newtheorem{defn}[thm]{Definition}

\newcommand{\gen}{\mbox{gen}}
\newcommand{\ord}{\mathrm{ord}}

\newcommand{\rank}{\mbox{rank}}
\newcommand{\chartic}{\mbox{char}}

\newcommand{\pr}{\mathrm{pr}}

\def\rep{{\to\!\!\! -}}

\def\nrep{{\to\!\!\not\! -}}

\makeatletter

\@addtoreset{equation}{section}
\makeatother
\numberwithin{equation}{section}

\begin{document}

\author{Zilong He}
\address{Department of Mathematics,	Dongguan University of Technology, Dongguan 523808, China}
\email{zilonghe@connect.hku.hk}
	
\title[]{On n-ADC integral quadratic lattices over algebraic number fields}
	\thanks{ }
	\date{\today}
	\begin{abstract}
 		In the paper, we extend the ADC property to the representation of quadratic lattices by quadratic lattices, which we define as $ n $-ADC-ness. We explore the relationship between $ n$-ADC-ness, $ n $-regularity and $ n $-universality for integral quadratic lattices. Also, for $ n\ge 2 $, we give necessary and sufficient conditions for an integral quadratic lattice over arbitrary non-archimedean local fields to be $ n $-ADC. Moreover, we show that over any algebraic number field $ F $, an integral $ \mathcal{O}_{F} $-lattice with rank $ n+1 $ is $n$-ADC if and only if it is $\mathcal{O}_{F}$-maximal of class number one. 
	\end{abstract}
	\maketitle
	
	\section{Introduction}	
	 The problem of representing quadratic forms by quadratic forms was first studied by Mordell. In \cite{mordell_waring_1930}, he proved that the sum of five squares represents all binary quadratic forms. Building on this work, B. M. Kim, M.-H. Kim, and S. Raghavan \cite{kim_2universal_1997} defined a positive definite classic integral quadratic form (i.e., a form with even cross terms) to be $n$-universal if it represents all $n$-ary classic integral quadratic forms. When $n=1$, this notion agrees with the concept of universal quadratic forms, which dates back to Lagrange's four-square theorem and has been extensively studied by mathematicians such as Ramanujan, Dickson and so on. Among the most famous are the 15-theorem by Conway and Schneeberger for classic integral quadratic forms and the 290-theorem by Bhargava and Hanke for quadratic forms with integer coefficients. Similar results have also been shown for $n$-universal quadratic forms in \cite{choh-canwerecover-2022,kim_finiteness_2005}. Another important topic is the study of regular quadratic forms, which represent all integers represented by their genus. This concept was first introduced and systematically studied by Dickson in \cite{dickson_ternary_1927}. Since then, a significant amount of research has been devoted to classifying them in the ternary case (cf. \cite{JKS_regular913_1997, jp_regular_1939,oliver_representation_2014,oh_regular_2011}). Similarly, Earnest \cite{earnest_binary_1994} introduced $ n $-regular quadratic forms and showed that there are only finitely many primitive quaternary $ 2 $-regular quadratic forms up to equivalence, which were partly classified by Oh \cite{oh_even2regular_2008}. For $n\geq 2$, Chan and Oh \cite{choh-finitness-2003} extended Earnest's result to $(n+2)$-ary (resp. $(n+3)$-ary) $n$-regular quadratic forms.
	
	A classical theorem, due to \textbf{A}ubry, \textbf{D}avenport and \textbf{C}assels, states: if $ Q(x) $ is a positive definite classic $ n $-ary quadratic form such that for all $ x\in \mathbb{Q}^{n} $ there exists $ y\in \mathbb{Z}^{n} $ such that $ Q(x-y)<1 $, then $ Q(x) $ satisfies the property: for all $ c\in \mathbb{Z} $, if the equation $ Q(x)=c $ has a solution in $ \mathbb{Q} $, then it has a solution in $ \mathbb{Z}  $. Based on this result,  Clark \cite{clark_ADC-I-2012} introduced the concept of ADC quadratic forms, which satisfy the property ``solvable over rationals implies solvable over integers". In general, he defined ADC and Euclidean quadratic forms over normed ring, and investigated their relationship. In \cite{clark_ADC-II-2014}, Clark and Jagy further determined all ADC forms in non-dyadic local fields and obtained some partial results in $2$-adic local fields. Additionally, they completely enumerated all $n$-ary ADC integral forms for $1\leq n\leq 4$ and all Euclidean integral forms.
	
	 As for universality and regularity, studying higher-dimensional analogues of the ADC property is a natural generalization, which motivates the introduction of $n$-ADC lattices (defined in Definitions \ref{defn:localnADC} and \ref{defn:globalnADC}) in this paper. Also, we find that such notion plays an important role between $n$-universality and $n$-regularity (Theorems \ref{thm:globallADC} and \ref{thm:locallyn-ADCn+3}(iii)). We will investigate $n$-ADC-ness from local fields to global fields. Precisely, we characterize $n$-ADC lattices of rank $\ge n$ over arbitrary non-archimedean local fields for $n\ge 2$ (Theorems \ref{thm:locallyn-ADCn+3}(i)(ii), \ref{thm:locallyn-ADCn+1} and \ref{thm:locallyn-ADC}), and give a counting formula (Theorem \ref{thm:count-sol}). The case $n=1$ requires a different approach than that for $n\ge 3$ odd, as discussed in Section \ref{sec:n-ADCdyadicfields-odd}, and it will be treated in a future paper. Due to the complexity of Jordan splittings, we will use a non-classical but effective theory, developed by Beli \cite{beli_integral_2003, beli_representations_2006,beli_Anew_2010}, to treat higher dimensional representations of quadratic lattices over general dyadic local fields (see \cite{He22,HeHu2} for recent progress). By virtue of these local classifications, we establish the equivalent condition on $n$-ADC lattices of rank $n+1$ over algebraic number fields (Theorem \ref{thm:globallyn-ADC-n+1}). Based on the work of previous researchers \cite{Hanke_maximal,kirschmer_one-class_2014,oh_even2regular_2008}, we determine all positive definite $n$-ADC lattices of rank $n+1$ over totally real number fields (Corollary \ref{cor:globallyn-ADC-n+1}), and partially classify $2$-ADC lattices over $\mathbb{Q}$ (Theorem \ref{thm:Z2-ADCquaternary}).
	
	\medskip
	First of all, we briefly introduce the arithmetic theory of quadratic forms. Any unexplained notations or definitions can be found in \cite{omeara_quadratic_1963}. For short, by local fields, we always mean non-archimedean local fields (cf. \cite[\S32:1 Definition]{omeara_quadratic_1963}).
	 
	\noindent \textbf{General settings.} 
	
	Let $ F $ be an algebraic number field or a local field with $  \chartic\, F\not=2 $, $ \mathcal{O}_{F} $ the ring of integers of $ F $ and $ \mathcal{O}_{F}^{\times}$ the group of units. Let $ V $ be a non-degenerate quadratic space over $ F $ together with the symmetric bilinear form $ B:V\times V\to F $ and set $ Q(x):=B(x,x) $ for all $ x\in V $. We call $ L $ an \textit{$ \mathcal{O}_{F} $-lattice} in $ V $ if it is a finitely generated $ \mathcal{O}_{F} $-submodule of $ V $, and say that $L$ is on $V$ if $V=FL$, i.e., $V$ is spanned by $L$ over $F$. For an $ \mathcal{O}_{F} $-lattice $ L $, we denote by $ \mathfrak{s}(L) $ (resp. $ \mathfrak{n}(L) $, $\mathfrak{v}(L)$) the scale (resp. the norm, the volume) of $ L $ as usual. We call $ L $ \textit{integral} if $ \mathfrak{n}(L)\subseteq \mathcal{O}_{F}$. For a non-zero fractional ideal $ \mathfrak{a} $ in $ F $, we also call $ L $ \textit{$ \mathfrak{a} $-maximal} if $ \mathfrak{n}(L)\subseteq \mathfrak{a} $ and there is no $\mathcal{O}_{F}$-lattice $L^{\prime}$ on $FL$ with $ \mathfrak{n}(L^{\prime})\subseteq\mathfrak{a} $ such that $L\subsetneq L^{\prime}$. We denote by $  \mathcal{L}_{F,n}$ the set of all integral  $ \mathcal{O}_{F} $-lattices of rank $n$ and by $\mathcal{M}_{n}$ the set of all $ \mathcal{O}_{F} $-maximal lattices of rank $ n $. When $F$ is an algebraic number field, we also denote by $ \Omega_{F}$ (resp. $\infty_{F}$) the set of all primes (resp. all archimedean primes) of $ F $.
 
   \medskip
		
  \noindent \textbf{Local settings.} 
  
  When $ F $ is a local field, write $ \mathfrak{p} $ for the maximal ideal of $\mathcal{O}_F$, $ \pi \in\mathfrak{p}$ for a uniformizer and $N\mathfrak{p}$ for the number of elements in the residue class field of $F$. Set $\mathfrak{p}^{0}=\mathcal{O}_{F}$ for convention. For $c\in F^{\times}$, let $c=\varepsilon\pi^{k}$ with $\varepsilon\in \mathcal{O}_{F}^{\times}$ and $k\in \mathbb{Z}$. We denote by $\ord(c)=k$ the \textit{order} of $c$ and formally put $\ord(0)=\infty$. Put $ e:=\ord(2)$. For a fractional or zero ideal $\mathfrak{c}$ of $F$, we put $\ord(\mathfrak{c})=\min\{\ord(c)\,|\,c\in \mathfrak{c}\}$. We fix $\Delta\in \mathcal{O}_{F}^{\times}$ such that $F(\sqrt{\Delta})/F$ is quadratic unramified. If $F$ is non-dyadic, then $\Delta$ is an arbitrary non-square unit; if $F$ is dyadic, then $\Delta$ is a non-square unit of the form $\Delta=1-4\rho$, with $\rho\in\mathcal{O}_{F}^{\times}$.
  
   If $F$ is dyadic, we define the \textit{quadratic defect} of $ c $ by $ \mathfrak{d}(c):=\bigcap_{x\in F}(c-x^{2})\mathcal{O}_{F}$ and the \textit{order of relative quadratic defect} by the map $ d $ from $ F^{\times}/F^{\times 2} $ to $ \mathbb{N}\cup \{\infty\} $: $ d(c):=\ord (c^{-1}\mathfrak{d}(c)) $. Recall some properties of the map $d$:
   
    \begin{enumerate}[itemindent=-0.5em,label=\rm (\roman*)]
   	\item The image of $d$ is $ \{0,1,3,\ldots,2e-1,2e,\infty\}$.
   	\item For $c\in F^{\times}$, $d(c)=0$ if and only if $\ord(c)$ is odd, $d(c)=2e$ if and only if $c\in \Delta F^{\times 2}$, and $d(c)=\infty$ if and only if $c\in F^{\times 2}$.
   	\item The domination principle: $d(ab)\ge \min\{d(a),d(b)\}$ for all $a,b\in F^{\times}$.
   \end{enumerate}

   Also, we denote by  $ \mathcal{U} $ a complete system of representatives of $\mathcal{O}_{F}^{\times}/\mathcal{O}_{F}^{\times 2}$ such that $ d(\delta)=\ord(\delta-1) $ for all $\delta\in \mathcal{U} $, and by $ \mathcal{V}:=\mathcal{U} \cup \pi\mathcal{U}=\{\delta,\pi\delta\mid \delta\in\mathcal{U} \} $ a set of representatives of $ F^{\times}/F^{\times 2} $. If $F$ is non-dyadic, we put $\mathcal{U}=\{1,\Delta\}$ and $\mathcal{V}=\{1,\Delta,\pi,\Delta\pi\}$.

        We write $ V\cong [a_{1},\ldots,a_{n}] $ (resp. $ L\cong \langle a_{1},\ldots,a_{n}\rangle $) if $ V=Fx_{1}\perp \ldots \perp Fx_{n} $ (resp. $ L=\mathcal{O}_{F}x_{1}\perp \ldots\perp \mathcal{O}_{F}x_{n} $) with  $ Q(x_{i})=a_{i} $. For $ \gamma\in F^{\times}$ and $\xi,\mu \in  F$, we  denote by $ \gamma A(\xi,\mu) $ the binary $ \mathcal{O}_{F} $-lattice associated with the Gram matrix $ \gamma\begin{pmatrix}
			 \xi   & 1\\
			1  & \mu
		\end{pmatrix} $. Write $\mathbf{H}=2^{-1}A(0,0)$ and $\mathbf{A}=2^{-1}A(2,2\rho)$. When $F$ is non-dyadic, we have $\mathbf{H}=\langle 1,-1\rangle$ and $\mathbf{A}=\langle 1,-\Delta\rangle$. Let $\mathbb{H}$ denote the usual hyperbolic plane. Clearly, $\mathbb{H}=F\mathbf{H}$. We further denote by  $\mathbf{H}^{k}$ (resp. $\mathbb{H}^{k}$) the orthogonal sum of $k$ copies of $\mathbf{H}$ (resp. $\mathbb{H}$) for any positive integer $k$. 

If instead of a given local field $F$ we talk about the localization $F_{\mathfrak{p}}$ of an algebraic number field $F$ at the finite prime $\mathfrak{p}$, then we will add the subscript $\mathfrak{p}$ to the notations $\pi$, $\ord$, $e$, $\mathfrak{d}$, $d$, $\mathcal{U}$ and $\mathcal{V}$.
          
\medskip

		\begin{defn}\label{defn:localnADC}
		Let $n$ be a positive integer. Let $ M $ be an integral $ \mathcal{O}_{F} $-lattice over a local field $ F $. Then
		
		(i) $M$ is called $ n $-universal if it represents all lattices $N$ in $\mathcal{L}_{F,n}$.
		
		(ii)  $M$ is called $ n $-ADC if it represents every lattice $N$ in $ \mathcal{L}_{F,n} $ for which $FM$  represents $ FN $.
	\end{defn}	 
	 In the rest of this section, we assume that $ F $ is an algebraic number field, $V$ is a quadratic space over $F$, and $ M $ is an integral $ \mathcal{O}_{F} $-lattice on $V$. For $ \mathfrak{p}\in \Omega_{F} $, let $ F_{\mathfrak{p}} $ be the completion of $ F $ at $ \mathfrak{p} $. Then write $ M_{\mathfrak{p}}:=\mathcal{O}_{F_{\mathfrak{p}}}\otimes M $ when $ \mathfrak{p}\in \Omega_{F}\backslash \infty_{F} $, and set $ M_{\mathfrak{p}}:=F_{\mathfrak{p}}\otimes M=V_{\mathfrak{p}}$ for convention when $ \mathfrak{p}\in \infty_{F} $. Thus $M_{\mathfrak{p}}$ is always $ n $-ADC for $ \mathfrak{p}\in \infty_{F} $. Then we say that $ M $ is \textit{locally $ n $-ADC} (resp. \textit{locally $n$-universal}) if $ M_{\mathfrak{p}} $ is $ n $-ADC (resp. $n$-universal) for all $ \mathfrak{p}\in \Omega_{F}\backslash \infty_{F} $.  
	 
	\begin{defn}\label{defn:globalnADC}
		Let $ n $ be a positive integer. Then
	  \begin{enumerate}[itemindent=-0.5em,label=\rm (\roman*)]
		\item $ M $ is called globally $ n $-universal, or simply $n$-universal, if it represents all lattices $N$ in $\mathcal{L}_{F,n}$ with compatible signatures, i.e., with $N_{\mathfrak{p}}\rep M_{\mathfrak{p}}$ at all real primes $\mathfrak{p}\in \infty_{F}$.

		\item $ M $ is called globally $ n $-ADC, or simply $n$-ADC, if it repersents every lattice $ N$ in  $\mathcal{L}_{F,n}$ for which $ FM $ represents $ FN $.
		\end{enumerate} 
	\end{defn}
Recall that an $\mathcal{O}_{F}$-lattice $M$ (that may not be integral) is called \textit{$n$-regular} if it represents every lattice $N$ in $\mathcal{L}_{F,n}$ for which $M_{\mathfrak{p}}$ represents $N_{\mathfrak{p}}$ for each $\mathfrak{p}\in\Omega_{F}$. The $n$-ADC property can be viewed as a transition between $n$-universality and $n$-regularity.  More specifically, an $\mathcal{O}_{F}$-lattice that is $n$-universal must be $n$-ADC from definition, and an $n$-ADC $\mathcal{O}_{F}$-lattice is $n$-regular. In fact, we have the following equivalent condition for $n$-ADC $\mathcal{O}_{F}$-lattices, which is a generalization of \cite[Theorem 25]{clark_ADC-I-2012} with $R=\mathcal{O}_{F}$.
\begin{thm}\label{thm:globallADC}
Let $ n $ be a positive integer. Then $M$ is globally $ n $-ADC if and only if it is locally $n$-ADC and $ n $-regular.
\end{thm}
\begin{thm}\label{thm:locallyn-ADCn+3}
	Suppose $ \rank\,M\ge n+3\ge 4$. Let $ \mathfrak{p}\in \Omega_{F}\backslash \infty_{F} $. Then
	 \begin{enumerate}[itemindent=-0.5em,label=\rm (\roman*)]
		\item  $M_{\mathfrak{p}} $ is $ n $-ADC if and only if it is $n$-universal. 
		\item  $ M$ is locally $ n $-ADC if and only if it is locally $n$-universal.
		
		\item $ M $ is globally $ n $-ADC if and only if it is globally $n$-universal.
	\end{enumerate} 
	\end{thm}
For $n\geq 1$, all $n$-universal lattices over non-dyadic/dyadic local fields have been completely determined in \cite{beli_universal_2020,HeHu2,hhx_indefinite_2021,xu_indefinite_2020}. Hence, from Theorems \ref{thm:globallADC} and \ref{thm:locallyn-ADCn+3}(i), determining the $n$-regularity for a given $\mathcal{O}_{F}$-lattice $M$ is crucial for its $n$-ADC-ness when $\rank\,M\ge  n+3$. Although it was shown in \cite[Theorem 1.1(1)]{hhx_indefinite_2021} that local-global principle holds for indefinite $n$-universality\footnote{In the indefinite case, the notion of  $n$-universal defined in this paper does not coincide with that of indefinite $n$-universal introduced in \cite[Definition 1.4(3)]{hhx_indefinite_2021} for $n\ge 2$.} with $n\ge 3$, it is difficult to verify $n$-regularity of a quadratic lattice in general for definite cases.
\begin{thm}\label{thm:locallyn-ADCn+1}
	Suppose $ \rank\,M=n\ge 2$ or $\rank\, M=n+1\ge 3$. Let $ \mathfrak{p}\in \Omega_{F}\backslash \infty_{F} $. Then 
 	  \begin{enumerate}[itemindent=-0.5em,label=\rm (\roman*)]
 	\item $ M_{\mathfrak{p}} $ is $ n $-ADC if and only if it is $ \mathcal{O}_{F_{\mathfrak{p}}} $-maximal. 
 	\item  $ M $ is locally $ n $-ADC if and only if it is $ \mathcal{O}_{F} $-maximal.
 \end{enumerate} 
\end{thm}
\begin{re}
 (i) All $ \mathcal{O}_{F_{\mathfrak{p}}} $-maximal lattices have been explicitly listed in \cite{HeHu2,hhx_indefinite_2021}. See Lemmas \ref{lem:maximallattices-non-dyadic}(i) and \ref{lem:maximallattices-dyadic}(i) (or \cite[Proposition 3.7]{HeHu2} described in terms of minimal norm splittings).

 (ii) Theorem 19 of \cite{clark_ADC-I-2012} states that $M_{\mathfrak{p}}$ is $\mathcal{O}_{F_{\mathfrak{p}}}$-maximal if and only if it is Euclidean with respect to the canonical norm (see \cite[\S 4.2]{clark_ADC-I-2012} for Euclidean property over local fields). Therefore, in Theorems \ref{thm:locallyn-ADCn+1} and \ref{thm:locallyn-ADC}, the term ``$\mathcal{O}_{F_{\mathfrak{p}}}$-maximal" can be smoothly replaced with ``Euclidean".
\end{re}

For $n\ge 2$, the class number of an $n$-regular $\mathcal{O}_F$-lattice $M$ may not be equal to one in general, but it is exactly one when the rank is $n+1$, as proved by Kitaoka in \cite[Corollary 6.4.1]{kitaoka_arithmetic_1993} for $F=\mathbb{Q}$ and $M$ is positive definite, which was extended by Meyer \cite[Corollary 5.3]{meyer_determination_2014} to the case when $F$ is  totally real and $M$ is definite. This is also true for indefinite cases (Corollary \ref{cor:n-regular-clsnumberone}). Based on these and Theorem \ref{thm:locallyn-ADCn+1}, we provide more explicit equivalent conditions on $n$-ADC $\mathcal{O}_{F}$-lattices with rank $n+1$.
\begin{thm}\label{thm:globallyn-ADC-n+1}
	If $ \rank\, M=n+1\ge 3 $, then $ M $ is $ n $-ADC if and only if it is $ \mathcal{O}_{F} $-maximal of class number one.
\end{thm}
When $F$ is totally real, all positive definite $ \mathcal{O}_{F} $-maximal lattices with $ \rank\,\ge 3 $ of class number one were enumerated by Hanke \cite{Hanke_maximal} for $ F=\mathbb{Q} $ (115 in total) and by Kirschmer \cite{kirschmer_one-class_2014} for $ F\not=\mathbb{Q} $ (471 in total), respectively. Thus, from Theorem \ref{thm:globallyn-ADC-n+1},  we have the following finiteness result.
\begin{cor}\label{cor:globallyn-ADC-n+1}
	Up to isometry, there are $  586 $ positive definite $ n $-ADC integral $ \mathcal{O}_{F} $-lattices of rank $ n+1\ge 3 $ in total, when $ F $ varies through all totally real number fields. 
\end{cor}
 \begin{thm}\label{thm:locallyn-ADC}
	Suppose $ \rank\,M=n+2\ge 4 $. Let $ \mathfrak{p}\in \Omega_{F}\backslash \infty_{F} $.
	 	  \begin{enumerate}[itemindent=-0.5em,label=\rm (\roman*)]
		\item If $ \mathfrak{p}$ is non-dyadic, then $ M_{\mathfrak{p}} $ is $ n $-ADC if and only if it is $ \mathcal{O}_{F_{\mathfrak{p}}} $-maximal.
		
		\item If $ \mathfrak{p}$ is dyadic and $ n $ is even, then $ M_{\mathfrak{p}} $ is $ n $-ADC if and only if it is either $ \mathcal{O}_{F_{\mathfrak{p}}} $-maximal or isometric to the non $\mathcal{O}_{F_{\mathfrak{p}}}$-maximal lattice 
		\begin{align*}
			\mathbf{H}\perp 2^{-1}\pi_{\mathfrak{p}} A(2\pi_{\mathfrak{p}}^{-1},2\rho_{\mathfrak{p}}\pi_{\mathfrak{p}}).
		\end{align*}
		\item If $ \mathfrak{p}$ is dyadic and $ n $ is odd, then $ M_{\mathfrak{p}} $ is $ n $-ADC if and only if it is either $ \mathcal{O}_{F_{\mathfrak{p}}} $-maximal or isometric to 
		\begin{align*}
			&\mathbf{H}^{\frac{n-1}{2}}\perp  \pi_{\mathfrak{p}}^{ -l_{\mathfrak{p}}}A(\pi_{\mathfrak{p}}^{ l_{\mathfrak{p}}},-(\delta_{\mathfrak{p}}-1)\pi_{\mathfrak{p}}^{ -l_{\mathfrak{p}}}) \perp \langle \varepsilon_{\mathfrak{p}}\pi_{\mathfrak{p}}^{k_{\mathfrak{p}}}\rangle\quad\quad\text{or}\\
			&\mathbf{H}^{\frac{n-1}{2}}\perp  \delta_{\mathfrak{p}}^{\#} \pi_{\mathfrak{p}}^{ -l_{\mathfrak{p}}}A(\pi_{\mathfrak{p}}^{ l_{\mathfrak{p}}},-(\delta_{\mathfrak{p}}-1)\pi_{\mathfrak{p}}^{ -l_{\mathfrak{p}}}) \perp \langle \varepsilon_{\mathfrak{p}}\pi_{\mathfrak{p}}^{k_{\mathfrak{p}}}\rangle,
		\end{align*}
		with $\delta_{\mathfrak{p}}\in  \mathcal{U}_{\mathfrak{p}}\backslash\{1,\Delta_{\mathfrak{p}}\}$,  $2l_{\mathfrak{p}}=d_{\mathfrak{p}}(\delta_{\mathfrak{p}})-1\le 2e_{\mathfrak{p}}-2$, $\varepsilon_{\mathfrak{p}}\in \mathcal{U}_{\mathfrak{p}}$ and $k_{\mathfrak{p}}\in \{0,1\}$, where $\delta_{\mathfrak{p}}^{\#}=1+4\rho_{\mathfrak{p}}(\delta_{\mathfrak{p}}-1)^{-1}$.
		
		Moreover, if $M_{\mathfrak{p}}$ is simultaneously $\mathcal{O}_{F_{\mathfrak{p}}}$-maximal and  has the described orthogonal splitting, then it is isometric to
		\begin{align*}
			 \mathbf{H}^{\frac{n-1}{2}}\perp 2^{-1}\pi_{\mathfrak{p}}A(2,2\rho_{\mathfrak{p}})\perp \langle \Delta_{\mathfrak{p}}\varepsilon_{\mathfrak{p}} \rangle,
		\end{align*}
		with $\varepsilon_{\mathfrak{p}}\in \mathcal{U}_{\mathfrak{p}}$.
	 \end{enumerate} 
\end{thm}
If $m\ge n+3$, then from Theorem \ref{thm:locallyn-ADCn+3}, the notions of $n$-ADC-ness and $n$-universality coincide. Because the $n$-universality was treated in \cite{HeHu2}, in this paper we deal with the remaining cases, with $n\le m\le n+2$. In these cases the number of $n$-ADC lattices is finite and it can be calculated as follows.
\begin{thm}\label{thm:count-sol}
	Let $n\ge 2$. Denote by $B (m,n)$ the number of $n$-ADC $\mathcal{O}_{F}$-lattices with rank $m\in \{n,n+1,n+2\}$ over a local field $F$.  Then $B(m,n)$ is given by
		\begin{align*}
		\begin{cases}
			8(N\mathfrak{p})^{e}-1+0     &\text{if $m=n=2$}, \\
				8(N\mathfrak{p})^{e}\hskip 0.75cm+1  &\text{if $m=n+2=4$ and $e \ge 1$}, \\
				8(N\mathfrak{p})^{e}\hskip 0.75cm+(8e-2)(N\mathfrak{p})^{e}   &\text{if  $m=n+2\ge 5$ with odd $n$ and $e \ge 1$},    \\
				8(N\mathfrak{p})^{e}\hskip 0.75cm+0   &\text{otherwise},
			\end{cases}
		\end{align*}
		where the second addend conuts the number of those lattices that are $n$-ADC, but not $\mathcal{O}_{F}$-maximal.
\end{thm}
 
From Theorem \ref{thm:globallADC}, one can determine whether an $ n $-regular $ \mathcal{O}_{F} $-lattice with rank $ n+2 $ is $ n $-ADC by virtue of Theorem \ref{thm:locallyn-ADC}. In particular, we classify the case $ \mathcal{O}_{F}=\mathbb{Z} $ and $ n=2 $ based on Oh's classification for stable $ 2 $-regular quaternary $ \mathbb{Z} $-lattices \cite{oh_even2regular_2008}.
\begin{thm}\label{thm:Z2-ADCquaternary}
	There are exactly $21$ quaternary positive definite 2-ADC $\mathbb{Z}$-lattices up to isometry, which are enumerated in Table \ref{table:2ADCquaternary}. Each 2-ADC $\mathbb{Z}$-lattice $L_i$ in the table is obtained by scaling some lattice $\mathcal{L}_j$ in Table \ref{table:stable2regularquaternary} by $1/2$. 
	
	Moreover, all of the lattices have class number one, and all except for $L_{10}$ are $\mathbb{Z}$-maximal. 
\end{thm}
\begin{re}
	All of the ternary $2$-ADC lattices have been determined by Theorem \ref{thm:globallyn-ADC-n+1}. For the quinary case, Theorem \ref{thm:locallyn-ADCn+3}(iii) indicates that the $2$-ADC property is equivalent to $2$-universality. However, currently it is only known from \cite[Theorem 2.4]{ji_even2universal_2021} that there are at most $55$ quinary $2$-universal $\mathbb{Z}$-lattices $M$ with $2\mathfrak{s}(M)=\mathbb{Z} $, of which the $2$-universality has not been completely confirmed yet.
\end{re}
	The rest of the paper is organized as follows. We first prove Theorems \ref{thm:globallADC} and \ref{thm:locallyn-ADCn+3} in Section \ref{sec:thm13}. Then, we review Beli's BONGs theory of quadratic forms in Section \ref{sec:BONGs}. In Section \ref{sec:maximal-lattices}, we study some basic notions including quadratic spaces and maximal lattices, and the related results in local fields. In Sections \ref{sec:n-ADCnon-dyadicfields}, \ref{sec:n-ADCdyadicfields-even} and \ref{sec:n-ADCdyadicfields-odd}, we establish equivalent conditions on $n$-ADC lattices in non-dyadic local fields, and in dyadic local fields for even and odd $n$, respectively. In the last section, we will prove our main results, including Theorems \ref{thm:locallyn-ADCn+1}, \ref{thm:globallyn-ADC-n+1}, \ref{thm:locallyn-ADC}, \ref{thm:count-sol} and \ref{thm:Z2-ADCquaternary}.
	
    Here and subsequently, all lattices under consideration are assumed to be integral.

\section{Proof of Theorems \ref{thm:globallADC} and \ref{thm:locallyn-ADCn+3}} \label{sec:thm13}
	
	To show Theorems \ref{thm:globallADC} and \ref{thm:locallyn-ADCn+3}, we need some lemmas.
\begin{lem}\label{lem:maximalequivADC}
	 Suppose that $F$ is an algebraic number field or a local field. Let $M$ be an $\mathcal{O}_{F}$-lattice over $F$. Then $ M $ is $ n $-ADC if and only if $M$ represents every lattice $N$ in $ \mathcal{M}_{n} $ for which $ FM $ represents $ FN $.
\end{lem}
\begin{proof}
		Necessity is trivial. Suppose that $ FM $ represents $ FN $. By \cite[82:18]{omeara_quadratic_1963}, there exists some lattice $ N^{\prime} $ inside $ \mathcal{M}_{n} $ on $ FN $ such that $ N\subseteq N^{\prime} $. Since $ FM $ represents $ FN\cong FN^{\prime} $, by the $ n$-ADC-ness, $ M $ represents $ N^{\prime} $, and therefore represents $ N $.
\end{proof}
\begin{lem}\label{lem:spaceappro}
		Suppose that $F$ is an algebraic number field. Let $ V $ be a quadratic space over  $ F $ and $ \mathfrak{p}\in \Omega_{F}\backslash \infty_{F} $. Given a subspace $ U(\mathfrak{p})\subseteq V_{\mathfrak{p}}$, there exists a subspace $ U\subseteq V$ such that $ U_{\mathfrak{p}}\cong U(\mathfrak{p}) $.
\end{lem}
\begin{proof}
		We prove the statement by induction on $ \dim U(\mathfrak{p}) $. When $ \dim U(\mathfrak{p})=1 $, then $ U(\mathfrak{p})=F_{\mathfrak{p}}u(\mathfrak{p}) $ for some $ u(\mathfrak{p})\in V_{\mathfrak{p}} $. Recall from  \cite[63:1b Corollary, 21:1]{omeara_quadratic_1963} that $ F_{\mathfrak{p}}^{\times 2} $ is open in $ F_{\mathfrak{p}} $ and $ V $ is dense in $ V_{\mathfrak{p}} $. Then there exists $ u \in V  $ such that $ Q(u)\in Q(u(\mathfrak{p}))F_{\mathfrak{p}}^{\times 2} $. Thus, $ Q(u)=c^{2}Q(u(\mathfrak{p}))=Q(cu(\mathfrak{p}))$ for some $ c\in F_{\mathfrak{p}}^{\times} $. Take $ U:=Fu $. Then $ U\subseteq V $ and $ F_{\mathfrak{p}}U=F_{\mathfrak{p}}u=F_{\mathfrak{p}}(cu(\mathfrak{p}))\cong F_{\mathfrak{p}}u(\mathfrak{p})=U(\mathfrak{p}) $.
		
		For $ \dim U(\mathfrak{p})>1 $, we may let $ U(\mathfrak{p})=W(\mathfrak{p})\perp F_{\mathfrak{p}}u(\mathfrak{p}) $. Then, by inductive assumption, there exists $ W\subseteq V $ such that $  W_{\mathfrak{p}}\cong W(\mathfrak{p})  $. Since $ U(\mathfrak{p}) $ is non-degenerate,  and so is $ W(\mathfrak{p}) $. Thus $W$ is also non-degenerate. It follows that $ V=W\perp W^{\perp} $, where $ W^{\perp}:=\{v\in V\mid B(v,W)=0\} $. This yields $ V_{\mathfrak{p}}=W_{\mathfrak{p}} \perp F_{\mathfrak{p}}W^{\perp}$. We also have $ V_{\mathfrak{p}}=W(\mathfrak{p})\perp W(\mathfrak{p})^{\perp} $. By Witt's cancellation theorem, $ F_{\mathfrak{p}}W^{\perp}\cong W(\mathfrak{p})^{\perp} $. Thus one can find $ u^{\prime}\in W_{\mathfrak{p}}^{\perp}\cong W(\mathfrak{p})^{\perp}\cong F_\mathfrak{p}W^{\perp} $ such that $ Q(u^{\prime})=Q(u(\mathfrak{p})) $. By the one-dimensional case of the lemma, there exists $ u\in W^{\perp} $ such that $ F_{\mathfrak{p}}u\cong F_{\mathfrak{p}}u^{\prime}\cong F_{\mathfrak{p}}u(\mathfrak{p}) $. Now take $ U=W\perp Fu $, as desired.
\end{proof}
\begin{proof}[Proof of Theorem \ref{thm:globallADC}]
		For sufficiency, suppose that $ FM $ represents $ FN $ for some $ N \in \mathcal{L}_{F,n} $. By \cite[66:3 Theorem]{omeara_quadratic_1963}, $  FM_{\mathfrak{p}} $ represents $ FN_{\mathfrak{p}} $ for all $\mathfrak{p}\in \Omega_{F}$, so $ M_{\mathfrak{p}} $  represents $ N_{\mathfrak{p}} $ by the $ n $-ADC-ness of $ M_{\mathfrak{p}} $. Hence $ M $ represents $ N $ by the $ n $-regularity of $ M $.	
		
		For necessity, we will first prove that $M $ is locally $n$-ADC, i.e., $M_{\mathfrak{p}}$ is $n$-ADC for each $ \mathfrak{p}\in \Omega_{F}\backslash \infty_{F} $. By Lemma \ref{lem:maximalequivADC}, it is sufficient to show that $M_{\mathfrak{p}} $ represents every $\mathcal{O}_{F_{\mathfrak{p}}}$-maximal lattice $ N(\mathfrak{p}) $ for which $FM_{\mathfrak{p}}$ represents $FN(\mathfrak{p}) $.	
		
		We may assume $ FN(\mathfrak{p})\subseteq FM_{\mathfrak{p}} $. By Lemma \ref{lem:spaceappro}, there exists a subspace $ U\subseteq FM $ such that $ U_{\mathfrak{p}}\cong FN(\mathfrak{p}) $. Hence, by \cite[82:18]{omeara_quadratic_1963}, $ FM $ represents $ FL $ for some $ \mathcal{O}_{F} $-maximal lattice $L $ on $ U $. So  $ M $ represents $ L $ from the $ n$-ADC-ness of $M$. Thus $M_{\mathfrak{p}}$ represents $L_{\mathfrak{p}}$. Note from \cite[\S 82K]{omeara_quadratic_1963} that $L_{\mathfrak{p}}$ is $\mathcal{O}_{F_{\mathfrak{p}}}$-maximal, so $ L_{\mathfrak{p}}\cong N(\mathfrak{p}) $ by \cite[91:2 Theorem]{omeara_quadratic_1963}. Hence $ M_{\mathfrak{p}} $ represents $  N(\mathfrak{p}) $, as desired.
		
		To show the $ n $-regularity of $M $, let $N\in \mathcal{L}_{F,n}$. Suppose that $ M_{\mathfrak{p}} $ represents $ N_{\mathfrak{p}} $ for all $\mathfrak{p}\in \Omega_{F} $. Then $FM_{\mathfrak{p}}$ represents $FN_{\mathfrak{p}}$ for all $\mathfrak{p}\in \Omega_{F}$. Hence, by \cite[66:3 Theorem]{omeara_quadratic_1963}, $ FM $ represents $ FN $. So $ M $ represents $ N $ from the $ n $-ADC-ness of $M$. 
\end{proof}
\begin{proof}[Proof of Theorem \ref{thm:locallyn-ADCn+3}]
		(i) Let $\mathfrak{p}\in \Omega_{F}\backslash\infty_{F}$. Since $\dim FM_{\mathfrak{p}}=\rank\,M_{\mathfrak{p}}\ge n+3$, by \cite[Theorem 2.3(1)]{hhx_indefinite_2021}, $FM_{\mathfrak{p}}$ represents all $n$-dimensional quadratic spaces. Hence for every lattice $N_{\mathfrak{p}}$ in $\mathcal{L}_{F_{\mathfrak{p}},n}$, $FM_{\mathfrak{p}}$ represents $FN_{\mathfrak{p}}$, so $M_{\mathfrak{p}}$ represents $N_{\mathfrak{p}}$ by the $n$-ADC-ness of $M_{\mathfrak{p}}$, i.e., $M_{\mathfrak{p}}$ is $n$-universal, as desired.
 
		(ii) This is clear from the definition and (i). 
		
		(iii) Note that $M$ is locally $n$-ADC (resp. locally $n$-universal) if and only if it is globally $n$-ADC (resp. globally $n$-universal) when $M$ is $n$-regular. Then we are done by Theorem \ref{thm:globallADC} and (ii).
\end{proof}
	
\section{Lattices in terms of BONGs}\label{sec:BONGs}
	
In this section, following Beli's work \cite{beli_thesis_2001,beli_integral_2003,beli_representations_2006,beli_Anew_2010,beli_representations_2019,beli_universal_2020}, we use bases of norm generators (abbr. BONGs) to describe the lattices in arbitrary dyadic local fields instead of Jordan splittings. Let us first review his BONGs theory and recent development \cite{He22,HeHu2}. 
	
Unless otherwise stated, we always assume $ F $ to be dyadic, i.e., $ e\ge 1 $. We write $ [h,k]^{E} $ (resp. $ [h,k]^{O} $) for the set of all even (resp. odd) integers $ i $ such that $ h\le i\le k$. For $ c_{i}\in F^{\times} $, we also write $ c_{i,j}=c_{i}\cdots c_{j} $ for short and put $ c_{i,i-1}=1 $. 
	
The vectors $x_{1},\ldots, x_{m}$ of $FM$ is called a \textit{BONG} for $M$ if $\mathfrak{n}(M)=Q(x_{1})\mathcal{O}_{F}$ and $x_{2},\ldots,x_{m}$ is a BONG for $\pr_{x_{1}^{\perp}}M$, and it is said to be \textit{good} if $\ord(Q(x_{i}))\le \ord(Q(x_{i+2}))$ for $1\le i\le m-2$. We denote by $ M\cong\prec a_{1},\ldots,a_{m}\succ $ if $ x_{1},\ldots, x_{m} $ forms a BONG for $M$ with $ Q(x_{i})=a_{i} $.
\begin{lem}\cite[Lemma 2.2]{HeHu2}\label{lem:goodBONGequivcon}
		Let $ x_{1},\ldots, x_{m}  $ be pairwise orthogonal vectors in $ V $ with $ Q(x_{i})=a_{i} $ and $ R_{i}=\ord (a_{i}) $.  Then $ x_{1},\ldots,x_{m} $ forms a good BONG for some lattice is equivalent to the conditions
		\begin{align}\label{eq:GoodBONGs}
			R_{i}\le R_{i+2} \quad \text{for all $ 1\le i\le m-2$}
		\end{align}
		and  
		\begin{align}\label{eq:BONGs}
			R_{i+1}-R_{i}+d(-a_{i}a_{i+1})\ge 0 \quad\text{and}\quad  R_{i+1}-R_{i}\ge -2e \quad \text{for all $ 1\le i\le m-1 $}. 
		\end{align}
\end{lem}
\begin{cor}\cite[Corollary 2.3]{HeHu2}\label{cor:R-R-odd}
		Suppose $1\le i\le m-1$.
		\begin{enumerate}[itemindent=-0.5em,label=\rm (\roman*)]
			\item If $ R_{i+1}-R_{i} $ is odd, then $ R_{i+1}-R_{i} $ must be positive. 
			
			\item If $ R_{i+1}-R_{i}=-2e $, then $ d(-a_{i}a_{i+1})\ge 2e $ and $ \prec a_{i},a_{i+1}\succ \cong 2^{-1}\pi^{R_{i}}A(0,0)$ or $ 2^{-1}\pi^{R_{i}}A(2,2\rho) $. Consequently, $[a_{i},a_{i+1}]\cong \mathbb{H}$ or $\pi^{R_{i}}[1,-\Delta]$.
		\end{enumerate}
\end{cor}
	
Let $M\cong \prec a_{1},\ldots, a_{m} \succ $ be an $ \mathcal{O}_{F} $-lattice  relative to some good BONG. Define the \textit{$ R_{i} $-invariants} $R_{i}(M):=\ord (a_{i})$ for $ 1\le i\le m $ and the \textit{$ \alpha_{i} $-invariants}
	\[ 
	\begin{split}
		\alpha_{i}(M):=\min( \{(R_{i+1}-R_{i})/2+e\}&\cup \{R_{i+1}-R_{j}+d(-a_{j}a_{j+1})\mid 1\le j\le i\}\\
		&\cup \{R_{j+1}-R_{i}+d(-a_{j}a_{j+1})\mid i\le j\le m-1\}  )
	\end{split}
	\]
for $ 1\le i\le m-1 $. Both are independent of the choice of the good BONG (cf. \cite[Lemma 4.7]{beli_integral_2003}, \cite[\S 2]{beli_Anew_2010}).
	
We give some useful properties for $ R_{i} $ and $ \alpha_{i} $ without proof (cf. \cite{HeHu2}  or \cite{beli_universal_2020}).
\begin{prop}\label{prop:Rproperty}
		Suppose $ 1\le i\le m-1 $. 
			\begin{enumerate}[itemindent=-0.5em,label=\rm (\roman*)]
			\item $ R_{i+1}-R_{i}>2e$ (resp. $ =2e $, $ <2e $) if and only if $ \alpha_{i}>2e $ (resp. $ =2e $, $ <2e $).
			
			\item If $ R_{i+1}-R_{i}\ge 2e$ or $ R_{i+1}-R_{i}\in \{-2e,2-2e,2e-2\} $, then $ \alpha_{i}=(R_{i+1}-R_{i})/2+e $.
			
			\item If $ R_{i+1}-R_{i}\le 2e $, then $ \alpha_{i}\ge R_{i+1}-R_{i} $. Also, the equality holds if and only if $ R_{i+1}-R_{i}=2e $ or $ R_{i+1}-R_{i} $ is odd.
		\end{enumerate} 
\end{prop}
\begin{prop}\label{prop:alphaproperty}
		Suppose $ 1\le i\le m-1 $.
			\begin{enumerate}[itemindent=-0.5em,label=\rm (\roman*)]
			\item Either $ 0\le \alpha_{i}\le 2e $ and $ \alpha_{i}\in \mathbb{Z} $, or $ 2e<\alpha_{i}<\infty $ and $ 2\alpha_{i}\in \mathbb{Z} $; thus $ \alpha_{i}\ge 0 $.
			
			\item $ \alpha_{i}=0 $ if and only if $ R_{i+1}-R_{i}=-2e $.
			
			\item $ \alpha_{i}=1 $ if and only if either $ R_{i+1}-R_{i}\in \{2-2e,1\} $, or $ R_{i+1}-R_{i}\in [4-2e,0]^{E} $ and $ d[-a_{i,i+1}]=R_{i}-R_{i+1}+1 $.
			
			\item If $ \alpha_{i}=0 $, i.e., $ R_{i+1}-R_{i}=-2e $, then $ d[-a_{i,i+1}]\ge 2e $.
	 
			\item If $ \alpha_{i}=1 $, then $ d[-a_{i,i+1}]\ge R_{i}-R_{i+1}+1 $. Also, the equality holds if $ R_{i+1}-R_{i}\not=2-2e $.
		\end{enumerate}  
\end{prop}
\begin{prop}\label{prop:Ralphaproperty3}
		Suppose that $ M $ is integral.
		\begin{enumerate}[itemindent=-0.5em,label=\rm (\roman*)]
			\item We have $ R_{j}\ge R_{i}\ge 0 $ for all odd integers $ i,j $ with $ j\ge i $ and $ R_{j}\ge R_{i}\ge -2e $ for all even integers $ i,j $ with $ j\ge i $.
			
			\item  If $ R_{j}=0 $ for some $ j\in [1,m]^{O} $, then $ R_{i}=0 $ for all  $ i\in [1,j]^{O} $ and $ R_{i} $ is even for all $ 1\le i\le j $.
			
			\item  If $ R_{j}=-2e$ for some $ j\in [1,m]^{E} $, then for each $ i\in [1,j]^{E} $, we have $ R_{i-1}=0 $, $ R_{i}=-2e $ and $ d(-a_{i-1}a_{i})\ge d[-a_{i-1,i}]\ge 2e $. Consequently, $ d[(-1)^{j/2}a_{1,j}]\ge 2e $.
			
			\item  If $ R_{j}=-2e$ for some $ j\in [1,m]^{E} $, then $ [a_{1},\ldots,a_{j}]\cong \mathbb{H}^{j/2} $ or $ \mathbb{H}^{(j-2)/2} \perp [1,-\Delta]$, according as $ d((-1)^{j/2}a_{1,j})=\infty  $ or $  2e $.
			
			\item If $ R_{j}=-2e$ and $ R_{j+1} $ is even for some $ j\in [1,m]^{E} $, then $ [a_{1},\ldots,a_{j+1}]\cong \mathbb{H}^{j/2}\bot [\varepsilon]$ for some $ \varepsilon\in \mathcal{O}_{F}^{\times} $ with $ \varepsilon \in a_{j+1}F^{\times 2}\cup  \Delta a_{j+1} F^{\times 2} $.
		\end{enumerate}  		 
\end{prop}
	
Let $N\cong \prec b_{1},\cdots, b_{n}\succ  $ be another $ \mathcal{O}_{F} $-lattice  relative to some good BONG, $ S_{i}=R_{i}(N) $ and $ \beta_{i}=\alpha_{i}(N) $. For $ 0\le i\le m $ and $0\le j\le n$, we define
	\begin{align*}
		d[ca_{1,i}b_{1,j}]=\min\{d(ca_{1,i}b_{1,j}),\alpha_{i},\beta_{j}\}\,,\quad c\in F^{\times},
	\end{align*}
	where $ \alpha_{i} $ is ignored if $ i\in \{0,m\} $ and $ \beta_{j} $ is ignored if $ j\in \{0,n\} $. In particular, if $M=N$ and $0\le i-1\le j\le m$, then we define
	\begin{align*}
		d[ca_{i,j}]:=d[ca_{1,i-1}a_{1,j}]=\min\{d(ca_{i,j}),\alpha_{i-1},\alpha_{j}\}.
	\end{align*}
	Here we ignore $\alpha_{i-1}$ if $i\in \{1,m+1\}$ and we ignore $\alpha_{j}$ if $j\in \{0,m\}$.
	 Recall that the invariants $d[ca_{1,i}b_{1,j}]$ satisfy the same domination principles as their $d(ca_{1,i}b_{1,j})$ correspondents. (See \cite[\S 1.4]{beli_universal_2020}.) With this notation, the invariant $ \alpha_{i} $ can be neatly expressed as  
	\begin{align}\label{eq:alpha-defn}
		\alpha_{i}=\min\{(R_{i+1}-R_{i})/2+e,R_{i+1}-R_{i}+d[-a_{i,i+1}]\}
	\end{align}
	(cf. \cite[Corollary\;2.5(i)]{beli_Anew_2010}). For any $ 1\le i\le \min\{m-1,n\} $, we define
	\[
	\begin{split}
		A_{i}=A_{i}(M,N):=\min\{&(R_{i+1}-S_{i})/2+e,R_{i+1}-S_{i}+d[-a_{1,i+1}b_{1,i-1}],\\
		&R_{i+1}+R_{i+2}-S_{i-1}-S_{i}+d[a_{1,i+2}b_{1,i-2}]\},
	\end{split}
	\]where the term $ R_{i+1}+R_{i+2}-S_{i-1}-S_{i}+d[a_{1,i+2}b_{1,i-2}] $ is ignored if $ i=1 $ or $ m-1 $.
	
	\medskip
	
	Our main tool is the representation theorem due to Beli \cite[Theorem 1.2 and Remarks 1]{beli_universal_2020} (see \cite[Theorem 4.5]{beli_Anew_2010} and \cite{beli_representations_2019} for more details).
	\begin{thm}\label{thm:beligeneral}
		Suppose $  n\le m$. Then $ N\rep M $ if and only if  $ FN\rep FM $ and the following conditions hold:
		 \begin{enumerate}[itemindent=-0.5em,label=\rm (\roman*)]
			\item For any $ 1\le i\le n $, we have either $ R_{i}\le S_{i} $, or $ 1<i<m $ and $ R_{i}+R_{i+1}\le S_{i-1}+S_{i} $.
			
			\item For any $ 1\le i\le \min\{m-1,n\} $, we have $ d[a_{1,i}b_{1,i}]\ge A_{i} $.
			
			\item For any $ 1<i\le \min\{m-1,n+1\} $, if
			\begin{align*} 
				R_{i+1}>S_{i-1} \quad
				\text{and}\quad d[-a_{1,i}b_{1,i-2}]+d[-a_{1,i+1}b_{1,i-1}]>2e+S_{i-1}-R_{i+1},
			\end{align*}
			then $ [b_{1},\ldots,b_{i-1}]\rep [a_{1},\ldots,a_{i}] $.
			
			\item  For any $ 1<i\le \min\{m-2,n+1\} $ such that $ S_{i}\ge R_{i+2}>S_{i-1}+2e\ge R_{i+1}+2e$, we have $ [b_{1},\ldots,b_{i-1}]\rep [a_{1},\ldots,a_{i+1}] $. (If $ i=n+1 $, the condition $ S_{i}\ge R_{i+2} $ is ignored.)
		\end{enumerate}  
	\end{thm}
	
\section{Preliminaries over local fields}\label{sec:maximal-lattices}
	
Unless otherwise stated, we always assume that $F$ is a local field and $n$ is a positive integer in this section. Clearly, $e=0$ if $F$ is non-dyadic and $e\ge 1$ if $F$ is dyadic.
	 
First, we extend \cite[Definitions 3.4 and 3.6, and Proposition 3.5]{HeHu2} to the non-dyadic case, including $n=1$.
	\begin{defn}\label{defn:space-maximallattice}
		Let $n\ge 1$. For $c\in \mathcal{V}$, we define the $ n $-dimensional quadratic space over $F$:
		\begin{align*}
			W_{1}^{n}(c):=
			\begin{cases}
				\mathbb{H}^{\frac{n-2}{2}}\perp [1,-c] &\text{if $ n $ is even}, \\
				\mathbb{H}^{\frac{n-1}{2}}\perp [c]    &\text{if $ n $ is odd},  			\end{cases}
		\end{align*}
		and define the $ n $-dimensional quadratic space $ W_{2}^{n}(c) $ with $ \det  W_{2}^{n}(c)=\det  W_{1}^{n}(c) $ and $ W_{2}^{n}(c)\not\cong W_{1}^{n}(c) $ if $n\not=1$ and $(n,c)\not=(2,1)$.
			We further define the $ \mathcal{O}_{F} $-maximal lattice on $ W_{\nu}^{n}(c) $ by $ N_{\nu}^{n}(c) $ provided that $W_{\nu}^{n}(c)$ is defined.
	\end{defn}		
	From Definition \ref{defn:space-maximallattice}, the notations $W_{\nu}^{n}(c)$ and $N_{\nu}^{n}(c)$ are defined only if
	\begin{align}\label{w2nc}
		(n,\nu)\not=(1,2)\quad\text{and}\quad(n,\nu,c)\not=(2,2,1),
	\end{align}
	 which is in essential due to \cite[63:22 Theorem]{omeara_quadratic_1963}.  Hereafter, we always assume that the conditions \eqref{w2nc} hold when a quadratic space $W_{\nu}^{n}(c)$ or an $\mathcal{O}_{F}$-maximal lattice $N_{\nu}^{n}(c)$ is discussed.  
	
If $F$ is dyadic, for $ c\in \mathcal{V}\backslash \{1,\Delta\}$, we let $c\pi^{-\ord(c)}=\xi^{2}(1+\mu\pi^{d(c)})$ with $\xi,\mu\in \mathcal{O}_{F}^{\times}$ when $\ord(c)$ is even. To describe $W_{2}^{n}(c)$ and $N_{2}^{n}(c)$ explicitly, we put
	\begin{align}\label{csharp}
		c^{\#}:=
		\begin{cases}
			\Delta    &\text{if $\ord(c)$ is odd}, \\
			1+4\rho \mu^{-1}\pi^{-d(c)}  &\text{if $\ord(c)$ is even},
		\end{cases}
	\end{align}
as in \cite[Definition 3.1]{HeHu2}.	From \cite[Proposition 3.2]{HeHu2}, we also have the properties for $c^{\#}$:
	\begin{align}\label{csharp-2}
		 d(c^{\#})=2e-d(c)\quad\text{and}\quad(c^{\#},c)_{\mathfrak{p}}=-1.
	\end{align} 
\begin{prop}\label{prop:space}
		Let $ n\ge 1 $, $\nu\in \{1,2\}$ and $c\in \mathcal{V}$.
		
		 \begin{enumerate}[itemindent=-0.5em,label=\rm (\roman*)]
			\item  The quadratic space $ W_{\nu}^{n}(c) $  is given by the following table,
			\begin{center}
				\renewcommand\arraystretch{1.5}
				\begin{tabular}{c|c|c|c}
					\toprule[1.2pt]
					$ n $	& $ c $ & $ W_{1}^{n}(c) $  & $ W_{2}^{n}(c) $  \\
					\hline
					\multirow{4}*{\text{Even}}	  & $ 1 $ & $ \mathbb{H}^{\frac{n}{2}} $ & $ \mathbb{H}^{\frac{n-4}{2}} \perp [1,-\Delta,\pi,-\Delta\pi]$   \\
					\cline{2-4}
					& $ \Delta $  & $ \mathbb{H}^{\frac{n-2}{2}}\perp [1,-\Delta]  $ & $ \mathbb{H}^{\frac{n-2}{2}}\perp [\pi,-\Delta\pi] $  \\
					\cline{2-4}
					& $\delta, \delta\in \mathcal{U} \backslash \{1,\Delta\} $  & $ \mathbb{H}^{\frac{n-2}{2}}\perp [1,-\delta]$ & $ \mathbb{H}^{\frac{n-2}{2}}\perp [\delta^{\#},-\delta^{\#}\delta]$  \\
					\cline{2-4}
					& $ \delta\pi, \delta\in \mathcal{U} $  & $ \mathbb{H}^{\frac{n-2}{2}}\perp [1,-\delta\pi]$ & $ \mathbb{H}^{\frac{n-2}{2}}\perp [\Delta,-\Delta\delta\pi]$  \\
					\hline
					\multirow{2}*{\text{Odd}} & $ \delta,\delta\in\mathcal{U}  $ & $ \mathbb{H}^{\frac{n-1}{2}}\perp [\delta] $  &   $ \mathbb{H}^{\frac{n-3}{2}}\perp [\pi,-\Delta\pi,\Delta\delta] $  \\
					\cline{2-4}
					& $ \delta\pi,\delta\in\mathcal{U} $ & $ \mathbb{H}^{\frac{n-1}{2}}\perp [\delta\pi] $  &   $ \mathbb{H}^{\frac{n-3}{2}}\perp [1,-\Delta ,\Delta\delta\pi] $  \\
					\bottomrule[1.2pt]
				\end{tabular}
			\end{center}
			where the third case is ignored when $e=0$. (If $e=0$, then $\mathcal{U}\backslash \{1,\Delta\}=\emptyset$.)

			\item   Every $ n $-dimensional quadratic space over $ F $ is isometric to one of the spaces in the table above.
			
			\item  For every $ n $-dimensional quadratic space $ W $, up to isometry, there is exactly one $ (n+2) $-dimensional quadratic space $ V $ representing all $ n $-dimensional quadratic spaces except for $ W $. Precisely, if $ W=W_{\nu}^{n}(c) $ with $(n,\nu)\not=(1,2)$ and $(n,\nu,c)\not=(2,2,1)$, then $ V=W_{3-\nu}^{n+2}(c) $.
		\end{enumerate}   
\end{prop}
\begin{re}\label{re:spaces-maximal-lattices}
			All $ n $-dimensional quadratic spaces have been exhausted by the above table from Proposition \ref{prop:space}(ii). Also, on each space $ W_{\nu}^{n}(c) $, by \cite[91:2 Theorem]{omeara_quadratic_1963}, there is exactly one $ \mathcal{O}_{F} $-maximal lattice, up to isometry. Hence $\mathcal{M}_{n}$ consists of all the $\mathcal{O}_{F}$-maximal lattices $N_{\nu}^{n}(c)$ for $\nu\in \{1,2\}$ and $c\in \mathcal{V} $ such that $(n,\nu)\not=(1,2)$ and $(n,\nu,c)\not=(2,2,1)$. So, by Proposition \ref{prop:space}(i) and \cite[63:9]{omeara_quadratic_1963}, one can count the number of $\mathcal{O}_{F}$-maximal lattices with rank $n$:
		\begin{align}\label{maximal-num}
			|\mathcal{M}_{n}|=\begin{cases}
				8(N\mathfrak{p})^{e}  &\text{if $n\ge 3$}, \\
				8(N\mathfrak{p})^{e}-1  &\text{if $n=2$},\\
				 4(N\mathfrak{p})^{e}   &\text{if $n=1$}.
			\end{cases}
		\end{align} 
\end{re}
Next we show Lemma \ref{lem:B-beli}, which refines \cite[63:21 Theorem]{omeara_quadratic_1963} slightly and provides an alternative proof for Proposition \ref{prop:space}(ii) and (iii), covering the case $n=1$.
\begin{lem}\label{lem:B-beli}
	 Let $n\ge 1$, $\nu,\nu^{\prime}\in \{1,2\}$ and $c,c^{\prime}\in \mathcal{V}$. 
		
	 \begin{enumerate}[itemindent=-0.5em,label=\rm (\roman*)]
		
		\item $W_{\nu^{\prime}}^{n}(c^{\prime})$ represents $W_{\nu}^{n}(c)$, i.e., $W_{\nu^{\prime}}^{n}(c^{\prime})\cong W_{\nu}^{n}(c)$ if and only if $(\nu^{\prime},c^{\prime})=(\nu,c)$.
		
		\item $W_{\nu^{\prime}}^{n+1}(c^{\prime})$ represents $W_{\nu}^{n}(c)$ if and only if $(c^{\prime},c)_{\mathfrak{p}}=(-1)^{\nu^{\prime}+\nu}$.
		
		\item $W_{\nu^{\prime}}^{n+2}(c^{\prime})$ represents $W_{\nu}^{n}(c)$ if and only if $c^{\prime}\not=c$ or $(\nu^{\prime},c^{\prime})=(\nu,c)$.
	\end{enumerate}
\end{lem}
\begin{proof}
		(i) This is clear from Definition \ref{defn:space-maximallattice}.
		
		(ii) Let  $D=\det W_{1}^{n+1}(c^{\prime})\det W_{1}^{n}(c)$. Then  $D=\det W_{\nu^{\prime}}^{n+1}(c^{\prime})\det W_{\nu}^{n}(c)$ for any $\nu,\nu^{\prime}\in \{1,2\}$. By  \cite[63:21 Theorem]{omeara_quadratic_1963}, $W_{\nu}^{n}(c)\rep W_{\nu^{\prime}}^{n+1}(c^{\prime})$ if and only if $W_{\nu^{\prime}}^{n+1}(c^{\prime})\cong W_{\nu}^{n}(c)\perp [D]$. Since $\det W_{\nu^{\prime}}^{n+1}(c^{\prime})=\det(W_{\nu}^{n}(c)\perp [D])$, which is equivalent to $S_{\mathfrak{p}}(W_{\nu^{\prime}}^{n+1}(c^{\prime}))=S_{\mathfrak{p}}(W_{\nu}^{n}(c)\perp [D])$.
		
		Consider the case $\nu=\nu^{\prime}=1$. If $n$ is even, then $W_{1}^{n}(c)=\mathbb{H}^{n/2-1}\perp [1,-c]$ and $W_{1}^{n}(c^{\prime})=\mathbb{H}^{n/2}\perp [c^{\prime}]$, so $W_{1}^{n}(c)\rep W_{1}^{n+1}(c^{\prime})$ if and only if $[1,-c]\rep \mathbb{H}\perp [c^{\prime}]\cong [c,-c,c^{\prime}]$, which is equivalent to $1\rep [c,c^{\prime}]$, i.e., $(c,c^{\prime})_{\mathfrak{p}}=1$. If $n$ is odd, then $W_{1}^{n}(c)=\mathbb{H}^{(n-1)/2}\perp [c]$ and $W_{1}^{n+1}(c^{\prime})=\mathbb{H}^{(n-1)/2}\perp [1,-c^{\prime}]$, so $W_{1}^{n}(c)\rep W_{1}^{n+1}(c^{\prime})$ if and only if $c\rep [1,-c^{\prime}]$, which is equivalent to $(c,c^{\prime})_{\mathfrak{p}}=1$. Hence, regardless of the parity of $n$, we have
		\begin{align*}
				S_{\mathfrak{p}}(W_{1}^{n+1}(c^{\prime}))=S_{\mathfrak{p}}(W_{1}^{n}(c)\perp [D])\Longleftrightarrow W_{1}^{n}(c)\rep W_{1}^{n+1}(c^{\prime})\Longleftrightarrow (c,c^{\prime})_{\mathfrak{p}}=1.
		\end{align*}
		It follows that
		\begin{align}\label{sp1}
			S_{\mathfrak{p}}(W_{1}^{n+1}(c^{\prime}))=(c,c^{\prime})_{\mathfrak{p}}S_{\mathfrak{p}}(W_{1}^{n}(c)\perp [D] ).
		\end{align}
		Also, we have
		\begin{align}\label{sp2}
			 S_{\mathfrak{p}}(W_{1}^{n}(c)\perp [D])=(-1)^{\nu-1}S_{\mathfrak{p}}(W_{\nu}^{n}(c)\perp [D]).
		\end{align}
		Indeed, if $\nu=1$, this is trivial. And if $\nu=2$, then $\det(W_{1}^{n}(c)\perp [D])=\det (W_{2}^{n}(c)\perp [D])$, but $W_{1}^{n}(c)\perp [D]\not\cong W_{2}^{n}(c)\perp [D]$, so $S_{\mathfrak{p}}(W_{1}^{n}(c)\perp [D])=-S_{\mathfrak{p}}(W_{2}^{n}(c)\perp [D])$.
		
		Similarly, $S_{\mathfrak{p}}(W_{1}^{n+1}(c^{\prime}))=-S_{\mathfrak{p}}(W_{2}^{n+1}(c^{\prime}))$, so for $\nu^{\prime}\in \{1,2\}$, we have
		\begin{align}\label{sp3}
			S_{\mathfrak{p}}(W_{1}^{n+1}(c^{\prime}))=(-1)^{\nu^{\prime}-1}S_{\mathfrak{p}}(W_{\nu^{\prime}}^{n+1}(c^{\prime})).
		\end{align}
		Plugging \eqref{sp2} and \eqref{sp3} into \eqref{sp1}, we deduce that $S_{\mathfrak{p}}(W_{\nu^{\prime}}^{n+1}(c^{\prime}))=(-1)^{\nu+\nu^{\prime}}(c,c^{\prime})_{\mathfrak{p}}S_{\mathfrak{p}}(W_{\nu}^{n}(c)\perp [D]) $. Hence
		\begin{align*}
			W_{\nu}^{n}(c)\rep W_{\nu^{\prime}}^{n+1}(c^{\prime})\Longleftrightarrow S_{\mathfrak{p}}(W_{\nu^{\prime}}^{n+1}(c^{\prime}))=S_{\mathfrak{p}}(W_{\nu}^{n}(c)\perp [D]) \Longleftrightarrow (-1)^{\nu+\nu^{\prime}}(c,c^{\prime})_{\mathfrak{p}}=1.
		\end{align*}
		
		(iii) If $V$ and $W$ are two quadratic spaces such that $\dim V-\dim W=2$, then $W\rep V$ if and only if either $\det V\not=-\det W=\det (W\perp \mathbb{H})$ or $V\cong W\perp \mathbb{H}$. In our case, since $W_{\nu}^{n}(c)\perp \mathbb{H}=W_{\nu}^{n+2}(c)$, we have that $W_{\nu^{\prime}}^{n+2}(c^{\prime})$ represents $W_{\nu}^{n}(c)$ if and only if either $\det W_{\nu^{\prime}}^{n+2}(c^{\prime})\not=\det W_{\nu}^{n+2}(c)$, i.e., $ c^{\prime}\not=c$, or $W_{\nu^{\prime}}^{n+2}(c^{\prime}) \cong  W_{\nu}^{n+2}(c)$, i.e., $(\nu^{\prime},c^{\prime})=(\nu,c)$.
\end{proof}
\begin{lem}\label{lem:spacerep-criterion}
	Let $V$ be a quadratic space over $F$. Let $W_{1}$ and $W_{2}$ be $n$-dimensional quadratic spaces over $F$ such that $\det W_{1}=\det W_{2}=D$ and $W_{1}\not\cong W_{2}$. 				
	\begin{enumerate}[itemindent=-0.5em,label=\rm (\roman*)]
		\item For $n\ge 2$, suppose either $\dim V=n+1$, or $\dim V=n+2$ with $\det V=-D$. Then $V$ represents precisely one of $W_{1}$ and $W_{2}$. 
						
		In particular, for every $c$, $ V $ represents exactly one of $ W_{1}^{n}(c) $ and $ W_{2}^{n}(c) $.
						
		\item For $n\ge 3$, suppose either $\dim V=n-1$, or $\dim V=n-2$ with $\det V=-D$. Then $V$ is represented by precisely one of $W_{1}$ and $W_{2}$.
						
		In particular, for every $c$, $ V $ is represented by exactly one of $ W_{1}^{n}(c) $ and $ W_{2}^{n}(c) $.
	\end{enumerate}
\end{lem}
\begin{proof}
	(i) See \cite[Lemma 3.13]{HeHu2}.
				
	(ii) Let $\dim V=n-1$. Recall from \cite[63:21 Theorem]{omeara_quadratic_1963} that $V\rep W$ if and only if $W\cong V\perp [\det W\det V]$. Since $W_{1}$ and $W_{2}$ are the exactly non-isometric $n$-dimensional spaces with the same determinant $D$, we see that $V\perp [D\det V]\cong W_{1}$ or $ W_{2}$, but not both. Hence $V$ is represented by precisely one of $W_{1}$ and $W_{2}$.
				
	If $\dim V=n-2$ and $\det V=-D$, then, by \cite[63:21 Theorem]{omeara_quadratic_1963}, $V\rep W$ if and only if $W\cong V\perp \mathbb{H}$. Similar to the previous case, we see that $V\perp \mathbb{H}\cong W_{1}$ or $W_{2}$, but not both, as desired.
\end{proof}
 
Under the $n$-ADC assumption, we also have the lattice versions of Lemma \ref{lem:spacerep-criterion}(i) and Proposition \ref{prop:space}(iii).
\begin{lem}\label{lem:spacerep-criterion-ADC}
	Let $n\ge 2$, $\nu\in\{1,2\}$ and $ c\in \mathcal{V} $. Let $M$ be an $n$-ADC $\mathcal{O}_{F}$-lattice.
	\begin{enumerate}[itemindent=-0.5em,label=\rm (\roman*)]
			\item  If either $ \rank\, M=n+1 $, or $  \rank\, M=n+2  $ and $ \det FM=-\det W_{\nu}^{n}(c) $, then $ M $ represents exactly one of $ N_{1}^{n}(c) $ and $ N_{2}^{n}(c) $.
			
			\item If $ FM\cong W_{\nu}^{n+2}(c)$, then $ M $ represents every  lattice $ N $ in $\mathcal{L}_{F,n}$ with $ FN\not\cong W_{3-\nu}^{n}(c) $. In particular, $ M $ represents every $ N $ in $ \mathcal{M}_{n} $ with $ N\not\cong N_{3-\nu}^{n}(c) $.
	\end{enumerate}
\end{lem}
\begin{proof}
	This follows from Lemma \ref{lem:spacerep-criterion}(i), Proposition \ref{prop:space}(iii) and the $ n $-ADC-ness of $ M $.
\end{proof}
	
 We treat the non-dyadic and dyadic case separately.
 
\medskip 
\noindent \textbf{Case I.}  $F$ is non-dyadic 

Recall from \cite[92:2 Theorem]{omeara_quadratic_1963} that a lattice $L$ over $ F $ has a unique Jordan splitting. Hence we may denote by $ J_{k}(L) $ the Jordan component of $ L $, with possible zero rank and scale $  \mathfrak{p}^{k} $, and write $ J_{i,j}(L):=J_{i}(L)\perp J_{i+1}(L)\perp\cdots\perp J_{j}(L) $ for integers $ i\le j $. 
	
We reformulate \cite[Proposition 3.2]{hhx_indefinite_2021} as below.
\begin{lem}\label{lem:maximallattices-non-dyadic}
	Let $n\ge 1$, $\nu\in\{1,2\}$ and $c\in \mathcal{V}$.
	 \begin{enumerate}[itemindent=-0.5em,label=\rm (\roman*)]
		\item The $ \mathcal{O}_{F} $-maximal lattice $ N_{\nu}^{n}(c) $ is given by the following table.
		\begin{center}
	 	\renewcommand\arraystretch{1.5}
			\begin{tabular}{c|c|c|c}
				\toprule[1.2pt]
				$ n $	& $ c $ & $ N_{1}^{n}(c) $  & $ N_{2}^{n}(c) $  \\
				\hline
				\multirow{3}*{\text{Even}}	  & $ 1 $ & $ \mathbf{H}^{\frac{n}{2}} $ & $ \mathbf{H}^{\frac{n-4}{2}} \perp \langle 1,-\Delta,\pi,-\Delta\pi\rangle $   \\
				\cline{2-4}
				& $ \Delta $  & $ \mathbf{H}^{\frac{n-2}{2}}\perp \langle 1,-\Delta\rangle  $ & $ \mathbf{H}^{\frac{n-2}{2}}\perp \langle \pi,-\Delta\pi\rangle $  \\
				\cline{2-4}
				& $ \delta\pi, \delta\in \mathcal{U} =\{1,\Delta\} $  & $ \mathbf{H}^{\frac{n-2}{2}}\perp \langle 1,-\delta\pi\rangle $ & $ \mathbf{H}^{\frac{n-2}{2}}\perp \langle \Delta,-\Delta\delta\pi\rangle$  \\
				\hline
				\multirow{2}*{\text{Odd}} & $ \delta,\delta\in \mathcal{U} =\{1,\Delta\}$ & $ \mathbf{H}^{\frac{n-1}{2}}\perp \langle \delta\rangle  $  &   $ \mathbf{H}^{\frac{n-3}{2}}\perp \langle \pi,-\Delta\pi,\Delta\delta\rangle $  \\
				\cline{2-4}
				& $ \delta\pi,\delta\in \mathcal{U} =\{1,\Delta\} $ & $ \mathbf{H}^{\frac{n-1}{2}}\perp \langle \delta\pi\rangle $  &   $ \mathbf{H}^{\frac{n-3}{2}}\perp \langle 1,-\Delta ,\Delta\delta\pi\rangle$  \\
				\bottomrule[1.2pt]
			\end{tabular}
		\end{center}
		
		\item  The set $ \mathcal{M}_{n} $ is a minimal testing set for $ n $-universality.
	 \end{enumerate}
\end{lem}
\begin{lem}\label{lem:maximal-rep-nondyadic}
	Let $N=N_{\nu}^{n}(c)$ be an $\mathcal{O}_{F}$-maximal lattice of rank $n$. Then $J_{0,1}(N)=N$. 
		
	Thus, an $\mathcal{O}_{F}$-lattice $M$ represents $N$ if and only if $FJ_{0}(M)$ represents $FJ_{0}(N)$ and $FJ_{0,1}(M)$ represents $FN$.
\end{lem}
\begin{proof}
	Recall that $\mathbf{H}=\langle 1,-1\rangle$ is unimodular, and so is $\mathbf{H}^{k}$. For each $N=N_{\nu}^{n}(c)$ in Lemma \ref{lem:maximallattices-non-dyadic}(i), one may obtain the Jordan splittings by collecting or reordering the components according to their scales. From these splittings, it follows that $J_{0,1}(N)=N$ for each $N=N_{\nu}^{n}(c)$. Furthermore, the second assertion holds by \cite[Theorem 1]{omeara_integral_1958}. 
\end{proof}
	
\noindent \textbf{Case II.} $F$ is dyadic
	
We rephrase \cite[Theorem 1.2]{HeHu2} in terms of BONGs by virtue of \cite[Remark 3.9, Lemmas 3.10 and 3.11]{HeHu2}
	
\begin{lem}\label{lem:maximallattices-dyadic}
	Let $n\ge 1$, $\nu\in\{1,2\}$ and $c\in \mathcal{V}$.
		\begin{enumerate}[itemindent=-0.5em,label=\rm (\roman*)]
			\item The $ \mathcal{O}_{F} $-maximal lattice $ N_{\nu}^{n}(c) $ is given by the following table,
        \begin{center}
        	\renewcommand\arraystretch{1.5}
        	\hskip -0.8cm \begin{tabular}{c|c|c|c}
        		\toprule[1.2pt]
        		$ n $& $ c $ & $ N_{1}^{n}(c) $  & $ N_{2}^{n}(c) $  \\
        		\hline
        		\multirow{4}*{\text{Even}}	 & $ 1 $ & $ \mathbf{H}^{\frac{n}{2}} $ & $ \mathbf{H}^{\frac{n-4}{2}} \perp  \prec 1,-\Delta\pi^{-2e}, \pi,-\Delta\pi^{1-2e}\succ$ \\
        		\cline{2-4}
        		& $ \Delta $  & $ \mathbf{H}^{\frac{n-2}{2}}\perp   \prec 1,-\Delta\pi^{-2e}\succ $ & $ \mathbf{H}^{\frac{n-2}{2}}\perp \prec \pi,-\Delta\pi^{1-2e}\succ $  \\
        		\cline{2-4}
        		& $ \delta,\delta\in \mathcal{U}\backslash \{1,\Delta\} $  & $ \mathbf{H}^{\frac{n-2}{2}}\perp \prec 1,-\delta\pi^{1-d(\delta)}\succ$ & $ \mathbf{H}^{\frac{n-2}{2}}\perp \prec\delta^{\#},-\delta^{\#}\delta\pi^{1-d(\delta)}\succ$  \\
        		\cline{2-4}
        		& $ \delta\pi, \delta\in \mathcal{U} $  & $ \mathbf{H}^{\frac{n-2}{2}}\perp \prec  1,-\delta\pi\succ$ & $ \mathbf{H}^{\frac{n-2}{2}}\perp \prec\Delta,-\Delta\delta\pi\succ $  \\
        		\hline
        		\multirow{2}*{\text{Odd}}	 & $ \delta,\delta\in\mathcal{U}  $ & $ \mathbf{H}^{\frac{n-1}{2}}\perp \prec \delta\succ $  &   $ \mathbf{H}^{\frac{n-3}{2}}\perp \prec \delta\kappa^{\#}, -\delta\kappa^{\#}\kappa\pi^{2-2e}, \delta \kappa\succ 
        		$  \\
        		\cline{2-4}
        		& $ \delta\pi,\delta\in\mathcal{U}  $ & $ \mathbf{H}^{\frac{n-1}{2}}\perp \prec\delta\pi\succ $  &   $ \mathbf{H}^{\frac{n-3}{2}}\perp \prec 1,-\Delta\pi^{-2e},\Delta\delta\pi\succ $  \\
        		\bottomrule[1.2pt]
        	\end{tabular}
        \end{center}
			where $ \kappa $ is a fixed unit with $ d(\kappa)=2e-1 $ and $\mathbf{H}\cong \prec 1,-\pi^{-2e}\succ$ (cf. \cite[Lemma 3.9(i)]{HeHu2}).
			
			\item The set $ \mathcal{M}_{n} $ is a minimal testing set for $ n $-universality.			
		\end{enumerate}
 \end{lem}	
\begin{re}\label{re:invariant}
	Each $ \mathcal{O}_{F} $-maximal lattice $ N_{\nu}^{n}(c) $ can be written as the form $ \mathbf{H}^{k}\perp L\cong \prec 1,-\pi^{-2e},\ldots, 1,-\pi^{-2e}, c_{1},\ldots,c_{\ell}\succ $ relative to a good BONG, where $ k,\ell $ are non-negative integers and $ L\cong \prec c_{1},\ldots,c_{\ell} \succ $ relative to a good BONG. Hence the above table gives the values of the invariants $ R_{i}(N_{\nu}^{n}(c)) $ (cf. \cite[Lemma 3.11]{HeHu2}).
\end{re}
	
The next two lemmas indicate that the invariants $ R_{i}(N) $ ($ 1\le i\le n $) and the space $ FN $ determine whether an $\mathcal{O}_{F}$-lattice $N$ of rank $n$ is $\mathcal{O}_{F}$-maximal or not. The proofs are the same as that of \cite[Proposition 3.7]{HeHu2}. (See also \cite[Lemma 3.11]{HeHu2}.)
\begin{lem}\label{lem:maximal-BONG-even}
		Let $ N $ be an $ \mathcal{O}_{F} $-lattice of even rank $ n\ge 2 $ and put $ S_{i}=R_{i}(N) $.
	\begin{enumerate}[itemindent=-0.5em,label=\rm (\roman*)]
		\item 	 If $ FN\cong W_{1}^{n}(c)$ with $ c\in \{1,\Delta\} $, then $ N\cong N_{1}^{n}(c) $ if and only if $ S_{i}=0 $ for $ i\in [1,n]^{O} $ and $ S_{i}=-2e $ for $ i\in [1,n]^{E} $.
		
		\item If $ FN\cong W_{2}^{n}(c)$ with $ c\in \{1,\Delta\} $, then $ N\cong N_{2}^{n}(c) $ if and only if  $ S_{i}=0 $  for $ i\in [1,n-2]^{O} $, $ S_{i}=-2e $ for $ i\in [1,n-2]^{E} $,
		$ S_{n-1}=1 $ and $ S_{n}=1-2e $.
		
		\item If $ FN\cong W_{\nu}^{n}(c)$, with $ \nu\in\{1,2\} $ and $ c\in \mathcal{V}\backslash \{1,\Delta\} $, then $ N\cong N_{\nu}^{n}(c) $ if and only if $ S_{i}=0 $ for $ i\in [1,n-1]^{O} $, $ S_{i}=-2e $ for $ i\in [1,n-1]^{E} $ and $ S_{n}=1-d(c)$.
		\end{enumerate}
\end{lem}
\begin{lem}\label{lem:maximal-BONG-odd}
		Let $ N $ be an $ \mathcal{O}_{F} $-lattice of odd rank $ n\ge 1 $ and put $ S_{i}=R_{i}(N) $.
		\begin{enumerate}[itemindent=-0.5em,label=\rm (\roman*)]
			\item 	 If $ FN\cong W_{1}^{n}(\delta)$ with $ \delta\in \mathcal{U}  $, then $ N\cong N_{1}^{n}(\delta) $ if and only if $ S_{i}=0 $ for $ i\in [1,n]^{O} $ and $ S_{i}=-2e $ for $ i\in [1,n]^{E} $.
			
			\item If $ FN\cong W_{2}^{n}(\delta)$ with $ \delta\in \mathcal{U}  $, then $ N\cong N_{2}^{n}(\delta) $ if and only if $ S_{i}=0 $  for $ i\in [1,n]^{O} $, $ S_{i}=-2e $ for $ i\in [1,n-2]^{E} $
			and $ S_{n-1}=2-2e $.
			
			\item  If $ FN\cong W_{\nu}^{n}(\delta\pi)$, with $ \nu\in \{1,2\} $ and $ \delta\in \mathcal{U}  $, then $ N\cong N_{\nu}^{n}(\delta\pi) $ if and only if  $ S_{i}=0 $ for $ i\in [1,n-1]^{O} $, $ S_{i}=-2e $ for $ i\in [1,n-1]^{E} $ and $ S_{n}=1$.
		\end{enumerate}
\end{lem}
\begin{prop}\label{prop:maximalproperty}
	Let $ N\cong \prec b_{1},\ldots,b_{n}\succ $ be an $ \mathcal{O}_{F} $-maximal lattice of odd rank $ n\ge 3 $. Put $ S_{i}=R_{i}(N) $ and $\beta_{i}=\alpha_{i}(N)$. Then
	\begin{enumerate}[itemindent=-0.5em,label=\rm (\roman*)]
		\item 	 $ S_{i}=0 $ for $ i\in [1,n-2]^{O} $, $ S_{i}=-2e $ for $ i\in [1,n-2]^{E} $ and $ S_{n-1}\in \{-2e,2-2e\}  $.
		
		\item If $ S_{n-1}=-2e $, then $ S_{n} \in \{0,1\} $, $ \beta_{n-2}=0 $ and $ \beta_{n-1}\ge d[-b_{n-2,n-1}]\ge 2e $.
		
		\item  If $ S_{n-1}=2-2e $, then $ S_{n}=0 $, $ \beta_{n-2}=1 $ and $  d[-b_{n-2,n-1}]=\beta_{n-1}=2e-1 $.
	\end{enumerate}	 
\end{prop}
\begin{proof}
		(i) It is clear from Lemma \ref{lem:maximal-BONG-odd}.
		
		(ii) If $ S_{n-1}=-2e $, then $S_{n}\in \{0,1\}$ by Lemma \ref{lem:maximal-BONG-odd}. Since $S_{n-1}-S_{n-2}=-2e$, by Proposition \ref{prop:alphaproperty}(ii) and (iv), we have  $ \beta_{n-2}=0 $ and $ \beta_{n-1}\ge d[-b_{n-2,n-1}]\ge 2e$.
		
		(iii) If $ S_{n-1}=2-2e $, then $S_{n}=0$ by Lemma \ref{lem:maximal-BONG-odd}. Since $S_{n-1}-S_{n-2}=2-2e$ and $ S_{n}-S_{n-1}=2e-2 $, by Proposition \ref{prop:Rproperty}(ii), we have $ \beta_{n-2}=1 $ and $ \beta_{n-1}=2e-1 $. Hence
		\begin{align*}
			2e-1=(2e-2)+1=S_{n-2}-S_{n-1}+\beta_{n-2}\le d[-b_{n-2,n-1}]\le \beta_{n-1}=2e-1
		\end{align*}
		by \eqref{eq:alpha-defn}, as desired.
\end{proof}

 \medskip
 We return to the case where $F$ is a local field. The following lemma shows that, over local fields, maximality implies $n$-ADC-ness.
\begin{lem}\label{lem:maximalrep}
   	Let $ M $ be an $ \mathcal{O}_{F} $-maximal lattice. If $ FM $ represents $ FN $, then $ M $ represents $ N $; thus $ M $ is $ n $-ADC for  $ 1\le n\le \rank\, M $. 
 \end{lem}
 \begin{proof}
   	If $ FM $ represents $ FN $, by \cite[Proposition 1]{omeara_integral_1958}, $ FM\cong FN \perp V $ for some quadratic space. Take an integral lattice $ L $ on $ V $. Then $ \mathfrak{n}(N\perp L)\subseteq \mathcal{O}_{F} $. By \cite[82:18]{omeara_quadratic_1963} and \cite[91:2 Theorem]{omeara_quadratic_1963}, there must be an $\mathcal{O}_{F}$-maximal lattice $ M^{\prime} $ of rank $n$ on $ FM $ such that $
   	N\subseteq N\perp L\subseteq M^{\prime}\cong M $. Thus $M$ represents $N$.
 \end{proof}
   
The following proposition characterizes $n$-ADC $\mathcal{O}_{F}$-lattices of rank $n$, thereby proving the simple case of Theorem \ref{thm:locallyn-ADCn+1}(i).
  
\begin{prop}\label{prop:n-ADC-n}
   	Let $ M $ be an $ \mathcal{O}_{F} $-lattice of rank $n\ge 2$. Then $ M $ is $ n $-ADC if and only if $M$ is $ \mathcal{O}_{F} $-maximal.
 \end{prop}
 \begin{proof}
   	Sufficiency is clear from Lemma \ref{lem:maximalrep}. From Remark \ref{re:spaces-maximal-lattices}, we may choose an $\mathcal{O}_{F}$-maximal lattice $N$ of rank $n$ such that $FN\cong FM$. Then $FN\rep FM$ by \cite[63:21 Theorem]{omeara_quadratic_1963}. Since $M$ is $n$-ADC, we have $N\rep M$, so $M\cong N$ by the maximality of $N$.
  \end{proof}
  Using Lemmas \ref{lem:maximallattices-non-dyadic}(i) and \ref{lem:maximallattices-dyadic}(i) with $n=4$, one can easily prove the following proposition for quaternary maximal lattices, which will be used in the proof of Theorem \ref{thm:Z2-ADCquaternary}.
 \begin{prop}\label{prop:maximal-quaternary}
 	Let $N$ be a quaternary $\mathcal{O}_{F}$-maximal lattice. Then $N$ represents $\mathbf{H}$ except when $N=N_{2}^{4}(1)$. In the exceptional case, $N\cong \mathbf{A}\perp \mathbf{A}^{(\pi)} $, 
 	where $\mathbf{A}^{(\pi)}$ denotes the lattice $\mathbf{A}$ scaled by $\pi$.
 \end{prop}
 
\section{$ n $-ADC lattices over non-dyadic local fields}\label{sec:n-ADCnon-dyadicfields}
	
Throughout this section, let $n$ be an integer with $n\ge 2$. We assume that $F$ is a non-dyadic local field and $M$ is an $\mathcal{O}_{F}$-lattice. In this case, we have $\mathcal{U}=\{1,\Delta\}$ and  $\mathcal{V}=\{1,\Delta,\pi,\Delta\pi\}$.
	
\begin{thm}\label{thm:nondyadic-n-ADCn+1n+2}
	 		 If $\rank\, M=n+1$ or $n+2$, then $ M $ is $ n $-ADC if and only if $M$ is $ \mathcal{O}_{F} $-maximal.
\end{thm}
\begin{lem}\label{cor:unimodular}  
	  Suppose that $M=J_{0,1}(M)$ and $\rank\,J_{1}(M)\le 1$. Then $M$ is $\mathcal{O}_{F}$-maximal; thus it is $n$-ADC.
	  
	  In particular, if $M$ is unimodular, then it is $\mathcal{O}_{F}$-maximal and $n$-ADC.
\end{lem}
\begin{proof}    	
	By the hypothesis, we have $\mathfrak{n}(M)=\mathfrak{s}(M)=\mathcal{O}_{F}$ and $\mathfrak{v}(M)\supseteq \mathfrak{p}$. It follows from \cite[82:19]{omeara_quadratic_1963} that $M$ is $\mathcal{O}_{F}$-maximal. Hence it is $n$-ADC from Lemma \ref{lem:maximalrep}.
	
    If $M$ is unimodular, then $M=J_{0}(M)$ and $\rank\,J_{1}(M)=0$, so the first assertion applies.
\end{proof}
\begin{lem}\label{lem:non-dyadic-rep-Nnu-eta}
	 \begin{enumerate}[itemindent=-0.5em,label=\rm (\roman*)]
			\item If for every $\delta\in \mathcal{U}$ there is some $\nu_{\delta}\in \{1,2\}$ such that  $M$ represents $N_{\nu_{\delta}}^{n}(\delta\pi)$, then $\rank\,J_{0}(M)\ge n-1$ and $\rank\,J_{0,1}(M)\ge n+1$.
			
			 \item If $M$ represents $N_{1}^{n}(\delta)$ for some $\delta\in \mathcal{U} $, then $\rank\,J_{0}(M)\ge n$.
			 
			 \item  If $M$ represents $N_{1}^{n}(\delta)$ for all $\delta\in \mathcal{U}$, then $\rank\,J_{0}(M)\ge n+1$.

		    \item  If $M$ represents both $N_{1}^{n}(c)$ and $N_{2}^{n}(c)$ for some $c\in \mathcal{V}$, then $\rank\,J_{0,1}(M)\ge n+2$.
			 	\end{enumerate}	
\end{lem}
\begin{proof}
		 (i) By Lemma \ref{lem:maximal-rep-nondyadic}, $FJ_{0}(M)$ represents $FJ_{0}(N_{\nu_{\delta}}^{n}(\delta \pi))$, which implies that
		 \begin{align*}
		 	\rank\,J_{0}(M)\ge \rank\, J_{0}(N_{\nu_{\delta}}^{n}(\delta \pi))=n-1.
		 \end{align*}
		 Also, $FJ_{0,1}(M)$ represents $FN_{\nu_{\delta}}^{n}(\delta\pi)=W_{\nu_{\delta}}^{n}(\delta\pi)$ for $\delta=1,\Delta$. Since $W_{\nu_{1}}^{n}(\pi)\not\cong W_{\nu_{\Delta}}^{n}(\Delta\pi)$ by Lemma \ref{lem:B-beli}(i), this implies that
		 \begin{align*}
		 	\rank\, J_{0,1}(M)=\dim FJ_{0,1}(M)\ge n+1.
		 \end{align*}		
		  
		 (ii) Observe from Lemma \ref{lem:maximallattices-non-dyadic}(i) that $N_{1}^{n}(\varepsilon)$ is unimodular for any $\varepsilon\in \mathcal{U}$, so $J_{0}(N_{1}^{n}(\varepsilon))=N_{1}^{n}(\varepsilon)$. By Lemma \ref{lem:maximal-rep-nondyadic}, $FJ_{0}(M)$ represents $FJ_{0}(N_{1}^{n}(\delta ))=FN_{1}^{n}(\delta)=W_{1}^{n}(\delta)$, which implies that
		 \begin{align*}
		 	\rank\,J_{0}(M)\ge \rank\,  N_{1}^{n}(\delta )=n.
		 \end{align*}
		 
		 (iii) By Lemma \ref{lem:maximal-rep-nondyadic}, $FJ_{0}(M)$ represents $FJ_{0}(N_{1}^{n}(\delta))=W_{1}^{n}(\delta)$  for $\delta=1,\Delta$. Since $W_{1}^{n}(1)\not\cong W_{1}^{n}(\Delta)$ by Lemma \ref{lem:B-beli}(i), this implies that
		  \begin{align*}
		 	\rank\, J_{0}(M)=\dim FJ_{0}(M)\ge n+1.
		 \end{align*}	 
		 
		 (iv) By Lemma \ref{lem:maximal-rep-nondyadic}, $FJ_{0,1}(M)$ represents $FN_{\nu}^{n}(c)=W_{\nu}^{n}(c)$ for $\nu=1,2$. This contradicts Lemma \ref{lem:spacerep-criterion}(i) if $\dim FJ_{0,1}(M)=n+1$. Hence
		 \begin{align*}
		 	\rank\, J_{0,1}(M)=\dim FJ_{0,1}(M)\ge n+2.
		 \end{align*} 
\end{proof}
\begin{lem}\label{lem:nondyadic-n-ADCn+1n+2}
		Suppose that $ M $ is $ n $-ADC of rank $n+1$ or $n+2$. Then $M=J_{0,1}(M)$ and $\rank\, J_{1}(M)\le 2$. 
\end{lem}
 \begin{proof} 
 	Let $\rank\, M=n+1$.  For each $\delta\in \{1,\Delta\}$, by Lemma \ref{lem:spacerep-criterion-ADC}(i), there is some $\nu_{\delta}\in \{1,2\}$ such that $M$ represents $N_{\nu_{\delta}}^{n}(\delta \pi)$. Hence, by Lemma \ref{lem:non-dyadic-rep-Nnu-eta}(i), we have
 	\begin{align*} 
 		\rank\, J_{0}(M)\ge n-1\quad\text{and}\quad\rank\, J_{0,1}(M)\ge n+1.
 	\end{align*}
 	 So $\rank\, J_{0,1}(M)=n+1$. Thus $J_{0,1}(M)=M$ and $\rank\, J_{1}(M)\le 2$.
 
 	Let $\rank\,M=n+2$ and $FM\cong W_{\nu}^{n+2}(c)$. Since $|\mathcal{U}|=2$ and $|\mathcal{V}|=4$, we have $\mathcal{U}\backslash \{c\}\not=\emptyset$ and $\mathcal{V}\backslash \{1,c\}\not=\emptyset$. Let $\delta\in \mathcal{U}\backslash\{c\}$ and $c^{\prime}\in \mathcal{V}\backslash \{1,c\}$. Since $\delta\not=c$, by Lemma \ref{lem:B-beli}(iii), $FM$ represents $W_{1}^{n}(\delta)$. Since $c^{\prime}\not=1$, both $W_{1}^{n}(c^{\prime})$ and $W_{2}^{n}(c^{\prime})$ are defined (including the case $n=2$) and, since $c^{\prime}\not=c$, by Lemma \ref{lem:B-beli}(iii), $FM$ represents both of them. Then, since $M$ is $n$-ADC, it represents $N_{1}^{n}(\delta)$, $N_{1}^{n}(c^{\prime})$ and $N_{2}^{n}(c^{\prime})$. By Lemma \ref{lem:non-dyadic-rep-Nnu-eta}(ii) and (iv), we get 
 	\begin{align*}
 		 \rank\, J_{0}(M)\ge n\quad\text{and}\quad\rank\, J_{0,1}(M)\ge n+2.
 	\end{align*}
 	So $\rank\, J_{0,1}(M)=n+2$. Thus $J_{0,1}(M)=M$ and $ \rank\, J_{1}(M)\le 2$.
 \end{proof}
\begin{proof}[Proof of Theorem \ref{thm:nondyadic-n-ADCn+1n+2}]
     	 Sufficiency is clear from Lemma \ref{lem:maximalrep}. Let $m=\rank\,M\in \{n+1,n+2\}$. By Lemma \ref{lem:nondyadic-n-ADCn+1n+2}, $M=J_{0,1}(M)$ and $\rank\,J_{1}(M)\le 2$. If $\rank\, J_{1}(M)\le 1$, then we are done by Lemma \ref{cor:unimodular}. 
     	 
     	 Assume $\rank\, J_{1}(M)=2$ and let $ M=J_{0}(M)\perp  M^{\prime (\pi)} $, where  $ M^{\prime}$ is unimodular of rank $2$. Since $J_{0}(M)$ and $M^{\prime}$ are unimodular, both are $\mathcal{O}_{F}$-maximal from Lemma \ref{cor:unimodular}. Hence, by Lemma \ref{lem:maximallattices-non-dyadic}(i),
     	 \begin{align*}
     	 	J_{0}(M)\cong N_{1}^{m-2}(1)\;\;\text{or}\;\;N_{1}^{m-2}(\Delta),\quad\text{and}\quad M^{\prime}\cong \mathbf{H}\;\;\text{or}\;\;\langle 1,-\Delta\rangle.
     	 \end{align*}
     	  If $M^{\prime}\cong \langle 1,-\Delta\rangle$, then $M\cong N_{2}^{m}(1)$ or $N_{2}^{m}(\Delta)$, as desired. Assume $M^{\prime}\cong \mathbf{H}$. Then $ FM\cong W_{1}^{m}(\eta^{\prime}) $ for some $ \eta^{\prime}\in \{1,\Delta\} $. Hence $FM$ represents $W_{1}^{n}(\eta)$ with $\eta\in \{1,\Delta\}$, by Lemma \ref{lem:B-beli}(ii), for $m=n+1$ (we have $(\eta,\eta^{\prime})_{\mathfrak{p}}=1$) and, by Lemma \ref{lem:B-beli}(iii), for $m=n+2$, so $M$ represents both of $N_{1}^{n}(1)$ and $N_{1}^{n}(\Delta)$ by $n$-ADC-ness of $M$. By Lemma \ref{lem:non-dyadic-rep-Nnu-eta}(iii), we have $\rank\, J_{0}(M)\ge n+1$, a contradiction.
\end{proof}
    
\section{ $ n $-ADC lattices over dyadic local fields I}\label{sec:even-n-ADC}\label{sec:n-ADCdyadicfields-even}
	
In this section, let $ n $ be an even integer with $ n\ge 2 $. We assume that $F$ is a dyadic local field and $  M\cong \prec a_{1},\ldots,a_{m}\succ $ is an $ \mathcal{O}_{F} $-lattice of rank $ m\ge n $, relative to some good BONG. Let $ R_{i}=R_{i}(M) $ for $ 1\le i\le m $ and $ \alpha_{i}=\alpha_{i}(M) $ for $ 1\le i\le m-1 $. We also suppose that $N\cong \prec b_1,\cdots, b_n\succ$ is an $\mathcal{O}_{F}$-lattice of rank $n$, relative to some good BONG, and denote its associated invariants by $S_i=R_i(N)$ and $\beta_i=\alpha_i(N)$ when an $\mathcal{O}_F$-lattice $N$ with rank $ n $ is discussed.
\begin{thm}\label{thm:dyadicACDeven-n+1}
		If $\rank\, M=n+1$, then $ M $ is $ n $-ADC if and only if $ M $ is $ \mathcal{O}_{F} $-maximal.		
\end{thm}
\begin{thm}\label{thm:dyadicACDeven}
	 	 If $\rank\, M=n+2$, then $ M $ is $ n $-ADC  if and only if either $ M $ is $ \mathcal{O}_{F} $-maximal, or $ n=2 $ and
	 	  \begin{align*}
	 	  	M\cong \mathbf{H}\perp \prec 1,-\Delta\pi^{2-2e}\succ\cong \mathbf{H}\perp 2^{-1}\pi A(2\pi^{-1},2\rho\pi),
	 	  \end{align*}
	 	 which is not $\mathcal{O}_{F}$-maximal. 	
\end{thm}
\begin{re}\label{re:dyadicACDeven}
	 	When $e=1$, by \cite[Corollary 3.4(ii)]{beli_integral_2003} and \cite[Lemma 3.10]{HeHu2}, we also have $\mathbf{H}\perp \prec 1,-\Delta\pi^{2-2e}\succ\cong \mathbf{H}\perp \langle 1,-\Delta\rangle$. 
\end{re}
\begin{lem}\label{lem:repN1N2}			
		\begin{enumerate}[itemindent=-0.5em,label=\rm (\roman*)]		
		\item If $ M $ represents $ N_{1}^{n}(\Delta) $ (resp. $ N_{1}^{n}(1) $), then $ R_{i-1}=R_{i}+2e=0$ for $  i\in [1,n]^{E} $. If moreover $ R_{n+1}>0 $, then $ d((-1)^{n/2}a_{1,n})=2e $ (resp. $ d((-1)^{n/2}a_{1,n})=\infty $).
		
		\item If $ M $ represents $ N_{1}^{n}(1) $ and $ N_{1}^{n}(\Delta) $, then $ R_{i-1}=R_{i}+2e=0$ for $  i\in [1,n]^{E} $ and $ R_{n+1}=0$.
		
		\item   If $ M $ represents $ N_{2}^{n}(\Delta) $ (resp. $ N_{2}^{n}(1) $, with $ n\ge 4 $), then $ R_{i-1}=R_{i}+2e=0$ for $  i\in [1,n-2]^{E} $ and either $ R_{n-1}=0$ and $ R_{n}\in \{-2e,2-2e\} $  or $ R_{n-1}=R_{n}+2e=1 $. 
		
		\item If $ M $ represents one of $N_{1}^{n}(1) $, $ N_{1}^{n}(\Delta) $ and $ N_{2}^{n}(\Delta) $, and one of $ N_{1}^{n}(\kappa)$ and $ N_{2}^{n}(\kappa) $, then $ R_{n+1}\in \{0,1,2\}$. (Here $\kappa$ is the unit with $d(\kappa)=2e-1$ from Lemma \ref{lem:maximallattices-dyadic}(i).)
	 	\end{enumerate}	 
\end{lem}
\begin{proof}
		(i) Let $N=N_{1}^{n}(\Delta)$ or $N_{1}^{n}(1)$. Then $ S_{n-1}=S_{n}+2e=0 $ by Lemma \ref{lem:maximal-BONG-even}(i). By Proposition \ref{prop:Ralphaproperty3}(i), we have $ R_{n-1} \ge 0 $ and $ R_{n}\ge  -2e $. If $M$ represents $N$, then $ -2e\le R_{n}\le R_{n-1}+R_{n}\le S_{n-1}+S_{n}=-2e $ by \cite[Lemma 4.6(i)]{beli_representations_2006}. So $R_{n}=-2e$ and hence $ R_{i-1}=R_{i}+2e=0 $ for $i\in [1,n]^{E}$, by Proposition \ref{prop:Ralphaproperty3}(iii).
		
		 If $ R_{n+1}>0 $, then $ R_{n+1}-S_{n}>2e$. By \cite[Corollary 2.10]{beli_representations_2019}, we have $ a_{1,n}b_{1,n}\in  F^{\times 2} $ and thus $a_{1,n}=b_{1,n}$ in $F^{\times}/F^{\times 2}$. So $(-1)^{n/2}a_{1,n}=(-1)^{n/2}b_{1,n}=\Delta$ or $1$, i.e., $d((-1)^{n/2}a_{1,n})=2e$ or $\infty$, according as $N=N_{1}^{n}(\Delta)$ or $N=N_{1}^{n}(1)$.
		 
		(ii) The first statement is clear from (i). By Proposition \ref{prop:Ralphaproperty3}(i), we have $R_{n+1}\ge 0$. Assume $ R_{n+1}>0 $. If $ M $ represents $ N_{1}^{n}(1) $ and $ N_{1}^{n}(\Delta) $, then $ d((-1)^{n/2}a_{1,n})=\infty $ and $ d((-1)^{n/2}a_{1,n})=2e $ by (i). This is impossible.	 
		
		(iii) Let $N=N_{2}^{n}(\Delta)$, or $N_{2}^{n}(1)$, with $n\ge 4$. Then $S_{n-3}=S_{n-2}+2e=0$ and $ S_{n-1}=S_{n}+2e=1 $ by Lemma \ref{lem:maximal-BONG-even}(ii). Similar to (i), we have $ R_{i-1}=R_{i}+2e=0 $ for $ i\in [1,n-2]^{E} $. Applying \cite[Lemma 4.6(i)]{beli_representations_2006}, we see that
		\begin{align}
			-2e+R_{n-1}=R_{n-2}+R_{n-1}&\le S_{n-2}+S_{n-1}=1-2e,\label{ineq1}\\
				R_{n-1}+R_{n}&\le S_{n-1}+S_{n}=2-2e \label{ineq2}.
		\end{align}
			Hence $ R_{n-1}\in \{0,1\} $ by \eqref{ineq1}. If $ R_{n-1}=0 $, then $-2e\le R_{n}\le 2-2e $ by \eqref{eq:BONGs} and \eqref{ineq2}, so $ R_{n}\in \{-2e,2-2e\} $ by Corollary \ref{cor:R-R-odd}(i). If $ R_{n-1}=1$, then $ R_{n}=1-2e $ similarly.
			
		(iv) Assume $R_{n+1}>2$. If $N=N_{\nu}^{n}(\varepsilon)$ is any of the five lattices under consideration, then, by Lemma \ref{lem:maximal-BONG-even}(i) and (iii), we have $S_{n}\le 2-2e$, so $R_{n+1}-S_{n}>2e$. Then, same as in the proof of (ii), we get $d((-1)^{n/2}a_{1,n})=d((-1)^{n/2}b_{1,n})=d(\varepsilon)$. 
				
		Since $M$ represents $N_{1}^{n}(1)$, $N_{1}^{n}(\Delta)$ or $N_{2}^{n}(\Delta)$, we have $d((-1)^{n/2}a_{1,n})=\infty$ or $2e$. Since $M$ also represents $N_{1}^{n}(\kappa)$ or $N_{2}^{n}(\kappa)$, we have $d((-1)^{n/2}a_{1,n})=d(\kappa)=2e-1$, a contradiction.
\end{proof}
\begin{lem}\label{lem:determineRn+1}
		Suppose $m=n+1 $ and $ R_{i-1}=R_{i}+2e=0 $ for all $ i\in [1,n-2]^{E} $. Let $ N=N_{\nu}^{n}(\varepsilon\pi) $, with $ \nu\in \{1,2\} $ and $ \varepsilon\in \mathcal{U}$.
	 	\begin{enumerate}[itemindent=-0.5em,label=\rm (\roman*)]	
	 	\item If $ R_{n-1}=0 $, $ R_{n}\in \{-2e,2-2e\} $ and $ R_{n+1}\ge 2 $, then Theorem \ref{thm:beligeneral}(ii) fails at $ i=n$.
	 	
	 	\item If $ R_{n-1}=R_{n}+2e=1 $, then Theorem \ref{thm:beligeneral}(ii) fails at $ i=n-1 $.
\end{enumerate}
\end{lem}
\begin{proof}
		 By Lemma \ref{lem:maximal-BONG-even}(iii), we have $ S_{i}=S_{i+1}+2e=0 $ for $ i\in [1,n-2]^{O} $, $ S_{n-1}=0 $ and $ S_{n}=1 $.  
		
		(i) If $ R_{n-1}=0 $ and $ R_{n}\in \{-2e,2-2e\} $, then $ \ord (a_{1,n}b_{1,n}) $ is odd and thus $ d[a_{1,n}b_{1,n}]=0 $. Also, $ R_{n+1}-S_{n}\ge 2-1=1 $ and  $ d[-a_{1,n+1}b_{1,n-1}]\ge 0 $. Hence		
		 	\begin{align*}
		 	A_{n}=\min\{(R_{n+1}-S_{n})/2 +e,R_{n+1}-S_{n}+d[-a_{1,n+1}b_{1,n-1}]\}
		 	&\ge \min\{1/2+e,1\}\\
		 	&=1>0=d[a_{1,n}b_{1,n}].
		 \end{align*}
		 
		(ii) If $ R_{n-1}=1 $, then $ \ord (a_{1,n-1}b_{1,n-1}) $ is odd and so $ d[a_{1,n-1}b_{1,n-1}]=0 $. Since $ R_{n}-R_{n-1}=-2e $, by Proposition \ref{prop:alphaproperty}(iv), we have $ d[-a_{n-1,n}]\ge 2e $. Since $R_{n-2}=S_{n-2}=-2e$, by Proposition \ref{prop:Ralphaproperty3}(iii), we have $d[(-1)^{(n-2)/2}a_{1,n-2}]\ge 2e$ and $d[(-1)^{(n-2)/2}b_{1,n-2}]\ge 2e$. So, by the domination principle, we see that
		\begin{align*}
			 d[-a_{1,n}b_{1,n-2}]\ge 2e.
		\end{align*}
		 Also, $ R_{n}-S_{n-1}=(1-2e)-0=1-2e $, and 
		 \begin{align*}
		 	R_{n+1}-S_{n-2}+d[a_{1,n+1}b_{1,n-3}]\ge R_{n+1}-S_{n-2}\ge R_{n-1}-S_{n-2}=1-(-2e)=2e+1
		 \end{align*}
		 from \eqref{eq:GoodBONGs}. Hence
		 	\begin{align*}
		 		\begin{split}
		 				A_{n-1}&  =\min\{ (R_{n}-S_{n-1})/2 +e,R_{n}-S_{n-1}+d[-a_{1,n}b_{1,n-2}],\\ &R_{n}-S_{n-1}+R_{n+1}-S_{n-2}+d[a_{1,n+1}b_{1,n-3}]\}\\
		 				&\ge \min\{(1-2e)/2+e,(1-2e)+2e, (1-2e)+(2e+1)\}=1/2>0=d[a_{1,n-1}b_{1,n-1}].
		 		\end{split}
		\end{align*}
\end{proof}
\begin{proof}[Proof of Theorem \ref{thm:dyadicACDeven-n+1}]
		Sufficiency follows from Lemma \ref{lem:maximalrep}. We claim that $ R_{i-1}=R_{i}+2e=0 $ for $ i\in [1,n-2]^{E} $, $ R_{n-1}=0 $, $ R_{n}\in \{-2e,2-2e\} $ and $ R_{n+1}\in \{0,1\} $. By Lemma \ref{lem:spacerep-criterion-ADC}(i), $M$ represents either $N_{1}^{n}(\Delta)$ or $N_{2}^{n}(\Delta)$. In both cases, by Lemma \ref{lem:repN1N2}(i) and (iii), we have $R_{i-1}=R_{i}+2e=0$ for $i\in [1,n-2]^{E}$ and either $R_{n-1}=0$ and $R_{n}\in \{2-2e,-2e\}$ or $R_{n-1}=R_{n}+2e=1$.
		
		Next, take $\varepsilon\in \mathcal{U}$. By Lemma \ref{lem:spacerep-criterion-ADC}(i), $ M $ represents $N_{\nu}^{n}(\varepsilon\pi)$ for some $\nu\in \{1,2\}$. Hence $M$ and $N_{\nu}^{n}(\varepsilon\pi)$ satisfy the conditions (i)-(iv) of Theorem \ref{thm:beligeneral}. Then, by Lemma \ref{lem:determineRn+1}(i) and (ii), we cannot have $R_{n-1}=0$, $R_{n}\in \{2-2e,-2e\}$ and $R_{n+1}\ge 2$ or $R_{n-1}=R_{n}+2e=1$, because Theorem \ref{thm:beligeneral}(ii) would fail at $i=n$ or $n-1$. Hence we are left with the case when $R_{n-1}=0$, $R_{n}\in \{2-2e,-2e\}$ and $R_{n+1}\le 1$. The claim is proved.
		
		If $ R_{n}=-2e $, then $[a_{1},\ldots,a_{n}]\cong W_{1}^{n}(\eta)$ with $\eta\in \{1,\Delta\}$ by Proposition \ref{prop:Ralphaproperty3}(iv). For any $\varepsilon\in \mathcal{U}$, since $(\eta,\varepsilon)_{\mathfrak{p}}=1\not=-1$, by Lemma \ref{lem:B-beli}(ii), $[a_{1},\ldots,a_{n}]\cong W_{1}^{n}(\eta)\nrep W_{2}^{n+1}(\varepsilon) $. Hence $ FM\not\cong W_{2}^{n+1}(\varepsilon) $, so $ FM $ is isometric to one of the remaining three types: $W_{1}^{n+1}(\delta)$, $W_{1}^{n+1}(\delta\pi)$ and $W_{2}^{n+1}(\delta\pi)$, with $\delta\in \mathcal{U}$ (cf. Proposition \ref{prop:space}(i)). 
		
		Recall that $R_{n+1}\in \{0,1\}$. If $FM\cong W_{1}^{n+1}(\delta) $, then $\ord (a_{1,n+1})$ is even, so $R_{n+1}=0$ and hence $ M $ is $ \mathcal{O}_{F} $-maximal by Lemma \ref{lem:maximal-BONG-odd}(i); if $FM\cong W_{1}^{n+1}(\delta\pi)$ or $W_{2}^{n+1}(\delta\pi)$, then $\ord(a_{1,n+1})$ is odd, so $R_{n+1}=1$ and hence $M$ is $\mathcal{O}_{F}$-maximal by Lemma \ref{lem:maximal-BONG-odd}(iii). 
		
		If $R_{n}=2-2e$, let $FM\cong W_{\nu}^{n+1}(c)$. Assume that $W_{1}^{n}(\eta)\rep FM$ for some $\eta\in \{1,\Delta\}$. Then, by $n$-ADC-ness, $N_{1}^{n}(\eta)\rep M$, which, by Lemma  \ref{lem:repN1N2}(i), implies $R_{n}=-2e$. Contradiction. Hence for $\eta\in \{1,\Delta\}$, we have $W_{1}^{n}(\eta)\nrep FM\cong W_{\nu}^{n+1}(c)$, which, by Lemma \ref{lem:repN1N2}(ii), is equivalent to $(1,c)_{\mathfrak{p}}=(\Delta,c)_{\mathfrak{p}}=-(-1)^{1+\nu}$, i.e., $1=(\Delta,c)_{\mathfrak{p}}=(-1)^{\nu}$.  But this happens precisely when $\nu=2$ and $c=\delta$ for some $\delta\in \mathcal{U}$. Thus $FM\cong W_{2}^{n+1}(\delta)$. Recall that $R_{n+1}\in \{0,1\}$ from the claim. Hence $ R_{n+1}=0 $ by the parity of $ \ord(a_{1,n+1}) $ and so $ M $ is $ \mathcal{O}_{F} $-maximal by Lemma \ref{lem:maximal-BONG-odd}(ii).
\end{proof}
\begin{lem}\label{lem:repN4N3-1}
		Suppose $ m=n+2 $, $ R_{i-1}=R_{i}+2e=0 $ for all $ i\in [1,n]^{E} $, $  R_{n+2} \ge 2-2e $ and $d[-a_{n+1,n+2}]>1-R_{n+2}$. 
	 	\begin{enumerate}[itemindent=-0.5em,label=\rm (\roman*)]	
		\item   For $ n\ge 2 $, if $ R_{n+1} $ is even or $ d((-1)^{n/2}a_{1,n})=2e $, then Theorem \ref{thm:beligeneral}(iii) fails at $ i=n+1 $ for $M$ and $  N_{2}^{n}(\Delta) $.
		
		\item  For $ n\ge 4 $, if $ R_{n+1} $ is even or $ d((-1)^{n/2}a_{1,n})=\infty $, then Theorem \ref{thm:beligeneral}(iii) fails at $ i=n+1 $ for $M$ and $ N_{2}^{n}(1) $. 
    \end{enumerate}
\end{lem}
\begin{proof}
	Let $ N=N_{2}^{n}(\eta)$, with $ \eta\in \{1,\Delta\} $. Then $S_{n-1}=S_{n}+2e=1 $, by Lemma \ref{lem:maximal-BONG-even}(ii). Thus $ R_{n+2}\ge  2-2e>S_{n} $. 
		 
	Since $ S_{n}-S_{n-1}=-2e $, by Proposition \ref{prop:alphaproperty}(iv), we have $ d[-b_{n-1,n}]\ge 2e $. Since $ R_{n}=S_{n-2}=-2e$, by Proposition \ref{prop:Ralphaproperty3}(iii), we have $ d[(-1)^{n/2}a_{1,n}]\ge 2e $ and $ d[(-1)^{(n-2)/2}b_{1,n-2}]\ge 2e $. Hence $
			d[a_{1,n}b_{1,n}]\ge 2e>1-R_{n+2}$
		by the domination principle. Combining with the assumption $ d[-a_{n+1,n+2}]>1-R_{n+2} $, we deduce that
		\begin{align*} 
			d[-a_{1,n+2}b_{1,n}]>1-R_{n+2}
		\end{align*}
	by the domination principle again. This, combined with $ d[-a_{1,n+1}b_{1,n-1}]\ge 0 $, shows that
		\begin{align*}
		d[-a_{1,n+1}b_{1,n-1}]+d[-a_{1,n+2}b_{1,n}]>0+(1-R_{n+2})&=2e+(1-2e)-R_{n+2}\\
		&=2e+S_{n}-R_{n+2}.
	\end{align*}

Now it remains to show that $ [a_{1},\ldots,a_{n+1}] $ fails to represent $ [b_{1},\ldots,b_{n}]\cong FN $, which, under hypothesis (i) (resp. (ii)), is isometric to $W_{2}^{n}(\Delta)$ (resp. $W_{2}^{n}(1)$). Equivalently, by Lemma \ref{lem:spacerep-criterion}(i), we must show that $[a_{1},\ldots, a_{n+1}]$ represents $W_{1}^{n}(\Delta)$ (resp. $W_{1}^{n}(1)$).

If $R_{n+1}=\ord (a_{n+1})$ is even, then, by Proposition \ref{prop:Ralphaproperty3}(v), $[a_{1},\ldots,a_{n+1}]\cong \mathbb{H}^{n/2}\perp [\varepsilon]=W_{1}^{n+1}(\varepsilon)$ for some $\varepsilon\in \mathcal{O}_{F}^{\times}$. For $\eta\in \{1,\Delta\}$, since $(\eta,\varepsilon)_{\mathfrak{p}}=1$, by Lemma \ref{lem:B-beli}(ii), $W_{1}^{n}(\eta)\rep W_{1}^{n+1}(\varepsilon)$, as required.

If $d((-1)^{n/2}a_{1,n})=2e$, then, by Proposition \ref{prop:Ralphaproperty3}(iv), $W_{1}^{n}(\Delta)\cong [a_{1},\ldots,a_{n}]\rep [a_{1},\ldots,a_{n+1}]$, so (i) holds. And if $d((-1)^{n/2}a_{1,n})=\infty$, then, by Proposition \ref{prop:Ralphaproperty3}(iv) again, $W_{1}^{n}(1)\cong [a_{1},\ldots,a_{n}]\rep[a_{1},\ldots,a_{n+1}]$, so (ii) holds. 
\end{proof}
\begin{lem}\label{lem:repN4N3-2}
			Suppose  $ m=n+2 $, $ R_{i-1}=R_{i}+2e=0 $ for all $ i\in [1,n]^{E} $ and $  R_{n+2} \ge 2-2e $. 
		 \begin{enumerate}[itemindent=-0.5em,label=\rm (\roman*)]	
			\item   Suppose that either $ R_{n+1} $ is even or $ d((-1)^{n/2}a_{1,n})=2e $. If $ M $ represents $ N_{2}^{n}(\Delta) $, then either $ \alpha_{n+1}=0 $, or $ \alpha_{n+1}=1$ and $ d(-a_{n+1}a_{n+2})=d[-a_{n+1,n+2}]=1-R_{n+2} $.
			
		 \item  Suppose that $ n\ge 4 $ and either $ R_{n+1} $ is even, or $ d((-1)^{n/2}a_{1,n})=\infty$. If $ M $ represents $ N_{2}^{n}(1) $, then either $ \alpha_{n+1}=0 $, or $ \alpha_{n+1}=1$ and $ d(-a_{n+1}a_{n+2})=d[-a_{n+1,n+2}]=1-R_{n+2} $.		
		\end{enumerate}			
\end{lem}
\begin{proof}
				(i) Assume $ d[-a_{n+1,n+2}]>1-R_{n+2} $. Then Theorem \ref{thm:beligeneral}(iii) fails at $ i=n+1 $ for $ N=N_{2}^{n}(\Delta) $ by Lemma \ref{lem:repN4N3-1}(i). But this contradicts the fact that $ M $ represents $ N_{2}^{n}(\Delta) $. Thus $ d[-a_{n+1,n+2}]\le 1-R_{n+2} $. By Proposition \ref{prop:Ralphaproperty3}(i), we have $ R_{n+1}\ge 0 $. Hence, by \eqref{eq:alpha-defn}, we deduce that
				\begin{align*}
					\alpha_{n+1}\le R_{n+2}-R_{n+1}+d[-a_{n+1,n+2}]\le R_{n+2}+d[-a_{n+1,n+2}]\le 1,
				\end{align*}
				which implies that $ \alpha_{n+1}\in \{0,1\} $ by Proposition \ref{prop:alphaproperty}(i), and $ d[-a_{n+1,n+2}]=1-R_{n+2} $ if $\alpha_{n+1}=1$. 
				
				 Since $ R_{n+1}-R_{n}\ge 2e $, by Proposition \ref{prop:Rproperty}(i) and the hypothesis that $R_{n+2}\ge 2-2e$, we have $ \alpha_{n}\ge 2e>1-R_{n+2}= d[-a_{n+1,n+2}]=\min\{d(-a_{n+1,n+2}),\alpha_{n}\}$. (We have  $n+2=m$, so $\alpha_{n+2}$ is ignored.) It follows that $d(-a_{n+1,n+2})=d[-a_{n+1,n+2}]=1-R_{n+2} $.
				
				(ii) Similar to (i).
\end{proof}
\begin{lem}\label{lem:n-ACDnecessityeven-1}
		 	Suppose that $ m=n+2 $ and $ M $ is $ n $-ADC.
		 	 \begin{enumerate}[itemindent=-0.5em,label=\rm (\roman*)]	
		 		\item  If $ FM\cong W_{1}^{n+2}(1) $, then $ M\cong N_{1}^{n+2}(1)  $. 
		 		
		 		\item  If $ n\ge 4 $ and $ FM\cong W_{1}^{n+2}(\Delta) $, then $ M\cong N_{1}^{n+2}(\Delta)$. 
		 		
		 		\item If $ FM\cong W_{2}^{n+2}(1) $, then $ M\cong N_{2}^{n+2}(1)$.
		 		
		 		\item If $ FM\cong W_{2}^{n+2}(\Delta) $, then $ M\cong N_{2}^{n+2}(\Delta) $.
		 		
		 		\item If $ c\in \mathcal{V} \backslash \{1,\Delta\} $ and $ FM\cong W_{1}^{n+2}(c)  $, then $ M\cong N_{1}^{n+2}(c)  $. 
		 		
		 		\item If $c\in \mathcal{V} \backslash \{1,\Delta\} $ and $ FM\cong W_{2}^{n+2}(c)  $, then $ M\cong N_{2}^{n+2}(c)  $.  
		 	\end{enumerate}			
\end{lem}
		 \begin{proof}
		 	(i) If $ n=2 $, then $ FM $ is $ 2 $-universal by \cite[Theorem 2.3]{hhx_indefinite_2021} and so $ M $ is $ 2 $-universal by $ 2 $-ADC-ness. So $ M\cong \mathbf{H}^{2}=N_{1}^{4}(1)$ by \cite[Remark 6.4]{HeHu2}. Suppose $ n\ge 4 $. Then,  by Lemma \ref{lem:spacerep-criterion-ADC}(ii), $ M $ represents every $ N $ in $ \mathcal{M}_{n}$ with $ N\not\cong N_{2}^{n}(1)$. Since $ M $ represents $ N_{1}^{n}(1) $ and $ N_{1}^{n}(\Delta) $, 	by Lemma \ref{lem:repN1N2}(ii), we have
		 	\begin{align}\label{inv:R1-Rn+1}
		 		R_{i}=0\quad\text{for $ i\in [1,n+1]^{O} $}\quad\text{and}\quad R_{i}=-2e\quad\text{for $ i\in [1,n]^{E} $}.
		 	\end{align}
		  	If $ R_{n+2}=R_{n+2}-R_{n+1}\ge 2-2e $, then $ \alpha_{n+1}\not=0 $, by Proposition \ref{prop:alphaproperty}(ii). Since $R_{n+1}=0$ is even and $ M $ represents $ N_{2}^{n}(\Delta) $, we have $ \alpha_{n+1}=1 $ and $ d[-a_{n+1,n+2}]=1-R_{n+2}$ by Lemma \ref{lem:repN4N3-2}(i). Since $R_{n}=-2e$, we also have $ d[(-1)^{n/2}a_{1,n}]\ge 2e $ by Proposition \ref{prop:Ralphaproperty3}(iii). So
		 	\begin{align*}
		 		d((-1)^{(n+2)/2}a_{1,n+2})=d[(-1)^{(n+2)/2}a_{1,n+2}]=1-R_{n+2}<2e
		 	\end{align*}
		 	by the domination principle. However, $ FM\cong W_{1}^{n+2}(1)  $ and so $ d((-1)^{(n+2)/2}a_{1,n+2})=\infty $, a contradiction. Hence $ R_{n+2}-R_{n+1}<2-2e $.
		 	
		 	Note that $ R_{n+2}-R_{n+1}\not=1-2e $ by Corollary \ref{cor:R-R-odd}(i). Hence $ R_{n+2}=R_{n+2}-R_{n+1}=-2e $ by \eqref{eq:BONGs}. Combining with \eqref{inv:R1-Rn+1}, we conclude that $ N\cong N_{1}^{n+2}(1) $ by Lemma \ref{lem:maximal-BONG-even}(i).
		 	
		 	(ii) If $ n\ge 4 $, then $ N_{2}^{n}(1) $ is defined. By Lemma \ref{lem:spacerep-criterion-ADC}(ii), $ M $ represents every $ N$ in $ \mathcal{M}_{n} $ with $ N\not\cong N_{2}^{n}(\Delta) $. In particular, it represents $ N_{1}^{n}(1) $, $ N_{1}^{n}(\Delta) $ and $ N_{2}^{n}(1) $. We repeat the reasoning from (i), but in this case we use Lemma \ref{lem:repN4N3-2}(ii) instead of Lemma \ref{lem:repN4N3-2}(i). Again we see that $ M $ satisfies \eqref{inv:R1-Rn+1} and $ R_{n+2}=-2e $. Since $ FM\cong W_{1}^{n+2}(\Delta) $, we deduce that $ N\cong N_{1}^{n+2}(\Delta) $ by Lemma \ref{lem:maximal-BONG-even}(i).
		 	
		 	(iii)(iv)  First, $ M $ represents every $ N $ in $ \mathcal{M}_{n} $ with $ N\not\cong N_{1}^{n}(1)$ (resp. $ N\not\cong N_{1}^{n}(\Delta)$) by Lemma \ref{lem:spacerep-criterion-ADC}(ii). Since $ M $ represents $ N_{1}^{n}(\Delta) $ (resp. $N_{1}^{n}(1)$) and $ N_{1}^{n}(\kappa) $, we see that
		 	\begin{align}\label{inv:R1-Rn}
		 		R_{i}=0\quad\text{for $ i\in [1,n]^{O} $}\quad\text{and}\quad  R_{i}=-2e\quad\text{for $ i\in [1,n]^{E} $}
		 	\end{align}
		 	by Lemma \ref{lem:repN1N2}(i) and  $ R_{n+1}\in \{0,1,2\} $ by Lemma \ref{lem:repN1N2}(iv). 
		 	
		 	By Proposition \ref{prop:Ralphaproperty3}(i), we have $ R_{n+2}\ge  -2e $. We assert $ R_{n+2}=1-2e $. 
		 	
		 	If $ R_{n+2}=-2e $, then $ FM\cong W_{1}^{n+2}(1) $ or $ W_{1}^{n+2}(\Delta) $ by Proposition \ref{prop:Ralphaproperty3}(iv). This contradicts $ FM\cong W_{2}^{n+2}(1) $ (resp. $FM\cong W_{2}^{n+2}(\Delta)$). 
		 	
		 	If $ R_{n+2}\ge 2-2e $, by Lemma \ref{lem:repN1N2}(i), we have either $ R_{n+1}\in \{0,2\} $, or 
		 	\begin{align*}
		 		 \quad R_{n+1}=1\quad \text{and}\quad d((-1)^{n/2}a_{1,n})=2e 
		 	\end{align*}
		 	 (resp. $R_{n+1}=1$ and $d((-1)^{n/2}a_{1,n})=\infty $). Hence the hypothesis of Lemma \ref{lem:repN4N3-2}(i) (resp. Lemma \ref{lem:repN4N3-2}(ii)) is satisfied. Since $ M $ represents $ N_{2}^{n}(\Delta) $ (resp. $N_{2}^{n}(1)$ with $n\ge 4$), we see that either $ \alpha_{n+1}=0 $, or $ \alpha_{n+1}=1 $ and $ d(-a_{n+1}a_{n+2})=1-R_{n+2} $ by Lemma \ref{lem:repN4N3-2}(i) (resp. Lemma \ref{lem:repN4N3-2}(ii)).
		 	
		 	\textbf{Case I: $ \alpha_{n+1}=0 $}
		 	
		 	By Proposition \ref{prop:alphaproperty}(ii), $R_{n+2}-R_{n+1}=-2e$. Since $R_{n+1}\le 2$ and, by our assumption, $R_{n+2}\ge 2-2e$, we must have $ R_{n+1}=2$ and $ R_{n+2}=2-2e $. This combined with \eqref{inv:R1-Rn} shows that for every $i\in [1,n+1]^{O}$, we have $R_{i+1}-R_{i}=-2e$ and $R_{i}$ is even. So, by Corollary \ref{cor:R-R-odd}(ii), $[a_{i},a_{i+1}]\cong \mathbb{H}$ or $[1,-\Delta]$. It follows that $ FM\cong  \mathbb{H}^{n/2}\perp [1,-\eta]=W_{1}^{n+2}(\eta) $ for $\eta=1$ or $\Delta$. This contradicts $ FM\cong W_{2}^{n+2}(1) $ (resp. $FM\cong W_{2}^{n+2}(\Delta)$).

		 	\textbf{Case II: $ \alpha_{n+1}=1 $}
		 	
		 	Since $ R_{n}=-2e $, we have $ d((-1)^{n/2}a_{1,n})\ge 2e $ by Proposition \ref{prop:Ralphaproperty3}(iii). Hence
		 	\begin{align*}
		 		d((-1)^{(n+2)/2}a_{1,n+2})=d(-a_{n+1}a_{n+2})=1-R_{n+2}<2e
		 	\end{align*}
		 	by the domination principle. This contradicts $ FM\cong W_{2}^{n+2}(1) $ (resp. $FM\cong W_{2}^{n+2}(\Delta)$) again. 
		 	
		 	With above discussion, the assertion is proved and thus $ R_{n+2}=1-2e $. 
		 	
		 	Recall that $ R_{n+1}\in \{0,1,2\} $ and so $ R_{n+1}=1 $ by Corollary \ref{cor:R-R-odd}(i). Combining with \eqref{inv:R1-Rn}, we deduce that $ M\cong N_{2}^{n+2}(1) $ (resp. $N_{2}^{n+2}(\Delta)$) by Lemma \ref{lem:maximal-BONG-even}(ii).
		 	
			(v) Let $ c\in \mathcal{V} \backslash \{1,\Delta\} $. By Lemma \ref{lem:spacerep-criterion-ADC}(ii), $ M $ represents every $ N $ in $ \mathcal{M}_{n} $ with $ N\not\cong N_{2}^{n}(c) $. In particular, $ M $ represents $ N_{1}^{n}(1) $ and $ N_{1}^{n}(\Delta) $, so it satisfies \eqref{inv:R1-Rn+1} by Lemma \ref{lem:repN1N2}(ii).

		 	 By Proposition \ref{prop:Ralphaproperty3}(i), we have $ R_{n+2}\ge  -2e $. If $ R_{n+2}=-2e $, then $ FM\cong W_{1}^{n+2}(1) $ or $ W_{1}^{n+2}(\Delta) $ by Proposition \ref{prop:Ralphaproperty3}(iv), which contradicts $ FM\cong W_{1}^{n+2}(c) $. Thus $ R_{n+2}>-2e $. Since $ R_{n+1}=0 $, Corollary \ref{cor:R-R-odd}(i) implies $ R_{n+2}\not=1-2e $. Hence $ R_{n+2}\ge 2-2e $ and so $ \alpha_{n+1}\not=0 $ by Proposition \ref{prop:alphaproperty}(ii).
		 	 
		 	 Now, we see that $R_{n+1}=0$ is even, $ M $ represents $ N_{2}^{n}(\Delta) $ and $\alpha_{n+1}\not=0$, so	$1-R_{n+2}=d(-a_{n+1}a_{n+2})$ by Lemma \ref{lem:repN4N3-2}(i). Since $R_{n}=-2e$, we also have $d((-1)^{n/2}a_{1,n})\ge d[(-1)^{n/2}a_{1,n}]\ge 2e$ by Proposition \ref{prop:Ralphaproperty3}(iii). On the other hand, $FM\cong W_{1}^{n+1}(c)$, so in $F^{\times}/F^{\times 2}$ we have $a_{1,n+2}=\det FM=(-1)^{(n+2)/2}c$. It follows that
		 	 \begin{align*}
		 	 	d((-1)^{(n+2)/2}a_{1,n+2})=d(c)<2e=d((-1)^{n/2}a_{1,n}).
		 	 \end{align*}
		 	   By the domination principle, this implies  $ 1-R_{n+2}=d(-a_{n+1}a_{n+2})=d(c) $. Combining with \eqref{inv:R1-Rn+1}, we conclude $ M\cong N_{1}^{n+2}(c) $ by Lemma \ref{lem:maximal-BONG-even}(iii).
		 	
		 	(vi) Similar to (v).
		 \end{proof}
\begin{lem}\label{lem:repN5N6-n=2-2}
	 Suppose $ m=4 $ and $ R_{1}=R_{3}=R_{2}+2e=0  $. If $ FM\cong W_{1}^{4}(\Delta) $ and $ M $ represents both $ N_{1}^{2}(\kappa) $ and $ N_{2}^{2}(\kappa) $, then $ R_{4}\in \{-2e,2-2e\} $.
\end{lem}
\begin{proof}
	Let $ N=N_{\nu}^{2}(\kappa) $, $ \nu\in \{1,2\} $. Then $ S_{1}=0 $ and $ S_{2}=2-2e $ by Lemma \ref{lem:maximal-BONG-even}(iii). Suppose $ R_{4}>2-2e $. Then $ R_{4}>S_{2}  $. Since $ R_{4}-R_{3}=R_{4}>2-2e>-2e $ and $ S_{2}-S_{1}=2-2e $, we have $ \alpha_{3}\ge 1=\beta_{1} $ 
	by Proposition \ref{prop:alphaproperty}(ii) and (iii). Since $\ord (a_{1,3}b_{1})$ is even, we also have $ d(-a_{1,3}b_{1})\ge 1 $. Combining these, we see that
	\begin{align*}
		d[-a_{1,3}b_{1}]=\min\{d(-a_{1,3}b_{1}),\alpha_{3},\beta_{1}\}=1.
	\end{align*}
	Also, $ d[-a_{1,4}b_{1,2}]=d(-a_{1,4}b_{1,2})=d(\Delta\kappa)=d(\kappa)=2e-1$ by the domination principle. So
	\begin{align*}
		d[-a_{1,3}b_{1}]+d[-a_{1,4}b_{1,2}]=1+(2e-1)>2e+(2-2e)-R_{4}=2e+S_{2}-R_{4}.
	\end{align*}
	 By definition, $ [b_{1},b_{2}]=FN\cong W_{1}^{2}(\kappa)  $ or $ W_{2}^{2}(\kappa) $. But, by Lemma \ref{lem:spacerep-criterion}(i), $ [a_{1},a_{2},a_{3}] $ represents exactly one of $ W_{1}^{2}(\kappa) $ and $ W_{2}^{2}(\kappa) $. Hence Theorem \ref{thm:beligeneral}(iii) fails at $ i=3 $ for either $ N=N_{1}^{2}(\kappa) $ or $N_{2}^{2}(\kappa)$. This contradicts the hypothesis that $ M $ represents both $ N_{1}^{2}(\kappa) $ and $ N_{2}^{2}(\kappa) $. Hence $ R_{4}\le 2-2e $.
	
	By Proposition \ref{prop:Ralphaproperty3}(i), we have $ R_{4}\ge -2e $. Recall that $R_{3}=0$ and so $R_{4}\not=1-2e$ by Corollary \ref{cor:R-R-odd}(i). Hence $ R_{4}\in \{-2e,2-2e\}$.
\end{proof} 
\begin{lem}\label{lem:McongN-2-ADC-rank-4}
	If $FM\cong W_{1}^{4}(\Delta)$ and $R_{1}=R_{3}=R_{2}+2e=R_{4}+2e-2=0$, then $M\cong \mathbf{H}\perp \prec 1,-\Delta\pi^{2-2e}\succ$.
\end{lem}
\begin{proof}
	Let $N=\mathbf{H}\perp \prec 1,-\Delta\pi^{2-2e}\succ$. By \cite[Lemma 3.10]{HeHu2}, $N\cong   \prec 1,-\pi^{-2e}, 1,-\Delta\pi^{2-2e}\succ$, $S_{1}=S_{3}=0$, $S_{2}=-2e$ and $S_{4}=2-2e$.
	
	To show $M\cong N$, we only need to verify that conditions (i)-(iv) in \cite[Theorem 3.2]{beli_representations_2006} are satisfied. We have $R_{2}-R_{1}=-2e$, $R_{3}-R_{2}=2e$ and $R_{4}-R_{3}=2-2e$. Hence $(\alpha_{1},\alpha_{2},\alpha_{3})=(0,2e,1)$ by Proposition \ref{prop:Rproperty}(ii). Since $R_{i}=S_{i}$ for $1\le i\le 4$, we have $(\beta_{1},\beta_{2},\beta_{3})=(0,2e,1)$ similarly. Hence conditions (i) and (ii) hold. For $i=1,3$, since $\ord(a_{1,i}b_{1,i}) $ is even, we have $d(a_{1,i}b_{1,i})\ge 1\ge \alpha_{i}$. Since $R_{2}-R_{1}=-2e$, we have $d(-a_{1}a_{2})\ge 2e$ by Corollary \ref{cor:R-R-odd}(ii). Similarly, $d(-b_{1}b_{2})\ge 2e$. Hence $d(a_{1,2}b_{1,2})\ge 2e\ge  \alpha_{2}$ by the domination principle. Thus condition (iii) is checked. Since $\alpha_{1}+\alpha_{2}=2e$ and $\alpha_{2}+\alpha_{3}=2e+1$, we only need to show that $[b_{1},b_{2}]\rep [a_{1},a_{2},a_{3}]$ for condition (iv). By definition, $[b_{1},b_{2}]\cong W_{1}^{2}(1) $. By Proposition \ref{prop:Ralphaproperty3}(v), $[a_{1},a_{2},a_{3}]\cong W_{1}^{3}(\varepsilon)$ for some $\varepsilon\in \mathcal{U}$. Hence $[b_{1},b_{2}]\rep [a_{1},a_{2},a_{3}]$ by Lemma \ref{lem:B-beli}(ii).
\end{proof}
\begin{lem}\label{lem:n-ACDnecessityeven-2}
	 Suppose that $ M $ is $ 2 $-ADC of rank $ 4 $. If $ FM\cong W_{1}^{4}(\Delta) $, then $ M\cong N_{1}^{4}(\Delta)$  or $  \mathbf{H}\perp \prec 1,-\Delta\pi^{2-2e}\succ  $. 
\end{lem}
\begin{proof}
    By Lemma \ref{lem:spacerep-criterion-ADC}(ii), $ M $ represents every $ N $  in $ \mathcal{M}_{2} $ with $ N\not\cong N_{2}^{2}(\Delta) $. Since $ M $ represents $ N_{1}^{2}(1) $ and $ N_{1}^{2}(\Delta) $, we have $ R_{1}=R_{3}=R_{2}+2e=0 $ by Lemma \ref{lem:repN1N2}(ii). Since $ M $ represents $ N_{1}^{2}(\kappa) $ and $ N_{2}^{2}(\kappa) $, we also have $R_{4}\in \{-2e,2-2e\}$ by Lemma \ref{lem:repN5N6-n=2-2}. If $ R_{4}=-2e $, then $ M\cong N_{1}^{4}(\Delta) $ by Lemma \ref{lem:maximal-BONG-even}(i). If $ R_{4}=2-2e $, then  $ M\cong \mathbf{H} \perp \prec 1,-\Delta\pi^{2-2e}\succ $ by Lemma \ref{lem:McongN-2-ADC-rank-4}.
\end{proof}
\begin{lem}\label{lem:2-ACDsufficiencyeven}
	 Let $ M\cong \mathbf{H}\perp \prec 1,-\Delta\pi^{2-2e}\succ $. Then
	 
	  \begin{enumerate}[itemindent=-0.5em,label=\rm (\roman*)]	
	 \item  $ M $ is $ 2 $-ADC, but not $\mathcal{O}_{F}$-maximal. 
	 
	 \item  $ M $ is not $ 3 $-ADC.
	\end{enumerate}
	\end{lem}
    \begin{proof}
    	 (i) We have $ FM\cong W_{1}^{4}(\Delta) $ and $R_{4}(M)=2-2e$. Hence $M$ is not $\mathcal{O}_{F}$-maximal from Lemma \ref{lem:maximal-BONG-even}(i).
    	 
    	 By Proposition \ref{prop:space}(iii), $FM$ represents $ FN $ for every $ N $ in $ \mathcal{M}_{2} $ with $ N\not\cong N_{2}^{2}(\Delta) $. So, by Lemma \ref{lem:maximalequivADC}, it suffices to show that $ M $ represents all $ N $ in $ \mathcal{M}_{2} $ except for $ N\cong N_{2}^{2}(\Delta) $. To do so, we will verify conditions (i)-(iv) in Theorem \ref{thm:beligeneral} for those $N$. Note that their invariants $S_{i}$ are clear from Lemma \ref{lem:maximal-BONG-even}.
    	
    	Let $\nu\in \{1,2\}$, $\eta\in \{1,\Delta\}$ and $c\in \mathcal{V} \backslash \{1,\Delta\}$. Then $d(c)<2e$. For condition (i), we have $ R_{1}=0\le S_{1} $ and $ R_{2}=-2e\le S_{2} $ for every $ N $ in $ \mathcal{M}_{2} $. Since $ S_{1}+2e\ge 2e>2-2e=R_{4} $, condition (iv) is verified.
    	
    	To verify condition (ii), for every $ N $ in $ \mathcal{M}_{2} $, we have
    	\begin{align*}
    		A_{1}\le \dfrac{R_{2}-S_{1}}{2}+e=\dfrac{-2e-S_{1}}{2}+e=\dfrac{-S_{1}}{2}\le 0\le d[a_{1}b_{1}]. 
    	\end{align*}
    Thus condition (ii) holds at $i=1$ for these $N$. 
    
        For $  N=N_{1}^{2}(\eta)$, since $ R_{2}=S_{2}=-2e $, by Proposition \ref{prop:Ralphaproperty3}(iii), we have $ d[-a_{1,2}]\ge 2e $ and $ d[-b_{1,2}]\ge 2e $. So $ d[a_{1,2}b_{1,2}]\ge 2e $ by the domination principle. Hence 
       \begin{align*}
       	A_{2}\le \dfrac{R_{3}-S_{2}}{2}+e=\dfrac{0-(-2e)}{2}+e=2e\le  d[a_{1,2}b_{1,2}].
       \end{align*}
   For $ N=N_{\nu}^{2}(c) $, by the domination principle, we have $ d[a_{1,2}b_{1,2}]=d[-b_{1,2}]=d(-b_{1}b_{2})=d(c)<2e $. Since $ S_{1}=0 $ and $ S_{2}=1-d(c) $, \eqref{eq:alpha-defn} gives $ \beta_{1}=1$. Hence
   \begin{align*}
   	A_{2}\le R_{3}-S_{2}+d[-a_{1,3}b_{1}]\le R_{3}-S_{2}+\beta_{1}=0-(1-d(c))+1=d(c)=d[a_{1,2}b_{1,2}].
   \end{align*}
 Hence condition (ii) also holds at $i=2$ for every $N$ in $\mathcal{M}_{2}$. Thus condition (ii) is verified.
   
To verify condition (iii), we have $ R_{3}=0\le S_{1} $ for every $ N $ in $ \mathcal{M}_{2} $. Thus condition (iii) holds at $i=2$ for these $N$. 
   
For $  N=N_{1}^{2}(\eta)$, we have $ [b_{1},b_{2}]\cong W_{1}^{2}(1) $ or $ W_{1}^{2}(\Delta) $. Also, $ [a_{1},a_{2},a_{3}]\cong W_{1}^{3}(\varepsilon) $ for some $ \varepsilon\in \mathcal{U} $ by Proposition \ref{prop:Ralphaproperty3}(v). Hence $[b_{1},b_{2}] \rep [a_{1},a_{2},a_{3}] $ by Lemma \ref{lem:B-beli}(ii). For $ N=N_{\nu}^{2}(c) $, in $F^{\times}/F^{\times 2}$ we have $a_{1,4}=\det FM=\Delta$ and $b_{1,2}=\det FN=-c$. Since $d(\Delta)=2e>d(c)$, by the domination principle, we have $  d(-a_{1,4}b_{1,2})=d(\Delta c)=d(c)=1-S_{2} $. Hence
\begin{align*}
   	d[-a_{1,3}b_{1}]+d[-a_{1,4}b_{1,2}]\le \beta_{1}+d(-a_{1,4}b_{1,2})= 1+(1-S_{2})\le 2e+S_{2}-R_{4},
 \end{align*}
   where the last inequality holds from $   2-2e\le S_{2} $ and $ R_{4}=2-2e $. Hence condition (iii) also holds at $i=3$ for every $N$ in $\mathcal{M}_{2}$. Thus condition (iii) is verified.
   
  (ii) Suppose that $ M $ is $ 3 $-ADC. Let $\varepsilon\in \mathcal{U}$. By Lemma \ref{lem:spacerep-criterion-ADC}(i), $ M $ represents $ N_{\nu}^{3}(\varepsilon\pi) $ for some $\nu\in\{1,2\}$. Then $ S_{1}=S_{2}+2e=0 $ and $ S_{3}=1$ by Lemma \ref{lem:maximal-BONG-odd}(iii). Since $ \ord(a_{1,3}b_{1,3}) $ is odd, $ d[a_{1,3}b_{1,3}]=0 $. By definition, we have $d(a_{1,4})=d(\Delta)=2e$. Since $S_{2}-S_{1}=-2e$, we also have $d(-b_{1}b_{2})\ge 2e$ by Corollary \ref{cor:R-R-odd}. Hence $ d(-a_{1,4}b_{1,2})\ge 2e $ by the domination principle. Since $S_{3}-S_{2}=2e+1$, Proposition \ref{prop:Rproperty}(ii) implies $ \beta_{2}=2e+1/2 $. So $ d[-a_{1,4}b_{1,2}]=\min\{d(-a_{1,4}b_{1,2}),\beta_{2}\}\ge 2e $. Also, $ R_{4}-S_{3}=(2-2e)-1=1-2e $. Hence
   \begin{align*}
   	A_{3}&=\min\{(R_{4}-S_{3})/2+e,R_{4}-S_{3}+d[-a_{1,4}b_{1,2}]\}\\
   	&\ge \min\{(1-2e)/2+e,(1-2e)+2e\}=1/2>0=d[a_{1,3}b_{1,3}].
   \end{align*}
   Thus Theorem \ref{thm:beligeneral}(ii) fails at $ i=3 $, which contradicts the fact that $M$ represents $N$.
    \end{proof}
 \begin{proof}[Proof of Theorem \ref{thm:dyadicACDeven}]
 	 Sufficiency follows by Lemmas \ref{lem:maximalrep} and \ref{lem:2-ACDsufficiencyeven}(i). Suppose that $ M $ is $ n $-ADC. Then, by Proposition \ref{prop:space}(ii), $ FM\cong W_{\nu}^{n}(c) $ for some $ \nu\in \{1,2\} $ and $ c\in \mathcal{V} $. So, by Lemmas \ref{lem:n-ACDnecessityeven-1} and \ref{lem:n-ACDnecessityeven-2}, $ M\cong N_{\nu}^{n}(c) $ or $ \mathbf{H}\perp \prec 1,-\Delta\pi^{2-2e}\succ $. Also, $ \prec 1,-\Delta\pi^{2-2e}\succ\cong 2^{-1}\pi A(2\pi^{-1},2\rho\pi) $ by \cite[Corollary 3.4(iii)]{beli_integral_2003} and \cite[93:17 Example]{omeara_quadratic_1963}.
 \end{proof}
	\section{$ n $-ADC lattices over dyadic local fields II}\label{sec:n-ADCdyadicfields-odd} 
	
	In this section, we keep the setting as the previous section, but let $ n $ be an odd integer with $ n\ge 3 $.	 
		 \begin{thm}\label{thm:dyadicACDodd-n+1}
		If $\rank\, M=n+1$, then $ M $ is $ n $-ADC if and only if $ M $ is $ \mathcal{O}_{F} $-maximal.		
	\end{thm}
 \begin{proof}
 	Sufficiency is clear from Lemma \ref{lem:maximalrep}. Suppose that $ M $ is $ n $-ADC. Then it is $ (n-1) $-ADC. Since $ n-1 $ is even, $ M $ is $ \mathcal{O}_{F} $-maximal except for $ n-1=2 $ and $ M\cong \mathbf{H}\perp \prec 1,-\Delta\pi^{2-2e} \succ$ by Theorem \ref{thm:dyadicACDeven}. However, $  \mathbf{H}\perp \prec 1,-\Delta\pi^{2-2e} \succ $ is not $ 3 $-ADC by Lemma \ref{lem:2-ACDsufficiencyeven}(ii). So the exceptional case cannot happen. 
 \end{proof}
  \begin{thm}\label{thm:dyadicACDodd-2}
	If $\rank\, M=n+2$, then $ M $ is $ n $-ADC if and only if either $ M $ is $ \mathcal{O}_{F} $-maximal, or
	\begin{align*}
		M\cong   N_{\nu}^{n+1}(\delta)\perp \langle \varepsilon \pi^{k}\rangle, 
	\end{align*}
	 with $\nu\in\{1,2\}$, $\delta\in \mathcal{U}\backslash \{1,\Delta\}$, $ \varepsilon\in \mathcal{U} $ and $ k\in \{0,1\} $.
	 
	 Also, if $M$ is simultaneously $\mathcal{O}_{F}$-maximal and isometric to the described orthogonal splitting, then $
	 	M\cong N_{2}^{n+2}(\varepsilon)$
	 with $\varepsilon\in \mathcal{U}$.
\end{thm}
 \begin{re}\label{re:dyadicACDodd-2}
 	For the lattice $N_{\nu}^{n+1}(\delta)$ given in Theorem \ref{thm:dyadicACDodd-2}, we see from Lemma \ref{lem:maximallattices-dyadic} and \cite[Remark 3.8, Lemma 3.9]{HeHu2} that 
 	\begin{align*}
 		N_{1}^{n+1}(\delta)&=\mathbf{H}^{(n-1)/2}\perp   N_{1}^{2}(\delta)
 		 \cong \; \mathbf{H}^{(n-1)/2}\perp \pi^{-l}A(\pi^{l},-(\delta-1)\pi^{-l}) \quad\text{and} \\
   N_{2}^{n+1}(\delta)&=\mathbf{H}^{(n-1)/2}\perp   N_{2}^{2}(\delta)
   \cong \; \mathbf{H}^{(n-1)/2}\perp \delta^{\#}\pi^{-l}A(\pi^{l},-(\delta-1)\pi^{-l}),
 	\end{align*}
  with $\delta\in \mathcal{U}\backslash \{1,\Delta\}$ and $2l=d(\delta)-1\le 2e-2$, where $\delta^{\#}=1+4\rho(\delta-1)^{-1}$. Similarly, we also see that
  \begin{align*}
  	N_{2}^{n+2}(\varepsilon)=\mathbf{H}^{(n-1)/2}\perp N_{2}^{3}(\varepsilon)\cong \mathbf{H}^{(n-1)/2}\perp 2^{-1}\pi A(2,2\rho)\perp \langle \Delta \varepsilon\rangle,
  \end{align*}
  with $\varepsilon\in \mathcal{U}$.
 \end{re}

 Before showing Theorem \ref{thm:dyadicACDodd-2}, we first prove the following theorem, which characterizes the $n$-ADC lattices with odd $n$. In the remainder of this section, we assume $ \rank\,M=n+2 $. 
\begin{thm}\label{thm:dyadicACDodd-R-alpha-invariant}
    	$ M $ is $ n $-ADC if and only if $R_{i}=0$ for $ i\in [1,n]^{O} $, $ R_{i}=-2e $ for $ i\in [1,n]^{E} $, $R_{n+1}\in [-2e,0]^{E}$ and $R_{n+2},\alpha_{n}\in \{0,1\}$.  
    \end{thm}
     \begin{proof}
    	We will show that the theorem is equivalent to Lemma \ref{lem:dyadicACDodd} below.

    	For necessity, by Proposition \ref{prop:alphaproperty}(ii), $\alpha_{n}=0$ if and only if $R_{n+1}=-2e<0$. Hence the conditions follows from Lemma \ref{lem:dyadicACDodd}(i)(ii)(iv).
    	
    	For sufficiency, from the hypothesis, we have $R_{n+1} \ge -2e$ and $R_{n+2}\le 1 $. It follows that $R_{n+2}-R_{n+1}\le 2e+1$, and the equality holds if and only if $R_{n+1}=-2e$ and $R_{n+2}=1$. This shows Lemma \ref{lem:dyadicACDodd}(iii). If $\alpha_{n}=1$, then $R_{n+1}=R_{n+1}-R_{n}\in [2-2e,0]^{E}\cup \{1\}$ by Proposition \ref{prop:alphaproperty}(iii), but $R_{n+1}\le 0$ and so $R_{n+1}\in [2-2e,0]^{E}$. Hence Lemma \ref{lem:dyadicACDodd}(i), (ii) and (iv) follow from the hypothesis except for the condition $R_{n+1}+d[-a_{n,n+1}]=1$.
    	
    	Since $\alpha_{n}=1$, by Proposition \ref{prop:alphaproperty}(v), we see that $d[-a_{n,n+1}]\ge 1-R_{n+1}$ and the equality holds when $R_{n+1}\not=2-2e$. Assume $R_{n+1}=2-2e$. Since $R_{n+2}-R_{n+1}\le 1-(2-2e)=2e-1$, Proposition \ref{prop:Rproperty}(i) implies that $d[-a_{n,n+1}]\le \alpha_{n+1}\le 2e-1=1-R_{n+1}$. Hence $d[-a_{n,n+1}]=1-R_{n+1}$, as desired.
    \end{proof}
      	\begin{lem}\label{lem:dyadicACDodd}
  	$ M $ is $ n $-ADC if and only if the following conditions hold:
  	 \begin{enumerate}[itemindent=-0.5em,label=\rm (\roman*)]	
  		\item  $R_{i}=0$ for $ i\in [1,n]^{O} $ and $ R_{i}=-2e $ for $ i\in [1,n]^{E} $. 
  		
  		\item Either $ \alpha_{n}=0 $ or $ \alpha_{n}=R_{n+1}+d[-a_{n,n+1}]=1 $.
  		
  		\item If $ R_{n+2}-R_{n+1}>2e $, then $ R_{n+1}=-2e $ and $ R_{n+2}=1 $.
  		
  		\item  If $ \alpha_{n}=1 $, then $ R_{n+1}\in [2-2e,0]^{E} $ and $ R_{n+2}\in \{0,1\} $.
  	\end{enumerate} 
  \end{lem}  
To establish Lemma \ref{lem:dyadicACDodd}, we need a series of lemmas. First, we review the invariants $S_{i}=R_{i}(N)$ from Proposition \ref{prop:maximalproperty} for $N$ in $\mathcal{M}_{n}$. Precisely, we have
\begin{equation}\label{maximal-con-Si}
	\begin{split}
		&S_{i}=0\quad\text{for $i\in [1,n-2]^{O}$},\quad S_{i}=-2e\quad\text{for $i\in [1,n-2]^{O}$},\\
		&S_{n-1}\in \{-2e,2-2e\}\quad\text{and}\quad S_{n}\in \{0,1\},
	\end{split}
\end{equation}
which will be repeatedly used for the argument in Lemmas \ref{lem:invariantD}, \ref{lem:sufficiency-B1B2B4} and \ref{lem:sufficiency-B3}.
\begin{lem}\label{lem:invariantD}
	Suppose that $ R_{i}=0$ for $ i\in [1,n]^{O} $ and $R_{i}=-2e $ for $ i\in [1,n]^{E} $. For any $ N $ in $ \mathcal{M}_{n} $, the following statements hold:
	 \begin{enumerate}[itemindent=-0.5em,label=\rm (\roman*)]	
		\item  $ d[a_{1,i}b_{1,i}]\ge 2e $ for $ i\in [1,n-2]^{E} $.	
		
		\item If $ S_{n-1}=-2e $, then $ d[a_{1,n-1}b_{1,n-1}]\ge 2e $; if $ S_{n-1}=2-2e$, then $ d[a_{1,n-1}b_{1,n-1}]=2e-1 $.
		
		\item If $ \alpha_{n}=1 $, then $ d[a_{1,n}b_{1,n}]=1-S_{n} $.
		
		\item   If $ S_{n-1}=-2e $, then $ d[-a_{1,n}b_{1,n-2}]=0 $; if $ S_{n-1}=2-2e $, then $ d[-a_{1,n}b_{1,n-2}]\le 1 $.
		
		\item If $ \alpha_{n}=R_{n+1}+d[-a_{n,n+1}]=1 $, then $ d[-a_{1,n+1}b_{1,n-1}]=1-R_{n+1}$.
	\end{enumerate} 
 
\end{lem}
\begin{proof}
(i) For $ i\in [1,n-2]^{E} $, since $ R_{i}=S_{i}=-2e $, Proposition \ref{prop:Ralphaproperty3}(iii) implies that $ d[(-1)^{i/2}a_{1,i}]\ge 2e  $ and $ d[(-1)^{i/2}b_{1,i}]\ge 2e $. Hence $ d[a_{1,i}b_{1,i}]\ge 2e $ by the domination principle.
	
(ii) Since $ R_{n-1}=S_{n-3}=-2e $, by Proposition \ref{prop:Ralphaproperty3}(iii), we have $ d[(-1)^{(n-1)/2}a_{1,n-1}]\ge 2e $ and $d[(-1)^{(n-3)/2}b_{1,n-3}] \ge 2e $. If $ S_{n-1}=-2e $, then $ d[-b_{n-2,n-1}]\ge 2e $ by Proposition \ref{prop:maximalproperty}(ii). If $ S_{n-1}=2-2e $, then $d[-b_{n-2,n-1}]=2e-1 $ by Proposition \ref{prop:maximalproperty}(iii). Hence 
	\begin{align*} 
		d[a_{1,n-1}b_{1,n-1}] 
		\begin{cases}
		 \ge 2e   &\text{if $ S_{n-1}=-2e $},\\
		 =2e-1  &\text{if $ S_{n-1}=2-2e $},
		\end{cases}
	\end{align*}
	by the domination principle.
	
(iii) First, $ \ord(a_{1,n}) $ is even from hypothesis and $ \ord (b_{1,n-1}) $ is also even from \eqref{maximal-con-Si}. If $ S_{n}=1 $, then $ d[a_{1,n}b_{1,n}]=d(a_{1,n}b_{1,n})=0 $; if $ S_{n}=0 $, then $ d(a_{1,n}b_{1,n})\ge 1=\alpha_{n} $, so $d[a_{1,n}b_{1,n}]=\min\{d(a_{1,n}b_{1,n}),\alpha_{n}\}=1$. In both cases, $ d[a_{1,n}b_{1,n}]=1-S_{n} $. 
	
(iv) If $S_{n-1}=-2e$, then $\beta_{n-2}=0$, by Proposition \ref{prop:maximalproperty}(ii); if $S_{n-1}=2-2e$, then $\beta_{n-2}=1$, by Proposition \ref{prop:maximalproperty}(iii). Note that $ 	0\le d[-a_{1,n}b_{1,n-2}]\le \beta_{n-2} $ and we are done.
	
(v) If $ \alpha_{n}=R_{n+1}+d[-a_{n,n+1}]=1 $, then $ R_{n+1}=R_{n+1}-R_{n}\ge 2-2e $, by Proposition \ref{prop:alphaproperty}(iii).  
	
By (ii), we have $
		  d[a_{1,n-1}b_{1,n-1}]\ge 2e-1\ge 1-R_{n+1}$. Moreover, the first inequality is strict unless $S_{n-1}=2-2e$ and the second is strict unless $R_{n+1}=2-2e$. Therefore, 
	\begin{align*}
		 d[a_{1,n-1}b_{1,n-1}]>1-R_{n+1}=d[-a_{n,n+1}],
	\end{align*} 
	 unless $S_{n-1}=R_{n+1}=2-2e$. Hence, by the domination principle, $d[-a_{1,n+1}b_{1,n-1}]=1-R_{n+1}$ holds except for  $S_{n-1}=R_{n+1}=2-2e$.
	
	In the exceptional case $S_{n-1}=R_{n+1}=2-2e$, we have
	\begin{align*}
		d[a_{1,n-1}b_{1,n-1}]\ge 2e-1=1-R_{n+1}=d[-a_{n,n+1}],
	\end{align*}
	so $d[-a_{1,n+1}b_{1,n-1}]\ge 2e-1$, by the domination principle.
	But, by Proposition \ref{prop:maximalproperty}(iii), we have $d[-a_{1,n+1}b_{1,n-1}]\le \beta_{n-1}=2e-1$. Hence $d[-a_{1,n+1}b_{1,n-1}]=2e-1=1-R_{n+1}$.
\end{proof}
\begin{lem}\label{lem:sufficiency-B1B2B4}
	Suppose that $R_{i}=0$ for $ i\in [1,n]^{O} $ and $ R_{i}=-2e $ for $ i\in [1,n]^{E} $. Then
	 \begin{enumerate}[itemindent=-0.5em,label=\rm (\roman*)]	
	\item   Theorem \ref{thm:beligeneral}(i) holds for every $N$ in $\mathcal{M}_{n}$.
	
	\item   If $\alpha_{n}=0$ or $ \alpha_{n}=R_{n+1}+d[-a_{n,n+1}]=1 $, then Theorem \ref{thm:beligeneral}(ii) holds for every $N$ in $\mathcal{M}_{n}$.
	
	\item  If $\alpha_{n}\in \{0,1\}$ and $R_{n+2}-R_{n+1}\le 2e$, then Theorem \ref{thm:beligeneral}(iv) holds for every $N$ in $\mathcal{M}_{n}$.
\end{enumerate}
\end{lem}
\begin{proof}
		(i) By Proposition \ref{prop:Ralphaproperty3}(i), if $i$ is odd, then $R_{i}=0\le S_{i}$, and if $i$ is even, then $R_{i}=-2e\le S_{i}$. Hence Theorem \ref{thm:beligeneral}(i) holds for $1\le i\le n$.
 
 (ii) For $ i\in [1,n-2]^{O} $, note that $ S_{i}=R_{i+1}+2e=0 $ and so
 \begin{align*}
 	A_{i}\le \dfrac{R_{i+1}-S_{i}}{2}+e=\dfrac{-2e-0}{2}+e=0\le d[a_{1,i}b_{1,i}].
 \end{align*}
 For $ i\in [1,n-2]^{E} $, since $ R_{i+1}=S_{i}+2e=0 $, we have
 \begin{align*}
 	A_{i}\le \dfrac{R_{i+1}-S_{i}}{2}+e=\dfrac{0-(-2e)}{2}+e=2e\le d[a_{1,i}b_{1,i}]
 \end{align*}
 by Lemma \ref{lem:invariantD}(i).
 For $ i=n-1 $, by Lemma \ref{lem:invariantD}(ii), we have
 \begin{align*} 
 	d[a_{1,n-1}b_{1,n-1}] 
 	\begin{cases}
 		\ge 2e   &\text{if $ S_{n-1}=-2e $},\\
 		=2e-1  &\text{if $ S_{n-1}=2-2e $}.
 	\end{cases}
 \end{align*}
 By Lemma \ref{lem:invariantD}(iv), we also have
 \begin{align*}
 	d[-a_{1,n}b_{1,n-2}]   
 	\begin{cases}
 		=0   &\text{if $ S_{n-1}=-2e $},\\
 		\le 1  &\text{if $ S_{n-1}=2-2e $}.
 	\end{cases}
 \end{align*}
 So
 \begin{align*}
 	A_{n-1}&\le R_{n}-S_{n-1}+d[-a_{1,n}b_{1,n-2}]\\&
 	\begin{cases}
 		=0-(-2e)+0=2e\le d[a_{1,n-1}b_{1,n-1}]   &\text{if $ S_{n-1}=-2e $},\\
 		\le 0-(2-2e)+1=2e-1=d[a_{1,n-1}b_{1,n-1}]  &\text{if $ S_{n-1}=2-2e $}.
 	\end{cases}
 \end{align*}
  For $ i=n $, if $ \alpha_{n}=0 $, then $ R_{n+1}=-2e $ by Proposition \ref{prop:alphaproperty}(ii) and so 
  \begin{align*}
  	A_{n}\le \dfrac{R_{n+1}-S_{n}}{2}+e=\dfrac{-S_{n}}{2}\le 0\le d[a_{1,n}b_{1,n}].
  \end{align*}
   If $ \alpha_{n}=R_{n+1}+d[-a_{n,n+1}]=1 $, then $
 d[-a_{1,n+1}b_{1,n-1}]=1-R_{n+1}$
 by Lemma \ref{lem:invariantD}(v). Also, $ d[a_{1,n}b_{1,n}]=1-S_{n} $ by Lemma \ref{lem:invariantD}(iii). So 
 \begin{align*}
 	A_{n}\le R_{n+1}-S_{n}+d[-a_{1,n+1}b_{1,n-1}]= R_{n+1}-S_{n}+(1-R_{n+1})=1-S_{n}= d[a_{1,n}b_{1,n}].
 \end{align*}
 Hence Theorem \ref{thm:beligeneral}(ii) holds for $ 1\le i\le n $.
 
 (iii) Since $\alpha_{n}\le 1$, Proposition \ref{prop:Rproperty}(i) implies $R_{n+1}-R_{n}<2e$. Combining with the hypothesis, for every $2\le i\le n$, we have $
 R_{i+2}-R_{i+1}\le 2e$, so Theorem \ref{thm:beligeneral}(iv) holds.
\end{proof}
\begin{lem}\label{lem:sufficiency-B3}
	Suppose that $R_{i}=0$ for $ i\in [1,n]^{O} $, $ R_{i}=-2e $ for $ i\in [1,n]^{E} $ and $R_{n+2}-R_{n+1}\le 2e$. If either $\alpha_{n}=0$, or $\alpha_{n}=R_{n+1}+d[-a_{n,n+1}]=1$, $R_{n+1}\in [2-2e,0]^{E} $ and $R_{n+2}\in \{0,1\}$, then Theorem \ref{thm:beligeneral}(iii) holds for every $N$ in $\mathcal{M}_{n}$.
\end{lem}
\begin{proof}
By \eqref{maximal-con-Si}, we have $ R_{i+1}=0=S_{i-1} $ for $ i\in [2,n-1]^{E} $ and $ R_{i+1}=-2e=S_{i-1} $ for $ i\in [2,n-1]^{O} $, Theorem \ref{thm:beligeneral}(iii) holds trivially for $ 2\le i\le n-1 $. 
	
For $ i=n $, if $ \alpha_{n}=0 $, then $ R_{n+1}=-2e\le S_{n-1} $. If $ \alpha_{n}=R_{n+1}+d[-a_{n,n+1}]=1 $, when $  S_{n-1}=2-2e $, we have
	\begin{align*}
		d[-a_{1,n}b_{1,n-2}]+d[-a_{1,n+1}b_{1,n-1}]\le 1+(1-R_{n+1})= 2e+S_{n-1}-R_{n+1}
	\end{align*}
	by Lemma \ref{lem:invariantD}(iv) and (v); when $ S_{n-1}=-2e$, note that $ [b_{1},\ldots ,b_{n-1}]\cong W_{1}^{n-1}(1) $ or $ W_{1}^{n-1}(\Delta) $ from Proposition \ref{prop:Ralphaproperty3}(iv), and $ [a_{1},\ldots,a_{n}]\cong W_{1}^{n}(\varepsilon) $ with $ \varepsilon\in \mathcal{O}_{F}^{\times} $ from Proposition \ref{prop:Ralphaproperty3}(v). Hence $ [b_{1},\ldots, b_{n-1}]\rep [a_{1},\ldots,a_{n}] $ by Lemma \ref{lem:B-beli}(ii).

	For $ i=n+1 $, we may assume $R_{n+2}>S_{n}\ge 0$. If $R_{n+1}=-2e$, then, by hypothesis, $R_{n+2}\le R_{n+1}+2e=0 $, a contradiction. Hence $R_{n+1}\not=-2e$, i.e., $\alpha_{n}\not=0$. So $\alpha_{n}=R_{n+1}+d[-a_{n,n+1}]=1$. Hence $ R_{n+1}\in [2-2e,0]^{E} $ and $ R_{n+2}=1 $. Now, we have $1=R_{n+2}>S_{n}\ge 0$, so $S_{n}=0$. It follows that $\ord(a_{1,n+2}b_{1,n}) $ is odd and so $ d[-a_{1,n+2}b_{1,n}]=0 $. Since $ R_{n+2} \le 2-2e $, we have
	\begin{align*}
		d[-a_{1,n+1}b_{1,n-1}]+d[-a_{1,n+2}b_{1,n}]= (1-R_{n+1})+0 \le 2e-1=2e+S_{n}-R_{n+2},
	\end{align*}
	by Lemma \ref{lem:invariantD}(v). Hence Theorem \ref{thm:beligeneral}(iii) holds for $ 2\le i\le n+1 $.
\end{proof}
Now, we are ready to show the sufficiency of Lemma \ref{lem:dyadicACDodd}.
\begin{proof}[Proof of Sufficiency of Lemma \ref{lem:dyadicACDodd}]
	If $ R_{n+2}-R_{n+1}>2e $, then $ R_{n+1}=-2e $ and $ R_{n+2}=1 $. Then $\det FM$ has an odd order and so $FM\cong W_{\nu}^{n+2}(\delta\pi)$ for some $\delta\in \mathcal{U}$ and $\nu\in \{1,2\}$. Then, by Lemma \ref{lem:maximal-BONG-odd}(iii), we have $ M\cong N_{\nu}^{n+2}(\delta\pi) $. So $ M $ is $ \mathcal{O}_{F} $-maximal and thus is $ n $-ADC by Lemma \ref{lem:maximalrep}. 
	
	Assume $ R_{n+2}-R_{n+1}\le  2e $. By Lemma \ref{lem:maximalequivADC}, it is sufficient to show that for every $ N $ in $ \mathcal{M}_{n} $, if $ FM $ represents $ FN $, then $ M $ represents $ N $. To do so, we need to verify that Theorem \ref{thm:beligeneral}(i)-(iv) hold for $ M$ and $N$. But this follows from Lemmas \ref{lem:sufficiency-B1B2B4} and \ref{lem:sufficiency-B3}.
\end{proof}
\begin{lem}\label{lem:n-ADC-implies-n-1-universal}
     If $M$ is $n$-ADC, then it is $(n-1)$-universal. 
\end{lem}
\begin{proof}
	Let $N$ be an $\mathcal{O}_{F}$-lattice of rank $n-1$. We take a non-zero element $c\in \mathcal{O}_{F}$ such that $c\not=-\det FM\det FN$, i.e., $c\det FN\not=-\det FM$. Define $N^{\prime}:=N\perp \langle c\rangle$. Then $N^{\prime}$ is integral and $\det FN^{\prime}=c\det FN\not=-\det FM$. Since $\dim FM-\dim FN^{\prime}=2$, it follows from \cite[63:21 Theorem]{omeara_quadratic_1963} that $FN^{\prime}\rep FM$. Since $M$ is $n$-ADC, we have $N^{\prime}\rep M$. Since also $N\rep N^{\prime}$, we have $N\rep M$. Thus $M$ is $(n-1)$-universal by the arbitrariness of $N$.
\end{proof}
 
In view of Lemma \ref{lem:n-ADC-implies-n-1-universal} and the classification for $ (n-1) $-universality in \cite[Theorem 4.1]{HeHu2}, we futher have the lemma.
\begin{lem}\label{lem:even-nuniversaldyadic}
	Suppose that $ M $ is $ n$-ACD. Then
  \begin{enumerate}[itemindent=-0.5em,label=\rm (\roman*)]	
	\item  $R_{i}=0$ for $ i\in [1,n]^{O} $ and $ R_{i}=-2e $ for $ i\in [1,n]^{E} $.
	
	\item   Either $ \alpha_{n}=0 $ or $ \alpha_{n}=R_{n+1}+d[-a_{n,n+1}]=1 $.
	
	\item   If $ R_{n+2}-R_{n+1}>2e $, then $ R_{n+1}=-2e $; and if moreover $ n\ge 5 $, or $ n=3 $ and $ d(a_{1,4})=2e $, then $ R_{n+2}=1 $.
\end{enumerate}
\end{lem}
 \begin{lem}\label{lem:repN3ON4On=3}
     Suppose $ n=3 $, $ d(a_{1,4})=\infty $, $ R_{1}=R_{3}=R_{2}+2e=R_{4}+2e=0$ and $ R_{5}>1 $. Then Theorem \ref{thm:beligeneral}(iii) fails at $ i=4 $ for all $N=N_{2}^{3}(c)$ with $ c\in \mathcal{V}$.
\end{lem}
\begin{proof}
	  Since $R_{4}-R_{3}=-2e$, we have $d[-a_{3,4}]\ge 2e$ by Proposition \ref{prop:alphaproperty}(iv).
	
	We have $c=\varepsilon$ or $\varepsilon\pi$ for some $\varepsilon\in \mathcal{U}$. For  $ N=N_{2}^{3}(\varepsilon) $, we have $S_{1}=0$, $S_{2}=2-2e$ and $S_{3}=0$ by Lemma \ref{lem:maximal-BONG-odd}(ii), so $d[a_{1,2}b_{1,2}]=2e-1$ by Lemma \ref{lem:invariantD}(ii). For $ N=N_{2}^{3}(\varepsilon\pi) $, we have $S_{1}=0$, $S_{2}=-2e$ and $S_{3}=1$ by Lemma \ref{lem:maximal-BONG-odd}(iii), so $d[a_{1,2}b_{1,2}]\ge 2e$ by Lemma \ref{lem:invariantD}(ii). Since also $d[-a_{3,4}]\ge 2e$, by the domination principle, we have $d[-a_{1,4}b_{1,2}]=2e-1$ or $\ge 2e$, according as $S_{3}=0$ or $1$. So, in both cases, $d[-a_{1,4}b_{1,2}]\ge 2e-1+S_{3}$. 
  	 So we conclude that $ R_{5}>1\ge S_{3} $ and
  	\begin{align}\label{S3R5}
  		d[-a_{1,4}b_{1,2}]+d[-a_{1,5}b_{1,3}]\ge  (2e-1+S_{3})+0>2e+S_{3}-R_{5}.
  	\end{align}
	
	It remains to show that $ [a_{1},a_{2},a_{3},a_{4}] $ fails to represent   $ [b_{1},b_{2},b_{3}] $. Since $ a_{1,4}\in F^{\times 2} $, $ [a_{1},a_{2},a_{3},a_{4}]\cong W_{1}^{4}(1)=\mathbb{H}^{2} $ by Proposition \ref{prop:Ralphaproperty3}(iv). Also, $ [b_{1},b_{2},b_{3}]\cong W_{2}^{3}(c) $. Hence $[b_{1},b_{2},b_{3}]\nrep [a_{1},a_{2},a_{3},a_{4}]$ by Lemma \ref{lem:B-beli}(ii).
\end{proof}

If $ M $ is $ n $-ADC, then it is $ (n-1) $-universal by Lemma \ref{lem:n-ADC-implies-n-1-universal}(iii) and thus $ M $ satisfies the hypothesis of \cite[Lemma 5.8]{HeHu2} from Lemma \ref{lem:even-nuniversaldyadic}. Hence we have the following lemma.
\begin{lem}\label{lem:dsharp}
	Suppose that $ M $ is $ n $-ADC. If $ \alpha_{n}=1 $ and either $ R_{n+1}=1 $ or $ R_{n+2}>1 $, then 
	\begin{align*}
		 d((-1)^{(n+1)/2}a_{1,n+1})=1-R_{n+1}<2e,
	\end{align*}
	 $ ((-1)^{(n+1)/2}a_{1,n+1})^{\#} $ is a unit and $ d(((-1)^{(n+1)/2}a_{1,n+1})^{\#})=2e+R_{n+1}-1 $.
\end{lem}
\begin{lem}\label{lem:repN1ON2ON3ON4O}
 Suppose that $ M $ is $ n $-ADC and $ FM\cong W_{\nu}^{n+2}(c) $. Thus $ c=(-1)^{(n+1)/2}a_{1,n+2} $. Let $ \tilde{c}=(-1)^{(n+1)/2}a_{1,n+1} $ and let $ N=N_{\nu}^{n}(c)$ or $ N_{\nu}^{n}(c\tilde{c}^{\#}) $.
 
 If $ \alpha_{n}=1 $ and either $ R_{n+1}=1 $ or $ R_{n+2}>1 $, then 
	\begin{enumerate}[itemindent=-0.5em,label=\rm (\roman*)]	
	\item  $ R_{n+2}>S_{n} $ and 	$d[-a_{1,n+1}b_{1,n-1}]+d[-a_{1,n+2}b_{1,n}]>2e+S_{n}-R_{n+2} $.
	
	\item  $ [a_{1},\ldots,a_{n+1}] $ fails to represent $ FN=[b_{1},\ldots,b_{n}] $.
\end{enumerate}
	Thus, Theorem \ref{thm:beligeneral}(iii) fails at $ i=n+1 $.
\end{lem}
  \begin{proof}
  	(i) First, $ \ord (a_{1,n}) $ is even from Lemma \ref{lem:even-nuniversaldyadic}(i) and $ \tilde{c}^{\#} $ is a unit from Lemma \ref{lem:dsharp}. Hence $ \ord(c)=\ord (c\tilde{c}^{\#})\equiv R_{n+2}-R_{n+1}\pmod{2} $. Therefore, by Lemma \ref{lem:maximal-BONG-odd}, both when $ N=N_{\nu}^{n}(c) $ or $ N_{\nu}^{n}(c\tilde{c}^{\#}) $ we have
  	\begin{align*}
  		S_{n}=
  		\begin{cases}
  			1      &\text{if $ R_{n+2}-R_{n+1} $ is odd},\\
  			0  &\text{if $ R_{n+2}-R_{n+1} $ is even}.
  		\end{cases}
  	\end{align*}
  Note that $ R_{n+2}\ge 0 $ by Proposition \ref{prop:Ralphaproperty3}(i). If $R_{n+2}=0$, then $R_{n+1}=1$ by the hypothesis. This contradicts Corollary \ref{cor:R-R-odd}(i). Thus $ R_{n+2}\ge 1$. If $R_{n+2}=1$, then $R_{n+1}=1$ by the hypothesis. Then $R_{n+2}-R_{n+1}$ is even, so $S_{n}=0$ and thus $R_{n+2}=1>0=S_{n}$. If $R_{n+2}>1$, then $R_{n+2}>1\ge S_{n}$. So, in all cases, $R_{n+2}>S_{n}$.
  	
  	Secondly, from the hypothesis, we have $a_{1,n+2}=(-1)^{(n+1)/2}c$ and $b_{1,n}=(-1)^{(n-1)/2}c$ or  $(-1)^{(n-1)/2}c\tilde{c}^{\#}$. By Lemma \ref{lem:dsharp}, we also have $ d(\tilde{c}^{\#})=2e+R_{n+1}-1  $. Hence
  	\begin{align*}
  		d[-a_{1,n+2}b_{1,n}]=d(-a_{1,n+2}b_{1,n})=
  		\begin{cases}
  		d(c^{2})=\infty        &\text{if $ N=N_{\nu}^{n}(c)$},   \\
  		d(c^{2}\tilde{c}^{\#})=2e+R_{n+1}-1	&\text{if $ N=N_{\nu}^{n}(c\tilde{c}^{\#})$}. 
  		\end{cases}
  	\end{align*}
  	Also, by Lemma \ref{lem:invariantD}(v), $d[-a_{1,n+1}b_{1,n-1}]=1-R_{n+1}$. Thus
  	 \begin{align*}
  		d[-a_{1,n+1}b_{1,n-1}]+d[-a_{1,n+2}b_{1,n}] \ge (1-R_{n+1})+(2e+R_{n+1}-1)=2e>2e+S_{n}-R_{n+2}.
  	\end{align*}
  	(Recall that we have shown $R_{n+2}>S_{n}$.)
  	 
  	 (ii) Let $V=[a_{1},\ldots,a_{n+1}]$. Then $\det V=a_{1,n+1}=(-1)^{(n+1)/2}\tilde{c}$, so $V\cong W_{\nu^{\prime}}^{n+1}(\tilde{c})$, with $\nu^{\prime}\in \{1,2\}$. Assume that $V$ represents $[b_{1},\ldots, b_{n}]\cong FN$ for $N=N_{\nu}^{n}(c)$ and $N=N_{\nu}^{n}(c\tilde{c}^{\#})$, i.e., $W_{\nu^{\prime}}^{n+1}(\tilde{c})$ represents both $W_{\nu}^{n}(c)$ and $W_{\nu}^{n}(c\tilde{c}^{\#})$. Then, by Lemma \ref{lem:B-beli}(ii), we have 
  	 \begin{align*}
  	 	(c,\tilde{c})_{\mathfrak{p}}=(-1)^{\nu+\nu^{\prime}}=  (c\tilde{c}^{\#},\tilde{c})_{\mathfrak{p}}=(c,\tilde{c})_{\mathfrak{p}}(\tilde{c}^{\#},\tilde{c})_{\mathfrak{p}},
  	 \end{align*}
  	  which implies $(\tilde{c}^{\#},\tilde{c})_{\mathfrak{p}}=1$. This contradicts \eqref{csharp-2}.
 \end{proof}
 \begin{proof}[Proof of necessity of Lemma \ref{lem:dyadicACDodd}]
	Let $ FM\cong W_{\nu}^{n+2}(c) $, where $ \nu\in \{1,2\} $ and $ c\in \mathcal{V}$. Suppose that $M$ is $n$-ADC. Then (i) and (ii) coincide with Lemma \ref{lem:even-nuniversaldyadic}(i) and (ii). 

For (iii), assume that $R_{n+2}-R_{n+1}>2e$. Then, by Lemma \ref{lem:even-nuniversaldyadic}(iii), $R_{n+1}=-2e$. If, moreover, either $n\ge 5$ or $n=3$ and $d(a_{1,4})=2e$, then also $R_{n+2}=1$. So (iii) holds.

Suppose now that in the remaining case, $n=3$ and $d(a_{1,4})\not=2e$, and (iii) fails, i.e., $R_{5}\not=1$. Again, by Lemma \ref{lem:even-nuniversaldyadic}(i) and (iii), $R_{1}=R_{3}=R_{2}+2e=R_{4}+2e=0$. Since $d(a_{1,4})\not=2e$, Proposition \ref{prop:Ralphaproperty3}(iv) implies $d(a_{1,4})=\infty$. Also from $R_{5}-R_{4}>2e$ we see that $R_{5}>R_{4}+2e=0$, so $R_{5}\not=1$ implies $R_{5}>1$.

Let $c^{\prime}\in \mathcal{V}\backslash \{c\}$ and let $N=N_{2}^{3}(c^{\prime})$. Since $c^{\prime}=c$, we have $N\not\cong N_{3-\nu}^{3}(c)$. So, by Lemma \ref{lem:spacerep-criterion-ADC}(ii), $M$ represents $N$. But, by Lemma \ref{lem:repN3ON4On=3}, Theorem \ref{thm:beligeneral}(iii) fails for $M$ and $N$, so $M$ cannot represent $N$. Contradiction. Hence $ R_{5}=1 $ and (iii) is proved.
	
	For (iv), suppose $ \alpha_{n}=1 $ and either $ R_{n+1}=1$ or $ R_{n+2}>1 $. By Lemma \ref{lem:B-beli}(i), $ N_{\nu}^{n}(c)\not\cong N_{3-\nu}^{n}(c) $ and $ N_{\nu}^{n}(c\tilde{c}^{\#}) \not\cong N_{3-\nu}^{n}(c) $, so, by Lemma \ref{lem:spacerep-criterion-ADC}(ii), $ M $ represents $ N_{\nu}^{n}(c) $ and $ N_{\nu}^{n}(c\tilde{c}^{\#}) $. But, by Lemma \ref{lem:repN1ON2ON3ON4O}, Theorem \ref{thm:beligeneral}(iii) fails for either $ N=N_{\nu}^{n}(c) $ or $ N=N_{\nu}^{n}(c\tilde{c}^{\#}) $. Hence $R_{n+1}=R_{n+1}-R_{n}\not=1$. Since $\alpha_{n}=1$, Proposition \ref{prop:alphaproperty}(iii) implies $R_{n+1}\in [2-2e,0]^{E}$. Also, $R_{n+2}\le 1$, i.e., $R_{n+2}\in \{0,1\}$. Thus (iv) is proved.
\end{proof}

Unlike in the even case, there are many $n$-ADC lattices that are not $\mathcal{O}_{F}$-maximal when $n$ is odd. Thus, to complete the proof of Theorem \ref{thm:dyadicACDodd-2}, we need to determine the structures of $n$-ADC lattices explicitly, as presented in Lemma \ref{lem:n-ADC-explicitform}. 

First, recall from Theorem \ref{thm:dyadicACDodd-R-alpha-invariant} that $M$ is $n$-ADC of rank $n+2$ if and only if
\begin{equation}\label{equiv-con-odd-n-ADC}
	\begin{split}
&(a)\;R_{i}=0\;\text{for}\;i\in [1,n]^{O}\quad\text{and}\quad R_{i}=-2e\; \text{for}\; i\in [1,n]^{E}; \\
&(b)\;  R_{n+1}\in [-2e,0]^{E};\quad (c)\;\alpha_{n}\in \{0,1\};\quad (d)\; R_{n+2}\in \{0,1\}. 
	\end{split}
\end{equation}
\begin{lem}\label{lem:Rn+2}
	Let $\nu\in \{1,2\}$ and $\varepsilon\in \mathcal{U}$. Suppose that $M$ is $n$-ADC. 
	\begin{enumerate}[itemindent=-0.5em,label=\rm (\roman*)]	
	\item  If $FM\cong W_{\nu}^{n+2}(\varepsilon)$, then $R_{n+2}=0$.
	
	\item  If $FM\cong W_{\nu}^{n+2}(\varepsilon\pi)$, then $R_{n+2}=1$.
	\end{enumerate}
\end{lem}
\begin{proof}
	   Clearly, $ \ord(a_{1,n+2}) $ is even or odd, according as $FM\cong W_{\nu}^{n+2}(\varepsilon)$ or $W_{\nu}^{n+2}(\varepsilon\pi)$. By \eqref{equiv-con-odd-n-ADC}(a) and (b), we have $\ord(a_{1,n+2})\equiv \sum_{i=1}^{n+2} R_{i}\equiv  R_{n+2} \pmod{2}$. By (d), we further have
	 \begin{align*} 
	 	R_{n+2}=\begin{cases}
	 		0  &\text{if $\ord(a_{1,n+2})$ is even},\\
	 		1 &\text{if $\ord(a_{1,n+2})$ is odd},
	 	\end{cases}
	 \end{align*}
	   as desired.
\end{proof}
\begin{lem}\label{lem:ADC-isometry-odd-Rn+1}
	Let $M,M^{\prime}$ be two $n$-ADC $\mathcal{O}_{F}$-lattices of rank $n+2$. Then  $M\cong M^{\prime}$ if and only if $FM\cong FM^{\prime} $ and $R_{n+1}(M)=R_{n+1}(M^{\prime})$.
\end{lem}
\begin{proof}
	 We only need to show the sufficiency. Let $M\cong \prec a_{1},\ldots,a_{n+2}\succ$,  $R_{i}=R_{i}(M)$ and $\alpha_{i}=\alpha_{i}(M)$. Since $M$ is $n$-ADC, the conditions (a)-(d) in \eqref{equiv-con-odd-n-ADC} hold. Let $M^{\prime}\cong \prec b_{1},\ldots,b_{n+2}\succ$, $S_{i}=R_{i}(M^{\prime})$ and $\beta_{i}=\alpha_{i}(M^{\prime})$. The same conditions (a')-(d') hold for the corresponding invariants $S_{i}$ and $\beta_{i}$ of $M^{\prime}$. 
	  
	  By \eqref{equiv-con-odd-n-ADC}(a) and (a'), we have $R_{i}=S_{i}$ for $1\le i\le n$. By hypothesis, $R_{n+1}=S_{n+1}$. And, by Lemma \ref{lem:Rn+2}, $R_{n+2}=S_{n+2}=0$ or $1$, according as $FM\cong FM^{\prime}\cong W_{\nu}^{n+2}(\varepsilon)$ or $W_{\nu}^{n+2}(\varepsilon\pi)$ for some $\varepsilon\in \mathcal{U}$. Thus 
	   \begin{align}\label{Ri=Si}
	   	R_{i}=S_{i}
	   \end{align}
	    for $1\le i\le n+2$, i.e., the condition (i) of \cite[Theorem 3.1]{beli_Anew_2010} is fulfilled.
	   
	    Suppose $R_{n+1}=-2e$. If $FM\cong FM^{\prime}\cong W_{1}^{n+2}(\varepsilon)$ with $\varepsilon\in \mathcal{U}$, by Lemma \ref{lem:Rn+2}(i), we have $R_{n+2}=0$. Hence, by Lemma \ref{lem:maximal-BONG-odd}(i), $M\cong M^{\prime}\cong N_{1}^{n+2}(\varepsilon)$. If $FM\cong FM^{\prime}\cong W_{\nu}^{n+2}(\varepsilon\pi)$, with $\nu\in\{1,2\}$ and $\varepsilon\in \mathcal{U}$, by Lemma \ref{lem:Rn+2}(ii), we have $R_{n+2}=1$. Hence, by Lemma \ref{lem:maximal-BONG-odd}(iii), $M\cong M^{\prime}\cong N_{\nu}^{n+2}(\varepsilon\pi)$.  
	      
	     Now, assume that $R_{n+1}\not=-2e$, i.e., $\alpha_{n}\not=0$. By \eqref{Ri=Si}, $M$ and $M^{\prime}$ satisfy the condition (i) of \cite[Theorem 3.1]{beli_Anew_2010}, so we are left to verify that the conditions (ii)-(iv) are fulfilled. 	  
	   
	 By (c), we have $\alpha_{n}=1$. By (a) and Proposition \ref{prop:Rproperty}(i)(ii), we have
	  \begin{align}\label{alpha-i}
	  	\alpha_{i}=\begin{cases}
	  		0    &\text{if $i\in [1,n-1]^{O}$},  \\
	  		2e   &\text{if $i\in [1,n-1]^{E}$}.
	  	\end{cases}
	  \end{align}
	   If $R_{n+2}=1$, then $R_{n+2}-R_{n+1}$ is odd, so Proposition \ref{prop:Rproperty}(iii) implies $\alpha_{n+1}=R_{n+2}-R_{n+1}=1-R_{n+1}$; if $R_{n+2}=0$, then $R_{n}=R_{n+2}$, so $R_{n+1}+\alpha_{n+1}=R_{n}+\alpha_{n}=1$ by \cite[Corollary 2.3(i)]{beli_Anew_2010}, i.e., $\alpha_{n+1}=1-R_{n+1}$. Hence, in both cases, we have
	   \begin{align*}
	   	\alpha_{n+1}=1-R_{n+1}.
	   \end{align*}
	   The same argument combined with \eqref{Ri=Si} gives the values of $\beta_{i}$'s. Thus $\alpha_{i}=\beta_{i}$ for $1\le i\le n+1$ and so \cite[Theorem 3.1(ii)]{beli_Anew_2010} holds for $M$ and $M^{\prime}$.
	  
	 For $i\in [1,n]^{O}$, we have $\alpha_{i}\le 1$ ($\alpha_{i}=0$ for $i\in [1,n-1]^{O}$ and $\alpha_{n}=1$). Since $R_{i}=S_{i}$, $\ord (a_{1,i}b_{1,i})=\sum_{k=1}^{i} (R_{k}+S_{k})$ is even, so $d(a_{1,i}b_{1,i})\ge 1\ge \alpha_{i}$. For $i\in [1,n]^{E}$, since $R_{i}=-2e$, by Proposition \ref{prop:Ralphaproperty3}(iii), we have $d((-1)^{i/2}a_{1,i})\ge d[(-1)^{i/2}a_{1,i}]\ge 2e$. Similarly, $d((-1)^{i/2}b_{1,i}) \ge 2e$. Hence, by the domination principle, $d(a_{1,i}b_{1,i})\ge 2e=\alpha_{i}$. For $i=n+1$, by Proposition \ref{prop:alphaproperty}(v), we have $d(-a_{n}a_{n+1})\ge  d[-a_{n,n+1}]=1-R_{n+1}$. Since $d((-1)^{(n-1)/2}a_{1,n-1})\ge 2e$, by the domination principle, we see that 
	 \begin{align*}
	 	d((-1)^{(n+1)/2}a_{1,n+1})&\ge \min\{d((-1)^{(n-1)/2}a_{1,n-1}),d(-a_{n}a_{n+1})\}\\
	 	&\ge \min\{2e,1-R_{n+1}\}=1-R_{n+1}.
	 \end{align*}
	  Similarly, $d((-1)^{(n+1)/2}b_{1,n+1})\ge 1-R_{n+1}$. So, by the domination principle again, we conclude that $d(a_{1,n+1}b_{1,n+1})\ge 1-R_{n+1}=\alpha_{n+1}$. Thus \cite[Theorem 3.1(iii)]{beli_Anew_2010} holds for $M$ and $M^{\prime}$.
	 
	 By \eqref{alpha-i}, we have $\alpha_{i}+\alpha_{i+1}=2e$ for $1\le i\le n-2$. Recall that $\alpha_{n}=1$, $\alpha_{n+1}=1-R_{n+1}$ and $R_{n+1}\in [2-2e,0]^{E}$. We also have $\alpha_{n}+\alpha_{n+1}=1+(1-R_{n+1})=2-R_{n+1}\le 2e$. For $i=n-1$, since $\alpha_{n-1}+\alpha_{n}=2e+1>2e$, we need to prove that $[b_{1},\ldots,b_{n-1}]\rep [a_{1},\ldots,a_{n}]$. By Proposition \ref{prop:Ralphaproperty3}(iv) and (v), $[b_{1},\ldots,b_{n-1}]\cong W_{1}^{n-1}(\eta)$, with $\eta\in \{1,\Delta\}$, and $[a_{1},\ldots,a_{n}]\cong W_{1}^{n}(\delta)$ for some $\delta\in \mathcal{O}_{F}^{\times}$. Then $W_{1}^{n-1}(\eta)\rep W_{1}^{n}(\delta)$ follows from Lemma \ref{lem:B-beli}(ii). (Both when $\eta=1$ or $\Delta$, we have $(\eta,\delta)_{\mathfrak{p}}=1$.) Thus \cite[Theorem 3.1(iv)]{beli_Anew_2010} holds for $M$ and $M^{\prime}$.
\end{proof}
 \begin{defn}\label{defn:Wvr}
 	Let $\nu\in \{1,2\}$, $r\in \{0,\ldots,e\}$ and $c\in \mathcal{V}$. We denote by $M_{\nu,r}^{n+2}(c)$ the only $n$-ADC lattice $M$ with $FM\cong W_{\nu}^{n+2}(c)$ and $R_{n+1}(M)=-2r$, provided that such lattice exists.
 \end{defn}
 \begin{re}\label{re:Wvr}
 	By Lemma \ref{lem:ADC-isometry-odd-Rn+1}, such lattice is unique up to isometry, if it exists.
 	
 	If $M$ is $n$-ADC of rank $n+2$, from \eqref{equiv-con-odd-n-ADC}(b) we have $R_{n+1}(M)\in [-2e,0]^{E}$, i.e., $R_{n+1}(M)=-2r$ for some $0\le r\le e$. Hence $M\cong M_{\nu,r}^{n+2}(c)$, where $FM\cong W_{\nu}^{n+2}(c)$. Thus every $n$-ADC lattice of rank $n+2$ is isometric to $M_{\nu,r}^{n+2}(c)$ for some, $\nu\in \{1,2\}$, $r\in \{0,\ldots,e\}$ and $c\in \mathcal{V}$.
 \end{re}
 \begin{lem}\label{lem:undefined}
 	Suppose that $M$ is $n$-ADC. If $FM\cong W_{2}^{n+2}(\varepsilon)$ for some $\varepsilon\in \mathcal{U}$, then $R_{n+1}\not=-2e$. Equivalently, $M_{2,e}^{n+2}(\varepsilon)$ is not defined.
 \end{lem}
 \begin{proof}
 	Assume that $R_{n+1}=-2e$. Since $FM\cong W_{2}^{n+2}(\varepsilon)$, by Lemma \ref{lem:Rn+2}(i), $R_{n+2}=0$. Proposition \ref{prop:Ralphaproperty3}(v), with $j=n+1$, implies that $FM\cong \mathbb{H}^{(n+1)/2}\perp [\varepsilon^{\prime}]= W_{1}^{n+2}(\varepsilon^{\prime})$ for some $\varepsilon^{\prime}\in \mathcal{U}$, a contradiction.
 \end{proof}
  \begin{lem}\label{lem:C}
 	Let $\nu\in \{1,2\}$, $c\in\mathcal{V}$ and $\delta\in \mathcal{U}$, with $d(\delta)<2e$. Then $M=N_{\nu}^{n+1}(\delta)\perp \langle c\rangle$ is $n$-ADC and $R_{n+1}(M)=1-d(\delta)\in [2-2e,0]^{E}$.
 \end{lem}
 \begin{proof}
 	By Lemma \ref{lem:maximallattices-dyadic} and Remark \ref{re:invariant}, we have $N_{\nu}^{n+1}(\delta)\cong \prec a_{1},\ldots,a_{n+1}\succ$ relative to a good BONG, with $(a_{1},\ldots,a_{n-1})=(1,-\pi^{-2e},\ldots,1,-\pi^{-2e})$ and $(a_{n},a_{n+1})=(1,-\delta\pi^{1-d(\delta)})$ or $(\delta^{\#},-\delta^{\#}\delta\pi^{1-d(\delta)})$, according as $\nu=1$ or $2$. Put $R_{i}=R_{i}(N_{\nu}^{n+1}(\delta))$. Then, by Lemma \ref{lem:maximal-BONG-even}(iii), $R_{i}=0$ for $i\in [1,n]^{O}$, $R_{i}=-2e$ for $i\in [1,n]^{E}$ and $R_{n+1}=1-d(\delta)$. Since $c\in \mathcal{V}$, we have $\ord(c)\in \{0,1\}$. Hence if $a_{n+2}:=c$ and $R_{n+2}:=\ord(a_{n+2})$, then $R_{n+2}\in \{0,1\}$.
 	
 	Since $R_{n+2}\ge 0=R_{n}$ and $R_{n+2}\ge 0\ge 1-d(\delta)=R_{n+1}$, by \cite[Corollary 4.4(v)]{beli_integral_2003}, we have 
 	\begin{align*}
 		M\cong \prec a_{1},\ldots,a_{n+1}\succ\perp \prec a_{n+2}\succ\cong \prec a_{1},\ldots,a_{n+1},a_{n+2}\succ
 	\end{align*}
 	relative to a good BONG and $R_{i}(M)=R_{i}$. In particular, since $\delta\in \mathcal{U}\backslash\{1,\Delta\}$, we have $d(\delta)\in [1,2e-1]^{O}$, so $R_{n+1}(M)=1-d(\delta)\in [2-2e,0]^{E}$.
 	
 	Write $\alpha_{n}=\alpha_{n}(M)$. Since $R_{n+1}-R_{n}=R_{n+1}>-2e$, Proposition \ref{prop:alphaproperty}(ii) implies that $\alpha_{n}\ge 1$. On the other hand, for $\nu\in \{1,2\}$, in $F^{\times}/F^{\times 2}$ we have $-a_{n}a_{n+1}=\delta$, so 
 	\begin{align*}
 		\alpha_{n}\le R_{n+1}-R_{n}+d(-a_{n}a_{n+1})=(1-d(\delta))-0+d(\delta)=1.
 	\end{align*}
 	Hence $\alpha_{n}=1$.
 	
 	With above discussion, we have shown the conditions (a)-(d) in \eqref{equiv-con-odd-n-ADC}. By Theorem \ref{thm:dyadicACDodd-R-alpha-invariant}, $M$ is $n$-ADC.
 \end{proof}
Let $c\in F^{\times}$. For convenience, we also write $c=\mathcal{U}$ (resp. $c\not=\mathcal{U}$) for $c\in \mathcal{\mathcal{U}}$ (resp. $c\not\in \mathcal{U}$) temporarily.
 \begin{lem}\label{lem:n-ADC-explicitform}
 	Let $\nu\in \{1,2\}$, $r\in \{0,\ldots,e\}$ and $c\in \mathcal{V}$. Then $M_{\nu,r}^{n+2}(c)$ is defined except for $(\nu,r,c)=(2,e,\mathcal{U})$.
 	\begin{enumerate}[itemindent=-0.5em,label=\rm (\roman*)]	
 		\item 	If $r=e$ and $(\nu,c)\not=(2,\mathcal{U})$, then  $M_{\nu,e}^{n+2}(c)\cong N_{\nu}^{n+2}(c)$.
 		
 		 \item If $r=e-1$ and $(\nu,c)=(2,\mathcal{U})$, then $M_{2,e-1}^{n+2}(c)\cong N_{2}^{n+2}(c)$.
 		
 		\item \footnote{I am thankful to the referee for suggesting this improved version for the case $\alpha_{n}=1$, which refines the original version of  Theorem \ref{thm:dyadicACDodd-2}.}	If $0\le r\le e-1$, then 
 		$M_{\nu,r}^{n+2}(c)\cong N_{\nu^{\prime}}^{n+1}(\omega_{r})\perp \langle \omega_{r}c\rangle$, where $\omega_{r}\in \mathcal{U}$ is arbitrary such that $d(\omega_{r})=2r+1$ and $ \nu^{\prime}\in \{1,2\}$ satisfies $(-1)^{\nu^{\prime}}=(-1)^{\nu}(\omega_{r},c)_{\mathfrak{p}}$. 
 	\end{enumerate}
  \end{lem}
 \begin{proof}
 	First, by Lemma \ref{lem:undefined}, $M_{2,e}^{n+2}(c)$ is undefined for every $c\in \mathcal{U}$. Next, we will show the assertions (i)-(iii), thereby confirming that the lattice $M_{\nu,r}^{n+2}(c)$ is defined for $(\nu,r,c)\not=(2,e,\mathcal{U})$.
 	
 	For (i) and (ii), by Lemma \ref{lem:maximalrep}, $N_{\nu}^{n+2}(c)$ is $\mathcal{O}_{F}$-maximal and thus is $n$-ADC. We also have $FN_{\nu}^{n+2}(c)\cong W_{\nu}^{n+2}(c)$. If $(\nu,c)\not=(2,\mathcal{U})$, then, by Lemma \ref{lem:maximal-BONG-odd}(i) and (iii), $R_{n+1}(N_{\nu}^{n+2}(c))=-2e$. So, by Definition \ref{defn:Wvr}, $N_{\nu}^{n+2}(c)\cong M_{\nu,e}^{n+2}(c)$. If $(\nu,c)=(2,\mathcal{U})$, then, by Lemma \ref{lem:maximal-BONG-odd}(ii), $R_{n+1}(N_{2}^{n+2}(c))=2-2e$. So, by Definition \ref{defn:Wvr}, $N_{2}^{n+2}(c)\cong M_{2,e-1}^{n+2}(c)$. 
 	
 	For (iii), let $M=N_{\nu^{\prime}}^{n+1}(\omega_{r})\perp \langle \omega_{r}c\rangle$ and $0\le r\le e-1$. Since $(\omega_{r},c)_{\mathfrak{p}}=(-1)^{\nu+\nu^{\prime}}$, by Lemma \ref{lem:B-beli}(ii), we have $W_{\nu^{\prime}}^{n+1}(\omega_{r})\rep W_{\nu}^{n+2}(c)$. Since $\det W_{\nu^{\prime}}^{n+1}(\omega_{r})\det W_{\nu}^{n+2}(c)=\omega_{r}c$, we get $FM\cong W_{\nu^{\prime}}^{n+1}(\omega_{r})\perp [\omega_{r}c]\cong W_{\nu}^{n+2}(c)$. Also, by Lemma \ref{lem:C}, $M$ is $n$-ADC and $R_{n+1}(M)=1-d(\omega_{r})=-2r$. Then, by Definition \ref{defn:Wvr}, $M\cong M_{\nu,r}^{n+2}(c)$.
 \end{proof}
 \begin{cor}\label{cor:counting-odd}
 	 Up to isometry, there are $(8e+6)(N\mathfrak{p})^{e}$ $n$-ADC lattices of rank $n+2$ with odd $n\ge 3$, of which  $(8e-2)(N\mathfrak{p})^{e}$ are not $\mathcal{O}_{F}$-maximal.
 \end{cor}
 \begin{proof}
 	If $M$ is $n$-ADC of the form (i) in Lemma \ref{lem:n-ADC-explicitform}, then $M\cong N_{\nu}^{n+2}(c)$ with $(\nu,c)\not=(2,\mathcal{U})$, and the number of these $\mathcal{O}_{F}$-maximal lattices is given by
 	\begin{align*}
 	 3|\mathcal{U}|=3[\mathcal{O}_{F}^{\times}:\mathcal{O}_{F}^{\times 2}]=6(N\mathfrak{p})^{e}
 	\end{align*} 
 	from \eqref{maximal-num} and \cite[63:9]{omeara_quadratic_1963}.
 	If $M$ is $n$-ADC of the form (iii) in Lemma \ref{lem:n-ADC-explicitform}, then the number of such lattices is given by 
 	\begin{align*}
 		4e|\mathcal{U}|=4e[\mathcal{O}_{F}^{\times}:\mathcal{O}_{F}^{\times 2}]=8e(N\mathfrak{p})^{e}. 
 	\end{align*}
 	 Then excluding out the $\mathcal{O}_{F}$-maximal lattices of the form (ii) in Lemma \ref{lem:n-ADC-explicitform}, i.e., $N_{2}^{n+2}(\varepsilon)$ with $\varepsilon\in \mathcal{U}$, gives the number for the non $\mathcal{O}_{F}$-maximal lattices: $
 	4e|\mathcal{U}|-|\mathcal{U}|=8e(N\mathfrak{p})^{e}-2(N\mathfrak{p})^{e}$,
 	as desired.
 \end{proof}
\begin{proof}[Proof of Theorem \ref{thm:dyadicACDodd-2}]
    This follows from Definition \ref{defn:Wvr}, Remark \ref{re:Wvr} and Lemma \ref{lem:n-ADC-explicitform}.
\end{proof} 

 \section{Proof of Theorems \ref{thm:locallyn-ADCn+1}, \ref{thm:globallyn-ADC-n+1}, \ref{thm:locallyn-ADC}, \ref{thm:count-sol} and \ref{thm:Z2-ADCquaternary}} 

We first prove Theorems \ref{thm:locallyn-ADCn+1}, \ref{thm:locallyn-ADC} and \ref{thm:count-sol}.
\begin{proof}[Proof of Theorem \ref{thm:locallyn-ADCn+1}]
	(i) Combine Proposition \ref{prop:n-ADC-n} and Theorems \ref{thm:nondyadic-n-ADCn+1n+2}, \ref{thm:dyadicACDeven-n+1} and \ref{thm:dyadicACDodd-n+1}. (ii) This follows from (i) and \cite[\S 82K]{omeara_quadratic_1963}.
\end{proof}
\begin{proof}[Proof of Theorem \ref{thm:locallyn-ADC}]
	Combine Theorems \ref{thm:nondyadic-n-ADCn+1n+2}, \ref{thm:dyadicACDeven} and \ref{thm:dyadicACDodd-2} and Remark \ref{re:dyadicACDodd-2}.
\end{proof}
\begin{proof}[Proof of Theorem \ref{thm:count-sol}]
Let $M$ be an integral $\mathcal{O}_{F}$-lattice of rank $m$ over a local field $F$.
	
	If $m\in \{n,n+1\}$, or  $m=n+2$ and $F$ is non-dyadic, then, by Theorems \ref{thm:locallyn-ADCn+1}(i) and \ref{thm:locallyn-ADC}(i), $M$ is $n$-ADC if and only if it is $\mathcal{O}_{F}$-maximal. Hence, by \eqref{maximal-num}, $B(m,n)=|\mathcal{M}_{m}|=8(N\mathfrak{p})^{e}$ or $8(N\mathfrak{p})^{e}-1$, according as $m\ge 3$ or $m=2$, as required. 
	
	Assume that $m=n+2$ and $F$ is dyadic. If $n$ is odd, then we are done by Corollary \ref{cor:counting-odd}. If $n$ is even, then, by Theorem \ref{thm:locallyn-ADC}(ii), $M$ is $n$-ADC if and only if $M$ is $\mathcal{O}_{F}$-maximal or $n=2$ and it is not $\mathcal{O}_{F}$-maximal. Consequently,  $B(m,n)=8(N\mathfrak{p})^{e}$ or $8(N\mathfrak{p})^{e}+1$. 
\end{proof}
 
\medskip
 
 In the rest of the paper, we always assume that $F$ is an algebraic number field and $M$ is an $\mathcal{O}_{F}$-lattice. To show Theorem \ref{thm:globallyn-ADC-n+1}, we need some results on the class number of $M$.
\begin{lem}\label{lem:n-ADCclassnumberone}
 Suppose that $M$ has class number one.
	 \begin{enumerate}[itemindent=-0.5em,label=\rm (\roman*)]	
	\item If $M$ is locally $n$-ADC, then it is globally $n$-ADC.
	
	\item If $M$ is $\mathcal{O}_{F}$-maximal, then it is globally $n$-ADC.
	\end{enumerate}
\end{lem}
\begin{proof}
	Since the class number of $ M $ is one, $ M $ is $ n $-regular.
	
	(i) If $M$ is locally $n$-ADC, then it is globally $ n $-ADC by Theorem \ref{thm:globallADC}.
	
	(ii) If $M$ is $\mathcal{O}_{F}$-maximal, then for each $ \mathfrak{p}\in \Omega_{F}\backslash \infty_{F} $, $ M_{\mathfrak{p}} $ is $ \mathcal{O}_{F_{\mathfrak{p}}} $-maximal by \cite[\S 82K]{omeara_quadratic_1963} and so it is $ n $-ADC by Lemma \ref{lem:maximalrep}. Hence $M$ is locally $n$-ADC, so it is globally $n$-ADC by (i).
\end{proof}

Based on Xu's work \cite[\S 1]{xu_springer_1999}, we extend \cite[Theorem 5.2 and Corollary 5.3]{meyer_determination_2014} to the indefinite case. (Also see \cite[\S 4]{hsia_indefinite_1998}.) 
 \begin{thm}\label{thm:meyer-hsx}
 	Suppose $\rank\, M=n+1\ge 3$. Then there exists an $\mathcal{O}_{F}$-lattice $N$ of rank $n$ such that
 	\begin{enumerate}[itemindent=-0.5em,label=\rm (\roman*)]	
 		\item $N\rep M$;
 		
 		\item if $N\rep M^{\prime}$ for some lattice $M^{\prime}$ in  $\gen\, (M)$, then $M^{\prime}\cong M$.
 	\end{enumerate}
 \end{thm}
 
 \begin{proof}
 	  This is clear from \cite[Theorem 5.2]{meyer_determination_2014} when $M$ is definite. Assume that $M$ is indefinite. Let $V=FM$ and take $H=O_{A}(M)O(V)O_{A}^{\prime}(V)$ in \cite[Theorem 1.5']{xu_springer_1999}. Then there exists an $\mathcal{O}_{F}$-lattice $N\subseteq M$ with rank $n$ such that
 	\begin{align*}
 		X_{M/N}O(V)O_{A}^{\prime}(V)=O_{A}(M)O(V)O_{A}^{\prime}(V).
 	\end{align*}
 	From the one-to-one correspondence in \cite[p.\hskip 0.1cm 181]{xu_springer_1999}, there is only one spinor genus in $\gen(M)$ representing $N$. Since $M$ is indefinite, by \cite[104:5 Theorem]{omeara_quadratic_1963}, there is exactly one class in $\gen(M)$ representing $N$.
 \end{proof}
 \begin{cor}\label{cor:n-regular-clsnumberone}
 	Suppose $\rank\, M=n+1\ge 3$. If $M$ is $n$-regular, then $M$ has class number one.
 \end{cor}
 \begin{proof}
 	Let $M^{\prime}$ be a lattice in $\gen(M)$. Then there exists some lattice $N$ of rank $n$ such that $N\rep M^{\prime}$ and if $N\rep M$ for some lattice $M$ in $\gen(M^{\prime})$, then $M\cong M^{\prime}$.
 	
 	 Since $N\rep M^{\prime}$, we see that $N_{\mathfrak{p}}\rep M_{\mathfrak{p}}^{\prime}\cong M_{\mathfrak{p}}$ for all $\mathfrak{p}\in \Omega_{F}$. Since $M$ is $n$-regular, it follows that $N\rep M$. So $M\cong M^{\prime}$ and thus the class number of $M$ is one.
 \end{proof}
\begin{proof}[Proof of Theorem \ref{thm:globallyn-ADC-n+1}]
	Sufficiency is clear from Lemma \ref{lem:n-ADCclassnumberone}(ii). To show necessity, suppose that $ M $ is $ n $-ADC of rank $ n+1 $. Then, by Theorem \ref{thm:globallADC}, it is locally $ n $-ADC. So, by Theorem \ref{thm:locallyn-ADCn+1}(ii), $M$ is $ \mathcal{O}_{F} $-maximal. Again by Theorem \ref{thm:globallADC}, $ M $ is $ n $-regular. Hence, by Corollary \ref{cor:n-regular-clsnumberone}, the class number of $ M $ is one.  
\end{proof}

\medskip

Now, we consider the case $F=\mathbb{Q}$ and $n=2$. Let $p$ be a prime number. For $\gamma\in \mathbb{Q}$, denote by $N^{(\gamma)}$ the $\mathbb{Z}$-lattice (resp. $\mathbb{Z}_{p}$-lattice) $N$ scaled by $\gamma$ (cf. \cite[\S 82J]{omeara_quadratic_1963}).  Assume that $M$ is a positive definite quaternary $\mathbb{Z}$-lattice. Following \cite{oh_even2regular_2008}, a lattice is called \textit{primitive} if $\mathfrak{s}(M)=\mathbb{Z}$. Also note that if $p=2$, then $\mathbb{H}$ and $\mathbb{A}$ from \cite{oh_even2regular_2008}  are $A(0,0)$ and $A(2,2)$, so they coincide with our $\mathbf{H}^{(2)}$ and $\mathbf{A}^{(2)} $. (We have $\mathbf{H}=2^{-1}A(0,0)$ and $\mathbf{A}=2^{-1}A(2,2)$.) If $p>2$, they are the same as our $\mathbf{H}$ and $\mathbf{A}$. But $2$ is a unit in $\mathbb{Q}_{p}$, so $\langle 1,-1\rangle\cong \langle2,-2\rangle$ and $\langle 1,-\Delta\rangle\cong \langle 2,-2\Delta\rangle $, i.e., $\mathbf{H}\cong \mathbf{H}^{(2)}$ and $\mathbf{A}\cong \mathbf{A}^{(2)}$. So again $\mathbb{H}$ and $\mathbb{A}$ from \cite{oh_even2regular_2008} coincide with $\mathbf{H}^{(2)}$ and $\mathbf{A}^{(2)}$. Then as defined in \cite{oh_even2regular_2008}, we call $ \mathcal{L} $ \textit{stable at $p$} if $\mathfrak{n}(\mathcal{L}_{p})=2\mathbb{Z}_{p}$ and $ \mathbf{H}^{(2)}\rep \mathcal{L}_{p}$ or $ \mathcal{L}_{p}\cong \mathbf{A}^{(2)}\perp \mathbf{A}^{(2p)} $. Moreover, we call $\mathcal{L}$ \textit{stable} if it is stable at every prime $ p  $.
\begin{lem}\label{lem:locally2-ADCimpliesstable}
	If $ M $ is $ 2 $-ADC, then $M^{(2)}$ is $2$-regular and stable.
\end{lem}
\begin{proof}
	If $M$ is $2$-ADC, then $\mathfrak{n}(M_{p})=\mathbb{Z}_{p}$, so $\mathfrak{n}(M_{p}^{(2)})=2\mathbb{Z}_{p}$. By Theorem \ref{thm:globallADC}, we see that $M$ is $2$-regular and locally $2$-ADC. Clearly, $M^{(2)}$ is $2$-regular because $2$-regularity is invariant under scaling. For any prime $p$, since $M_{p}$ is $2$-ADC, by Theorem \ref{thm:dyadicACDeven} and Proposition \ref{prop:maximal-quaternary}, $\mathbf{H}\rep M_{p}$ or $M_{p}\cong \mathbf{A}\perp \mathbf{A}^{(p)}$. This implies that   $\mathbf{H}^{(2)}\rep M_{p}^{(2)}$ or $M_{p}^{(2)}\cong \mathbf{A}^{(2)}\perp \mathbf{A}^{(2p)}$, so $M_{p}^{(2)}$ is $p$-stable. Thus $M^{(2)}$ is stable. 
	 \end{proof}
	 
By Theorem \ref{thm:globallADC}	and Lemma \ref{lem:locally2-ADCimpliesstable}, we have the following corollary.
\begin{cor}\label{cor:2-ADC-Z}
	 $M$ is $2$-ADC if and only if it is locally $2$-ADC and isometric to $\mathcal{L}^{(1/2)}$ for some stable $2$-regular lattice $\mathcal{L}$.
\end{cor}
 
As in \cite{oh_even2regular_2008}, we put $ \mathcal{L}\cong [a,b,c,d,f_{1},f_{2},f_{3},f_{4},f_{5},f_{6}] $ if
\begin{align*}
	\mathcal{L}\cong \begin{pmatrix}
		a  &  f_{1}  &f_{2}  &f_{4}  \\
		f_{1}  &b    &f_{3}   &f_{5} \\
		f_{2}  &f_{3}  &c     &f_{6} \\
		f_{4}  &f_{5}  &f_{6}  &d
	\end{pmatrix}.
\end{align*}
Table \ref{table:stable2regularquaternary} adopted from \cite[\S 4]{oh_even2regular_2008} enumerates all primitive stable $2$-regular quaternary $\mathbb{Z}$-lattices, where we list all the primes for which $(\mathcal{L}_{i}^{(1/2)})_{p}$ is not $2$-ADC in the last column. 
Then, we relabel these lattices $ \mathcal{L}_{j}^{(1/2)} $, as shown in the first two columns of Table \ref{table:2ADCquaternary}. The third and fourth columns provide the local structures of each $L_{i}$ for the primes $p$, where $(L_{i})_{p}$ is not unimodular.
\begin{proof}[Proof of Theorem \ref{thm:Z2-ADCquaternary}]
	  As mentioned before, Table \ref{table:stable2regularquaternary} lists all stable $ 2 $-regular quaternary $ \mathbb{Z} $-lattices. Hence, by Corollary \ref{cor:2-ADC-Z}, the lattices $ \mathcal{L}_{i}^{(1/2)}  $ ($ i=1,\ldots,48 $) in Table \ref{table:stable2regularquaternary} are all possible candidates, and it suffices to determine which of them are locally $2$-ADC. For each prime $p>2$ with $p\nmid d\mathcal{L}$, $L_{p}$ is unimodular and so $\mathbb{Z}_{p}$-maximal. Thus it is $2$-ADC. Therefore, one only needs to check if $L_{p}$ is $2$-ADC for $p=2$ and the primes $p>2$ with $p\mid d\mathcal{L}$, and we complete the verification by hand.
\end{proof}
 
\begin{table}[h]\tiny
	\begin{center}
		\caption{\small Quaternary positive definite stable $ 2 $-regular integral $ \mathbb{Z} $-lattices $ \mathcal{L}_{i}$}	\label{table:stable2regularquaternary}	
					\renewcommand\arraystretch{1.2}
		\begin{tabular}{c|c|c|c}
			\toprule[1pt]
		$\mathcal{L}$ &$[a,b,c,d,f_{1},f_{2},f_{3},f_{4},f_{5},f_{6}]$&$ d\mathcal{L} $  &The primes $p$ where $\mathcal{L}_{p}^{(1/2)}$ is not $2$-ADC    \\
	
			\hline
		 $\mathcal{L}_{1}$ &$ [2,2,2,2,0,0,0,1,1,1] $  &$  2^{2} $  &None  \\
			\hline
		 $ \mathcal{L}_{2} $ &$  [2,2,2,2,1,0,0,1,0,1] $  &$ 5 $         &None    \\
			\hline
		 $ \mathcal{L}_{3} $ &$  [2,2,2,2,0,0,0,1,1,0] $  &$  2^{3} $         &  None  \\
			\hline
		 $ \mathcal{L}_{4} $ &$  [2,2,2,2,1,0,0,0,0,1] $  &$ 3^{2} $         & None     \\
			\hline 
		 $ \mathcal{L}_{5} $ &$  [2,2,2,4,1,1,0,1,0,0] $  &$ 2^{2} \cdot 3 $         & None   \\
			\hline
		 $ \mathcal{L}_{6} $ &$  [2,2,2,2,1,0,0,0,0,0] $  &$ 2^{2}\cdot 3 $         & None   \\
			\hline
		 $ \mathcal{L}_{7} $ &$  [2,2,2,4,1,1,0,0,1,0] $  &$13 $         & None     \\
			\hline
		 $ \mathcal{L}_{8} $ &$  [2,2,2,4,1,0,0,1,0,1] $  &$17 $         & None    \\
			\hline
		 $ \mathcal{L}_{9} $ &$  [2,2,2,4,0,0,0,1,1,1] $  &$ 2^{2}\cdot 5 $         & None   \\
			\hline
		 $ \mathcal{L}_{10} $ &$  [2,2,2,4,1,0,0,1,0,0] $  &$ 2^{2}\cdot 5 $         & None    \\
			\hline
		 $ \mathcal{L}_{11} $ &$  [2,2,2,4,1,0,0,0,0,1] $  &$ 3\cdot 7 $         &None  \\
			\hline
		 $ \mathcal{L}_{12} $ &$  [2,2,2,6,1,1,0,0,1,0] $  &$ 3\cdot 7 $         &None     \\
			\hline 
		 $ \mathcal{L}_{13} $ &$  [2,2,2,4,0,0,0,1,1,0] $  &$ 2^{3} \cdot 3$         &None   \\
			\hline
		 $ \mathcal{L}_{14} $ &$  [2,2,4,4,1,1,0,1,1,2] $  &$ 5^{2} $         & None    \\
			\hline
		 $ \mathcal{L}_{15} $ &$  [2,2,4,4,1,1,0,0,1,1] $  &$ 2^{2}\cdot 7 $         & None    \\
			\hline
		  $ \mathcal{L}_{16} $ &$ [2,2,2,6,1,0,0,1,0,0] $  &$  2^{5} $         & $2$       \\
			\hline
	 	$ \mathcal{L}_{17} $ &$[2,2,4,4,0,0,0,1,1,2]  $  &$ 2^{5}  $         & $  2$        \\
			\hline
		 $ \mathcal{L}_{18} $ &$[2,2,4,4,1,1,0,1,0,0]  $  &$  2^{5} $         & $2$       \\
			\hline
		 $ \mathcal{L}_{19} $ &$  [2,2,4,4,0,1,1,1,0,2] $  &$ 3\cdot 11 $         & None     \\
			\hline
		$ \mathcal{L}_{20} $ &$[2,2,2,10,1,1,0,1,0,0]  $  &$  2^{2}\cdot3^{2} $         & 3     \\
			\hline
		$ \mathcal{L}_{21} $ &$[2,2,2,6,0,0,0,1,1,1]  $  &$  2^{2}\cdot3^{2} $         & $3$     \\
			\hline
		$ \mathcal{L}_{22} $ &$[2,2,2,6,1,0,0,0,0,0]  $  &$  2^{2}\cdot3^{2} $         & 2     \\
			\hline
		$ \mathcal{L}_{23} $ &$ [2,2,4,4,1,0,0,0,0,2] $  &$  2^{2}\cdot3^{2} $         & $3$      \\
			\hline
		$ \mathcal{L}_{24} $ &$[2,2,4,4,0,1,1,1,1,1]  $  &$  2^{2}\cdot3^{2} $         & $2$     \\	
		   	\hline
		 $ \mathcal{L}_{25} $ &$  [2,2,4,4,1,0,0,0,0,1] $  &$ 3^{2}\cdot 5 $         & None       \\
		 	\hline
		$ \mathcal{L}_{26} $ &$ [2,2,4,4,0,1,0,0,1,1] $  &$  3^{2}\cdot 5 $         & $3$       \\	
			\hline
	    $ \mathcal{L}_{27} $ &$ [2,2,4,6,1,1,0,0,1,1] $  &$ 2^{4}\cdot 3 $         & $2$   \\	
	    \hline
	     $ \mathcal{L}_{28} $ &$[2,2,4,4,0,1,1,0,0,0]  $  &$ 2^{4}\cdot 3  $         & $2$      \\	
	     	\hline
	     $ \mathcal{L}_{29} $ &$[2,4,4,4,0,0,0,1,2,2]  $  &$ 2^{4}\cdot 3  $         & $2$     \\	
			\hline
		 $ \mathcal{L}_{30} $ &$  [2,2,4,4,0,1,0,0,1,0] $  &$ 7^{2} $         &  None      \\
			\hline
		 $ \mathcal{L}_{31} $ &$  [2,4,4,4,1,0,2,0,1,2] $  &$ 2^{2}\cdot 3\cdot 5 $         & None   \\
			\hline
		 $ \mathcal{L}_{32} $ &$  [2,2,4,6,0,1,0,1,1,0] $  &$ 3\cdot 23 $         & None  \\
		 	\hline
		 $ \mathcal{L}_{33} $ &$[2,4,4,4,1,1,0,1,0,0 ]  $  &$  2^{4}\cdot 5 $         & $2$    \\	
		 	\hline
		$ \mathcal{L}_{34} $ &$[2,2,4,8,0,1,0,0,0,2]  $  &$  2^{5}\cdot 3 $         & $ 2$    \\	
			\hline
		 $ \mathcal{L}_{35} $ &$[2,4,4,4,0,0,0,1,1,1]  $  &$   2^{5}\cdot 3 $         & $ 2$     \\	
		 	\hline
		 $ \mathcal{L}_{36} $ &$[2,2,6,6,0,1,1,1,1,1]  $  &$  2^{2}\cdot 5^{2}  $         & $5$   \\					 
			\hline
		$ \mathcal{L}_{37} $ &$[2,4,4,6,1,0,2,0,1,2 ]  $  &$ 2^{2}\cdot 5^{2}  $         & $2$    \\	
			\hline
		$ \mathcal{L}_{38} $ &$[2,4,4,6,0,0,2,-1,1,-1 ]  $  &$ 2^{2}\cdot 3^{3}  $         & $3$    \\	
			\hline
		$ \mathcal{L}_{39} $ &$[2,4,4,6,1,0,2,0,1,0]  $  &$ 2^{4}\cdot 7 $         & $ 2$    \\	
			\hline
		$ \mathcal{L}_{40} $ &$[2,2,6,8,0,1,1,1,0,3]  $  &$ 5^{3}  $         & $5$     \\	
			\hline
		$ \mathcal{L}_{41} $ &$[2,4,4,8,1,0,2,0,2,0]  $  &$ 2^{7}  $         & $2 $   \\	
			\hline
		$ \mathcal{L}_{42} $ &$[2,4,4,6,1,1,0,0,1,1]  $  &$ 2^{7}  $         & $2 $    \\	
	 		\hline
    	$ \mathcal{L}_{43} $ &$[2,4,4,8,1,1,0,1,2,2]  $  &$ 2^{4}\cdot 3^{2} $         & $2$    \\	
			\hline
		$ \mathcal{L}_{44} $ &$[2,4,4,8,1,0,1,1,1,2] $  &$ 13^{2} $         &  None      \\
			\hline
		$ \mathcal{L}_{45} $ &$[2,4,6,6,0,1,1,1,2,1]  $  &$ 3^{3}\cdot 7  $         & $ 3$     \\	
	    	\hline 
		$ \mathcal{L}_{46} $ &$[2,4,6,6,1,0,1,0,-1,2]  $  &$  2^{6}\cdot 3 $         & $2$       \\	
		\hline  
		$ \mathcal{L}_{47} $ &$[2,4,6,10,0,1,2,0,2,1]  $  &$  2^{2}\cdot 3^{4} $         & $ 3$      \\	
		\hline 
		$ \mathcal{L}_{48} $ &$[2,4,6,12,0,1,0,0,2,0]  $  &$  2^{2}\cdot 11^{2} $         & $ 11$     \\	
			\bottomrule[1pt]
		\end{tabular}	 
	\end{center}
\end{table}

\begin{table}[!ht]
	\begin{center}
		\caption{Quaternary positive definite $ 2 $-ADC integral $ \mathbb{Z} $-lattices $ L_{i}  $}	\label{table:2ADCquaternary}	
		\renewcommand\arraystretch{1.3}
		\begin{tabular}{c|c|c|c|c}
			\toprule[1.2pt]
			\multirow{2}{*}{\text{$L $}}&	\multirow{2}{*}{\text{$\mathcal{L}^{(1/2)} $}}  & \multicolumn{2}{c|}{$L_{p}$ is not unimodular}  &\multirow{2}{*}{\text{$ \mathbb{Z} $-maximal}}  \\
			\cline{3-4}
			&  & $p=2$ & $p>2$   &     \\
			\hline
			$ L_{1}$	&$\mathcal{L}_{1}^{(1/2)} $  &$N_{2}^{4}(1) $ &  & \text{True}   \\
			\hline
			$ L_{2} $	&$ \mathcal{L}_{2}^{(1/2)}  $  &$ N_{1}^{4}(5)$ &$ N_{2}^{4}(5)$, $p=5$        & \text{True}   \\
			\hline
			$ L_{3} $&$ \mathcal{L}_{3}^{(1/2)}  $  &$ N_{2}^{4}(2) $   &      & \text{True}   \\
			\hline
			$ L_{4} $&$ \mathcal{L}_{4}^{(1/2)}  $  &$N_{1}^{4}(1)$ &$ N_{2}^{4}(1) $, $p=3$        & \text{True}   \\
			\hline 
			$ L_{5} $&$ \mathcal{L}_{5}^{(1/2)}  $  &$N_{2}^{4}(3)$ &$ N_{1}^{4}(3) $, $p=3$     & \text{True}   \\
			\hline
			$ L_{6} $&$ \mathcal{L}_{6}^{(1/2)}  $  &$ N_{1}^{4}(3)$ &$ N_{2}^{4}(3)$, $p=3$         & \text{True}   \\
			\hline
			$ L_{7} $&$ \mathcal{L}_{7}^{(1/2)}  $  &$N_{1}^{4}(5)$ &$ N_{2}^{4}(13)$, $p=13$       & \text{True}   \\
			\hline
			$ L_{8} $&$ \mathcal{L}_{8}^{(1/2)}  $ &$N_{1}^{4}(1)$ &$ N_{2}^{4}(17) $, $p=17$     & \text{True}   \\
			\hline
			$ L_{9} $&$ \mathcal{L}_{9}^{(1/2)}  $   &$N_{2}^{4}(5)$ &$N_{1}^{4}(5) $, $p=5$      & \text{True}   \\
			\hline
			$ L_{10} $&$ \mathcal{L}_{10}^{(1/2)}  $  &$\mathbf{H}\perp \langle 1,-5\rangle$&$N_{2}^{4}(5)$, $p=5$        & \text{False}   \\
			\hline
			\multirow{2}{*}{$ L_{11} $}&\multirow{2}{*}{$ \mathcal{L}_{11}^{(1/2)}  $} &\multirow{2}{*}{$N_{1}^{4}(5)$} &$N_{2}^{4}(3)  $, $p=3$         & \multirow{2}{*}{\text{True}}   \\
			\cline{4-4}
		            &  & & $N_{1}^{4}(7\Delta_{7}) $, $p=7$ &     \\
			\hline
			
			\multirow{2}{*}{$ L_{12} $}&\multirow{2}{*}{$ \mathcal{L}_{12}^{(1/2)}  $} &\multirow{2}{*}{$N_{1}^{4}(5)$} &$N_{1}^{4}(3) $, $p=3$         & \multirow{2}{*}{\text{True}}   \\
			\cline{4-4}
			&  & & $N_{2}^{4}(7\Delta_{7}) $, $p=7$ &     \\
			\hline 
			$ L_{13} $&$ \mathcal{L}_{13}^{(1/2)}  $  &$N_{1}^{4}(6)$ &$N_{2}^{4}(3\Delta_{3}) $, $p=3$       & \text{True}   \\
			\hline
			$ L_{14} $&$ \mathcal{L}_{14}^{(1/2)}  $   &$N_{1}^{4}(1)$ &$ N_{2}^{4}(1) $, $p=5$         & \text{True}   \\
			\hline
			$ L_{15} $&$ \mathcal{L}_{15}^{(1/2)}  $    &$N_{2}^{4}(7)$ &$N_{1}^{4}(7)$, $p=7$        & \text{True}   \\
			\hline
			\multirow{2}{*}{$ L_{16} $}&\multirow{2}{*}{$ \mathcal{L}_{19}^{(1/2)}  $} &\multirow{2}{*}{$N_{1}^{4}(1)$} &$N_{2}^{4}(3\Delta_{3}) $, $p=3$         & \multirow{2}{*}{\text{True}}   \\
			\cline{4-4}
			&  & & $N_{1}^{4}(11)  $, $p=11$ &     \\
			\hline
			\multirow{2}{*}{$ L_{17} $}&\multirow{2}{*}{$ \mathcal{L}_{25}^{(1/2)} $} &\multirow{2}{*}{$N_{1}^{4}(5)$} &$N_{2}^{4}(5) $, $p=3$         & \multirow{2}{*}{\text{True}}   \\
			\cline{4-4}
			&  & & $N_{1}^{4}(5)  $, $p=5$ &     \\
			\hline
			$ L_{18} $&$ \mathcal{L}_{30}^{(1/2)}  $   &$N_{1}^{4}(1)$  &$N_{2}^{4}(1) $, $p=7$      & \text{True}   \\
				\hline
			\multirow{2}{*}{$ L_{19} $}&\multirow{2}{*}{$ \mathcal{L}_{31}^{(1/2)}  $} &\multirow{2}{*}{$N_{2}^{4}(7)$} &$N_{2}^{4}(3\Delta_{3})  $, $p=3$         & \multirow{2}{*}{\text{True}}   \\
			\cline{4-4}
			&  & & $N_{2}^{4}(5\Delta_{5})$, $p=5$ &     \\			 
			\hline
	     	\multirow{2}{*}{$ L_{20} $}&\multirow{2}{*}{$ \mathcal{L}_{32}^{(1/2)}  $} &\multirow{2}{*}{$N_{1}^{4}(5)$} &$N_{2}^{4}(3\Delta_{3}) $, $p=3$         & \multirow{2}{*}{\text{True}}   \\
	     	\cline{4-4}
	     	&  & & $N_{1}^{4}(23)$, $p=23$ &     \\
		    \hline	
			$ L_{21} $&$ \mathcal{L}_{44}^{(1/2)}  $  &$N_{1}^{4}(1) $ &$N_{2}^{4}(1) $, $p=13$         & \text{True}   \\ 
			\bottomrule[1.2pt]
		\end{tabular}	 
	\end{center}
\end{table}

\section*{Acknowledgments} 
  I am grateful to the referee for providing many detailed corrections and suggestions, which significantly improved the exposition of this paper. I would also like to thank Prof. Yong Hu and Prof. Fei Xu for helpful discussions, to thank Prof. Pete L. Clark for enlightening comments, and to thank Prof. Andrew G. Earnest for detailed and valuable suggestions. This work was supported by a grant from the National Natural Science Foundation of China (Project No. 12301013).

			\end{document}